\documentclass[12pt,a4paper]{amsart}

\usepackage{amsmath,amsthm,color}
\usepackage{enumitem}
\usepackage{graphicx}
\usepackage{amssymb}

\usepackage{pdfsync}

\usepackage{enumitem}

\newcommand{\C}{{\mathbb C}}       
\newcommand{\R}{{\mathbb R}}       
\newcommand{\MD}{{\mathcal {MD}}}

\newcommand{\wt}[1]{{\widetilde{#1}}}
\newcommand{\wh}[1]{{\widehat{#1}}}

\renewcommand{\le}{\leqslant}
\renewcommand{\ge}{\geqslant}
\newcommand{\complex}{\mathbb{C}}
\newcommand{\diam}{\text{\,\rm diam}}
\newcommand{\dist}{\text{\,\rm dist}}
\renewcommand{\Re}{\text{\sf Re\,}}

\newcommand{\spt}{\,\mathrm{spt}\,}

\newcommand{\LD}{{\mathsf{LD}}}
\newcommand{\HD}{{\mathsf{HD}}}
\newcommand{\UB}{{\mathsf{UB}}}

\newcommand{\nex}{{\mathsf{Next}}}

\newtheorem{theorem}{Theorem}
\newtheorem*{theorem*}{Theorem}
\newtheorem*{lemma*}{Lemma}
\newtheorem*{mainlemma*}{Main Lemma}
\newtheorem{lemma}{Lemma}
\newtheorem{corollary}{Corollary}
\theoremstyle{definition}

\newcommand{\ps}[1]{\left( #1 \right)}

\theoremstyle{remark}
\newtheorem{remark}{\bf Remark}

\numberwithin{equation}{section}

\usepackage{geometry}
\newtheoremstyle{citedth}%
  {5pt}
  {5pt}
  {\itshape}
  {}
  {\bfseries}
  {.}
  {.3em}
  {\thmname{#1} \thmnumber{#2} \thmnote{\normalfont#3}}

\theoremstyle{citedth}

\begin{document}
\title[Singular integral operators which control the Cauchy transform]{A family of singular integral operators\\ which control the Cauchy transform}
\author{Petr Chunaev}
\address{Petr Chunaev, Departament de Matem\`{a}tiques, Universitat Aut\`{o}noma de Barce\-lo\-na,  08193 Bellaterra (Barcelona), Catalonia} \email{chunayev@mail.ru,
chunaev@mat.uab.cat}
\author{Joan Mateu}
\address{Joan Mateu, Departament de Matem\`{a}tiques, Universitat Aut\`{o}noma de Barcelona, 08193 Bellaterra (Barcelona), Catalonia} \email{mateu@mat.uab.cat}
\author{Xavier Tolsa}
\address{Xavier Tolsa, ICREA, Passeig Llu\'{\i}s Companys 23 08010 Barcelona, Catalonia, and 
Departament de Matem\`atiques and BGSMath,
Universitat Aut\`onoma de Barce\-lo\-na,
08193 Bellaterra (Barcelona), Catalonia
} \email{xtolsa@mat.uab.cat}

\begin{abstract}
We study the behaviour of singular integral operators $T_{k_t}$ of convolution type on $\mathbb{C}$ associated with the parametric kernels
$$
k_t(z):=\frac{(\Re z)^{3}}{|z|^{4}}+t\cdot \frac{\Re z}{|z|^{2}}, \quad t\in \mathbb{R},\qquad k_\infty(z):=\frac{\Re z}{|z|^{2}}\equiv \Re \frac{1}{z},\quad z\in \mathbb{C}\setminus\{0\}.
$$
It is shown that for any positive locally finite Borel measure with linear growth the corresponding $L^2$-norm of $T_{k_0}$ controls
 the $L^2$-norm of $T_{k_\infty}$ and thus of the Cauchy transform.
As a corollary, we prove that the $L^2(\mathcal{H}^1\lfloor
E)$-boundedness of $T_{k_t}$ with a fixed $t\in (-t_0,0)$, where $t_0>0$ is an absolute constant, implies that $E$ is rectifiable. This is so in spite of the fact that the usual curvature method fails to be 
applicable in this case.  
 Moreover, as a corollary of our techniques, we provide an alternative and simpler proof of the bi-Lipschitz invariance of the $L^2$-boundedness of the Cauchy transform, which is the key ingredient for the bilipschitz invariance of analytic capacity.
\end{abstract}

\maketitle

\tableofcontents

\section{Introduction and Theorems}
\label{section1}
 In this paper we study the behaviour of
\textit{singular integral operators} (\textit{SIOs}) in the complex plane associated with the
kernels
\begin{equation}
\label{k_t}
k_t(z):=\frac{(\Re z)^{3}}{|z|^{4}}+t\cdot \frac{\Re z}{|z|^{2}}, \quad t\in \mathbb{R},\qquad k_\infty(z):=\frac{\Re z}{|z|^{2}}\equiv \Re \frac{1}{z},
\end{equation}
where $z\in \mathbb{C}\setminus\{0\}$. This topic was previously discussed in \cite{Chunaev2016,ChMT2016}.
{Among other things, we show that there exists $t_0>0$ such that, given $t\in(-t_0,0)$,
 the $L^2$-boundedness of $T_{k_t}$  implies the $L^2$-boundedness of a wide class of SIOs. We also establish the equivalence between the $L^2(\mu)$-boundedness of $T_{k_t}$  and the uniform rectifiability of $\mu$ in the case when $\mu$ is Ahlfors-David regular. Moreover, as a corollary of our techniques, we also provide an alternative and simpler proof of the bi-Lipschitz invariance of the $L^2$-boundedness of the Cauchy transform, which in turn implies the bilipschitz invariance of analytic capacity modulo constant factors.
  Note that analogous problem in higher dimensions for the Riesz transform is still an open challenging problem.

We start with necessary notation and background facts. Note that we
work only in $\mathbb{C}$ and therefore usually skip dimension
markers in definitions.

Let $E\subset \mathbb{C}$ be a Borel set and $B(z,r)$ be an open disc with center $z\in \mathbb{C}$ and radius~${r>0}$.
 We denote by $\mathcal{H}^1(E)$ the ($1$-dimensional) \textit{Hausdorff measure} of $E$, i.e. \textit{length}, and call $E$ a \textit{$1$-set} if $0<\mathcal{H}^1(E)<\infty$. A set $E$ is called \textit{rectifiable} if it is contained in a countable union of Lipschitz graphs, up to a set of $\mathcal{H}^1$-measure zero. A set $E$ is called \textit{purely unrectifiable} if it intersects any Lipschitz graph in a set of $\mathcal{H}^1$-measure zero.
 By a \textit{measure} often denoted by $\mu$ we mean a \textit{positive locally finite Borel measure on~$\mathbb{C}$}.

Given a measure $\mu$, a
kernel $k_t$ of the form (\ref{k_t}) and an ${f\in L^1(\mu)}$,
 we define the following truncated  SIO 
\begin{equation}
\label{integral_operator} T_{k_t,\varepsilon}f(z):=\int_{E \setminus
B(z,\varepsilon)}
f(\zeta)k_t(z-\zeta)d\mu(\zeta),\qquad \text{where }E=\spt \mu \quad \text{and}\quad \varepsilon>0.
\end{equation}
We do not define the SIO $T_{k_t}$ explicitly because several delicate
problems such as the existence of the principal value might arise. Nevertheless, we  say that $T_{k_t}$ is
\textit{$L^2(\mu)$-bounded} and write $\|T_{k_t}\|_{L^2(\mu)}<\infty$ if
the operators $T_{k_t,\varepsilon}$ are $L^2(\mu)$-bounded uniformly
on $\varepsilon$.

\textit{How to relate the $L^2(\mu)$-boundedness of a certain SIO to
the geometric properties of the support of $\mu$} is an old problem
in Harmonic Analysis. It stems from Calder\'{o}n~\cite{Calderon} and Coifman, McIntosh and Meyer \cite{CMM} who proved that the Cauchy transform is $L^2(\mathcal{H}^1\lfloor
E)$-bounded on  Lipschitz graphs $E$. In \cite{David} David fully characterized
rectifiable \textit{curves} $\Gamma$, for which the Cauchy transform
is $L^2(\mathcal{H}^1\lfloor \Gamma)$-bounded. These results led to
further development of tools for understanding the above-mentioned
problem.

Our purpose is to relate the $L^2(\mathcal{H}^1\lfloor
E)$-boundedness of $T_{k_t}$ associated with the kernel (\ref{k_t}) to
the geometric properties of $E$. Let us mention the known results (we formulate them in a slightly different form than in the original papers).
 In what follows we suppose that $E\subset\mathbb{C}$ is a $1$-set.

The first one is due to David and L\'{e}ger \cite{L} and related to
$k_\infty$, i.e. the real part of the Cauchy kernel, although it was
proved for the Cauchy kernel originally:
 $$
\|T_{k_\infty}\|_{L^2(\mathcal{H}^1\lfloor E)}<\infty \quad
\Rightarrow \quad E \text{ \textit{is rectifiable}}. \eqno{(A)}
$$
This is a very difficult result which generalizes the classical one
of Mattila, Melnikov and Verdera \cite{MMV} for Ahlfors-David
regular sets $E$. As in \cite{MMV}, the proof in \cite{L} uses the
so called \textit{Menger} curvature and the fact that it is
non-negative. Since we use similar tools, all the necessary
definitions will be given below.

A natural question arose consisting in proving analogues of  $(A)$ for SIOs associated with kernels different from
 the Cauchy kernel or its coordinate parts, see
\cite{MMV,CMPT}. Recently Chousionis, Mateu, Prat and Tolsa
\cite{CMPT} gave the first non-trivial example of such SIOs. Namely,
they proved the following implication:
$$
\|T_{k_0}\|_{L^2(\mathcal{H}^1\lfloor E)}<\infty \quad \Rightarrow
\quad E \text{ \textit{is rectifiable}}. \eqno{(B)}
$$

The authors of \cite{CMPT} used a curvature type method. It
allowed them to modify the required parts of the proof from \cite{L}
to obtain their result. Extending this technique, Chunaev
\cite{Chunaev2016} proved that the same is true for a quite large
range of the parameter $t$, additionally to $t=0$ and $t=\infty$:
$$
\|T_{k_t}\|_{L^2(\mathcal{H}^1\lfloor E)}<\infty \text{ \textit{for a fixed }}
t\in (-\infty, -2]\cup (0,+\infty)\quad \Rightarrow \quad E \text{
\textit{is rectifiable}}. \eqno{(C)}
$$

It is also shown in \cite{Chunaev2016} that a direct curvature type
method cannot be applied for $t\in(-2,0)$. Moreover, it is known that for some of these $t$ there exist
counterexamples to the above-mentioned implication due to results of
Huovinen \cite{H} and Jaye and Nazarov \cite{JN}:
$$
t=-1 \text{ \textit{or} } t=-\tfrac{3}{4} \quad \Rightarrow\quad\exists \text{ \textit{purely unrectifiable} } E:
\|T_{k_t}\|_{L^2(\mathcal{H}^1\lfloor E)}<\infty .  \eqno{(D)}
$$

Note that the examples by Huovinen and Jaye and Nazarov  are
 different and essentially use the analytical properties of each of
the kernels. Moreover, the corresponding constructions are quite
complicated and this apparently indicates that
 constructing such examples for some more or less special class
of  kernels is not an easy task. This is an example of
the difficulty of dealing with $T_{k_t}$ for $t\in(-2,0)$. We however
succeeded in \cite{ChMT2016} in proving the following result:
$$
\|T_{k_t}\|_{L^2(\mathcal{H}^1\lfloor E)}<\infty \text{ \textit{for a fixed } }
t\in (-2,-\sqrt{2})\quad \Rightarrow \quad E \text{ \textit{is
rectifiable}}. \eqno{(E)}
$$

This is the first example in the plane when the curvature method cannot be
applied directly (as the corresponding pointwise curvature-like
expressions called \textit{permutations} change sign) but it can
still be proved that $L^2$-boundedness implies rectifiability.

The aim of this paper is to move forward in understanding  the behaviour of $T_{k_t}$ for a fixed $t\in(-2,0)$ when direct curvature methods are not available.
First we prove the following.
\begin{theorem}
\label{paper3-theorem1} 
There exist absolute constants $t_0>0$ and $c>0$ such that for any finite measure $\mu$ with $C_*$-linear
growth it holds that
\begin{equation}
\label{paper3-L^2-norms} \sup_{\varepsilon>0}\|T_{k_\infty,\varepsilon} 1\|_{L^2(\mu)}\le
t_0^{-1}\sup_{\varepsilon>0}\|T_{k_0,\varepsilon}
1\|_{L^2(\mu)}+cC_*\sqrt{\mu(\mathbb{C})}.
\end{equation}
\end{theorem}


Note that, by \cite[Lemma 3]{ChMT2016} under the same assumptions
on $\mu$,
\begin{equation}
\label{L^2-norms-opposite}  \|T_{k_0,\varepsilon} 1\|_{L^2(\mu)}\le 
\sqrt{2}  \|T_{k_\infty,\varepsilon} 1\|_{L^2(\mu)}
+cC_*\sqrt{\mu(\mathbb{C})},
\end{equation}
where $\varepsilon>0$ and $c>0$ is independent of $\varepsilon$.
With respect to the proof of (\ref{L^2-norms-opposite}) in
\cite{ChMT2016}, the proof of (\ref{paper3-L^2-norms}) is more difficult as we will see in this paper.

Denote by $C_\mu$ the Cauchy transform with respect $\mu$. That is,
$$C_\mu f(z) = \int \frac1{z-\xi}\,f(\xi)\,d\mu(\xi).$$ 
From Theorem~\ref{paper3-theorem1} and a perturbation argument, using the same $t_0$, we will
show the next result.

\begin{theorem}
\label{teonou} 
Let $\mu$ be a measure
with linear growth and  $t\in  (-t_0,0)$. If the SIO $T_{k_t}$  is $L^2(\mu)$-bounded, then
the Cauchy transform $C_\mu$ is also $L^2(\mu)$-bounded.
\end{theorem}

See also Corollary \ref{paper3-corollary2}  below for a more
general statement.

As an immediate consequence of Theorem \ref{teonou} and the statement $(A)$, we obtain the following.

\begin{corollary}
\label{paper3-theorem2} 
Let $t\in (-t_0,0)$. 
If $\|T_{k_t}\|_{L^2(\mathcal{H}^1\lfloor E)}<\infty$, then $E$ is rectifiable.
\end{corollary}

This corollary complements the assertions $(A)-(E)$ so that we have
the overall picture as in Figure~\ref{FIG1}. It is clear from $(D)$ that
necessarily $t_0\in (0,3/4)$. What is more, it is very important
here that the pointwise curvature-like expressions (\textit{permutations}), corresponding to $t\in (-t_0,0)$, also
change sign as in $(E)$ so that the curvature method cannot be
applied directly but  $L^2$-boundedness
still implies rectifiability.

\begin{figure}
  \includegraphics[width=15cm]{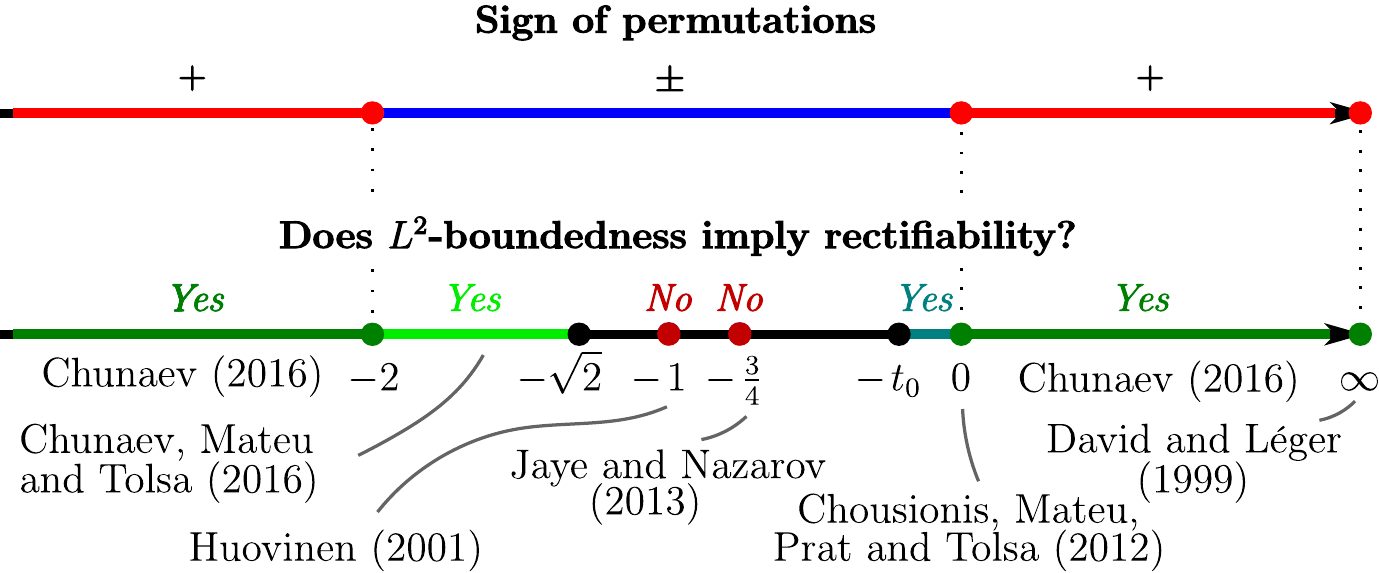}\\
 \caption{The overall picture for SIO associated with the
 kernels~$k_t$}
  \label{FIG1}
\end{figure}

\begin{remark}
By simple analysis one can show that the kernel $k_t$ has
 $$
 \left\{
 \begin{array}{l}
   \text{one zero line if } t\in (-\infty,-1)\cup [0,\infty], \\
   \text{two zero lines if } t=-1, \\
   \text{three zero lines if } t\in (-1,0).
 \end{array}
 \right.
 $$
By a zero line we mean a straight line $L\subset \mathbb{C}$ such
that $k_t(z)=0$ for $z\in L$.

In this sense, it is interesting to compare Corollary~\ref{paper3-theorem2} with the part of $(D)$ deduced from
\cite{JN}. Observing Figure~\ref{FIG1}, one can see that the number of zero lines along is not determinant.
\end{remark}
\begin{remark}
Let $t_1$ and $t_2$ be such that $-\sqrt{2}\le t_1<t_2\le -t_0$. If there exist finite \textit{purely
unrectifiable} (i.e. concentrated on purely
unrectifiable sets) measures $\mu_1$ and $\mu_2$ with linear growth such that  $T_{k_{t_1}}$ is
$L^2(\mu_1)$-bounded and  $T_{k_{t_2}}$ is
$L^2(\mu_2)$-bounded, then $\mu_1$ is different from $\mu_2$.

Indeed, let $\mu$ be a finite purely unrectifiable measure with linear growth such that
$T_{k_{\tilde{t}}}$ is $L^2(\mu)$-bounded for a fixed $\tilde{t}\in
[-\sqrt{2},-t_0]$. By the triangle inequality for any real $t$,
$$
\|T_{k_t}1\|_{L^2(\mu)}=\|(T_{k_0}+(t-\tilde{t})\cdot T_{k_\infty} +\tilde{t} \cdot T_{k_\infty})1\|_{L^2(\mu)}\ge
|t-\tilde{t}|\|T_{k_\infty}1\|_{L^2(\mu)}-\|T_{k_{\tilde{t}}}1\|_{L^2(\mu)}.
$$
Consequently,
$\|T_{k_t}1\|_{L^2(\mu)}=\infty$ for all $t\neq \tilde{t}$ as
$\|T_{k_\infty}1\|_{L^2(\mu)}=\infty$ since $\mu$ is purely unrectifiable.
Thus an example of a purely unrectifiable measure $\mu$ such that
$T_{k_{\tilde{t}}}$ is $L^2(\mu)$-bounded for a fixed
$\tilde{t}\in[-\sqrt{2},-t_0]$ does not work for $t\neq \tilde{t}$.
\end{remark}

\section{Notation and definitions}

\subsection{Constants} We use the letters $c$ and $C$ to denote
constants which may change their values at different occurrences. On
the other hand, constants with subscripts such as $A_0$ or ${\sf c}_1$ do not
change their values throughout the paper. In a majority of cases
constants depend on some parameters which are usually indicated explicitly and will be fixed at the very
end so that the constants become absolute.

If there is a constant $C$ such that $A\le C\,B$, we write
$A\lesssim B$. Furthermore, $A\approx B$ is equivalent to saying
that $A\lesssim B\lesssim A$, possible with different implicit
constants. If the implicit constant in expressions with
``$\lesssim$'' or ``$\approx$'' depends on some positive parameter,
say, $\alpha$, we write  $A\lesssim_{\alpha} B$ or $A\approx_\alpha
B$.

\subsection{Curvature and permutations of measure}
For an odd and  real-valued kernel
$K$, consider the following \textit{permutations}:
\begin{equation}
\label{paper1-permutation_Main}
\begin{split}
&p_K(z_1,z_2,z_3)\\
&:=K(z_{1}-z_{2})K(z_{1}-z_{3})+K(z_{2}-z_{1})K(z_{2}-z_{3})+K(z_{3}-z_{1})K(z_{3}-z_{2}).
\end{split}
\end{equation}
Supposing that $\mu_1$, $\mu_2$ and $\mu_3$ are measures, set
\begin{equation}
\label{p_K(mu)} p_K(\mu_1,\mu_2,\mu_3):=\iiint
p_K(z_1,z_2,z_3)\;d\mu_1(z_1)\,d\mu_2(z_2)\,d\mu_3(z_3).
\end{equation}
We write $p_K(\mu):=p_K(\mu,\mu,\mu)$ for short and call it
\textit{permutation of the measure $\mu$}. Moreover, in what follows $p_{K,\varepsilon}(\mu_1,\mu_2,\mu_3)$  stands  for the integral in the right hand side of (\ref{p_K(mu)}) defined over the set
\begin{equation*}
\{(z_1,z_2,z_3)\in \mathbb{C}^3:
|z_k-z_j|\ge \varepsilon>0, \quad 1\le k,j\le 3, \quad j\neq k\},
\end{equation*}
and $p_{K,\varepsilon}(\mu):=p_{K,\varepsilon}(\mu,\mu,\mu)$.

Identities similar to (\ref{paper1-permutation_Main}) and (\ref{p_K(mu)}) were first considered by Melnikov \cite{M} in the case of the Cauchy kernel $K(z)=1/z$. It can be easily seen that in the related case of $K(z)=\Re (1/z)=k_\infty(z)$ one has
\begin{equation}
\label{curvature_pointwise}
p_{k_\infty}(z_1,z_2,z_3)=\tfrac{1}{4}c(z_1,z_2,z_3)^2,
\end{equation}
where 
$$
c(z_1,z_2,z_3) =\frac{1}{R(z_1,z_2,z_3)}
$$
is the so called {\it Menger curvature} and $R(z_1,z_2,z_3)$ stands for the radius of the circle passing through $z_1$, $z_2$ and $z_3$. Clearly, $c(z_1,z_2,z_3) \ge 0$ for any  $(z_1,z_2,z_3)\in \mathbb{C}^3$ which is very important in applications. In what follows, $c^2(\mu):=4p_{k_\infty}(\mu)$ and $c^2_{\varepsilon}(\mu):=4p_{k_\infty,\varepsilon}(\mu)$ for a measure $\mu$.

The permutations (\ref{paper1-permutation_Main}) and (\ref{p_K(mu)})  for more general kernels $K$ were considered later by Chousionis, Mateu, Prat and Tolsa in \cite{CMPT} (see also \cite{CMPTcap}).

Now let $K$ be an odd real-valued Calder\'{o}n-Zygmund (i.e. satisfying well-known growth and smoothness conditions) kernel with permutations (\ref{paper1-permutation_Main}), being non-negative for any  $(z_1,z_2,z_3)\in \mathbb{C}^3$. If $\mu$ has \textit{$C_*$-linear growth}, i.e. there exists a
constant $C_*>0$ such that
\begin{equation}
\label{linear_growth}
\mu(B(z,r))\le C_*r\qquad \text{for all }  r>0 \text{ and }z\in \spt \mu,
\end{equation}
then the following relation between $p_K(\mu)$ and the $L^2(\mu)$-norm of $T_K1$ holds:
\begin{equation}
\label{MV-paper3}
\|T_{K,\varepsilon}1\|_{L^2(\mu)}^2=\tfrac{1}{3}p_{K,\varepsilon}(\mu)+\mathcal{R}_{K,\varepsilon}(\mu),\qquad
|\mathcal{R}_{K,\varepsilon}(\mu)|\lesssim C_{*}^{2}\mu(\complex).
\end{equation}
An analogous relation was first proved for the Cauchy kernel $K(z)=1/z$. It was done in the seminal paper \cite{MV} by Melnikov and Verdera. It turns out that one can follow  Melnikov-Verdera's proof to obtain the more general identity (\ref{MV-paper3})  (see, for example, \cite[Lemma 3.3]{CMPT}).

The formulas (\ref{curvature_pointwise}) and (\ref{MV-paper3}),
generating \textit{the curvature  method} (also knows as \textit{the symmetrization method}), are remarkable in the sense that
they relate an analytic notion (the operator $T_K$, in particular, the Cauchy transform) with a
metric-geometric one (permutations, in particular, curvature).

Note that the  $L^2(\mathcal{H}^1\lfloor E)$-boundedness of the
Cauchy transform and the identities   (\ref{curvature_pointwise}) and~(\ref{MV-paper3}) imply
that $c^2(\mathcal{H}^1\lfloor E)<\infty$. Consequently, it is
enough to show that $c^2(\mathcal{H}^1\lfloor E)<\infty$ implies rectifiability. This is actually how it was done in \cite{L}.

Take into account that  we usually write $p_t$    instead of $p_{k_t}$ in what follows, in order to simplify notation.  

What is more, recall that it is shown in \cite[Theorem 1 and Remark
1]{Chunaev2016} that
$$
  \left\{
  \begin{array}{l}
    p_{t}(z_1,z_2,z_3)\ge 0 \text{ for any } (z_1,z_2,z_3)\in \mathbb{C}^3,\text{ if } t\notin (-2,0), \\
    p_{t}(z_1,z_2,z_3) \text{ may change sign for some }(z_1,z_2,z_3)\in
\mathbb{C}^3, \text{ if } t\in(-2,0).
  \end{array}
  \right.
$$
These facts are illustrated in Figure~\ref{FIG1}. Moreover, by
\cite[Lemma~2]{ChMT2016},
\begin{equation}
\label{paper2-Main_inequality} p_0(z_1,z_2,z_3)\le
2p_\infty(z_1,z_2,z_3)\text{ for any } (z_1,z_2,z_3)\in
\mathbb{C}^3.
\end{equation}

\subsection{Beta numbers and densities}

For any closed ball $B=B(x,r)$ with center $x\in \mathbb{C}$ and radius $r>0$ and $1\le p <\infty$, let
\begin{equation}
\label{paper3-beta_ball}
\beta_{\mu,p}(B) := \inf_L \left(\frac1{r} \int_{B} \left(\frac{\dist(y,L)}{r}\right)^p\,d\mu(y)\right)^{1/p},
\end{equation}
where the infimum is taken over all affine lines $L\subset \mathbb{C}$.
The $\beta_{\mu,p}$ coefficients were introduced by David and Semmes \cite{DS1} and  are the generalization of the well-known Jones $\beta$-numbers  \cite{J-TSP}.

We will mostly deal with $\beta_{\mu,2}(2B_Q)$ and so by $L_Q$ we   denote a corresponding best approximating line, i.e. a line where the infimum is reached in (\ref{paper3-beta_ball}) for $B=2B_Q$ (see the definition of $B_Q$ below) and $p=2$ .

Throughout the paper we also use the following densities:
$$
\Theta_\mu(B):=\Theta_\mu(x,r)=\dfrac{\mu(B(x,r))}{r}, \quad \text{where}\quad B=B(x,r),\qquad
x\in\C,\qquad r>0.
$$

\section{Main Lemma and proofs of Theorems}
Theorem~\ref{paper3-theorem1} is implied by the following lemma.
\begin{mainlemma*}
\label{Main_Lemma_paper3}
There exist 
absolute constants $t_0>0$ and $c>0$ such that for any finite 
 measure $\mu$ with $C_*$-linear growth it holds that 
\begin{equation}
\label{paper3-eq_main_lemma} p_{\infty}(\mu)\le
t_0^{-2}p_{0}(\mu)+cC_*^2\mu(\mathbb{C}).
\end{equation}
\end{mainlemma*}
The proof of this result is long and technical and actually
takes the biggest part of this paper. Note that
(\ref{paper3-eq_main_lemma}) is a counterpart to the inequality
$p_{0}(\mu)\le 2p_{\infty}(\mu)$ that follows from
(\ref{paper2-Main_inequality}).

\subsection{Proof of Theorem~\ref{paper3-theorem1}} Suppose that Main Lemma holds. Then  the  identity~(\ref{MV-paper3}) and inequality (\ref{paper3-eq_main_lemma})  yield
\begin{align*}
\sup_{\varepsilon>0}\|T_{k_{\infty},\varepsilon}1\|_{L^2(\mu)}^2&\le\tfrac{1}{3}p_{{\infty}}(\mu)+cC_*^2\mu(\mathbb{C})\\
&\le \tfrac{1}{3}\,t_0^{-2}\,p_{{0}}(\mu)+cC_*^2\mu(\mathbb{C})\\
&\le t_0^{-2}\sup_{\varepsilon>0}\|T_{k_{0},\varepsilon}1\|_{L^2(\mu)}^2+cC_*^2\mu(\mathbb{C}),
\end{align*}
where $c>0$ is an absolute constant.
Applying the inequality $\sqrt{ax^2+b}\le\sqrt{a}x+\sqrt{b}$ that is
valid for $a,b,x\ge 0$, gives Theorem~\ref{paper3-theorem1}.

\subsection{Proof of Theorem~\ref{teonou}}
We now apply the perturbation method from \cite{ChMT2016}.
By the triangle inequality and Theorem~\ref{paper3-theorem1},
\begin{align*}
\sup_{\varepsilon>0}\|T_{k_t,\varepsilon}1\|_{L^2(\mu)}&=
\sup_{\varepsilon>0}\|(T_{k_0,\varepsilon}+t\cdot
T_{k_\infty,\varepsilon})1\|_{L^2(\mu)}\\
&\ge\sup_{\varepsilon>0} \|T_{k_0,\varepsilon}1\|_{L^2(\mu)}-|t|\sup_{\varepsilon>0}\|T_{k_\infty,\varepsilon}1\|_{L^2(\mu)}\\
&\ge
(t_0-|t|)\sup_{\varepsilon>0}\|T_{k_\infty,\varepsilon}1\|_{L^2(\mu)}-cC_*\sqrt{\mu(\complex)}.
\end{align*}
Consequently,
$$ \sup_{\varepsilon>0}\|T_{k_\infty,\varepsilon}1\|_{L^2(\mu)}\le
\frac{\sup_{\varepsilon>0}\|T_{k_t,\varepsilon}1\|_{L^2(\mu)}+cC_*\sqrt{\mu(\complex)}}{t_0-|t|},\qquad
|t|<t_0.
$$
Therefore, given any cube $Q \subset \mathbb{C}$,  applying this estimate to the measure $\mu\lfloor Q$,
 we get
\begin{equation}
\label{paper3-norms} \sup_{\varepsilon>0}\|T_{k_\infty,\varepsilon}\chi_Q\|_{L^2(\mu\lfloor Q)}\le
\frac{\sup_{\varepsilon>0}\|T_{k_t,\varepsilon}\chi_Q\|_{L^2(\mu\lfloor
Q)}+cC_*\sqrt{\mu(Q)}}{t_0-|t|},\qquad |t|<t_0.
\end{equation}
By a variant of the $T1$ Theorem of Nazarov, Treil and Volberg
from \cite[Theorem 9.40]{Tolsa_book}, we infer from (\ref{paper3-norms})
 that the $L^2(\mu)$-boundedness of $T_{k_t}$ with a fixed $t$ such that  $|t|<t_0$ implies that of $T_{k_\infty}$, and thus the Cauchy transform is
$L^2(\mu)$-bounded.

\section{Other corollaries}

Recall that a measure $\mu$
is \textit{Ahlfors-David regular} (\textit{AD-regular})
if 
$$
C^{-1}r\le \mu(B(z,r))\le
Cr,\qquad\text{where}\qquad z\in \spt \mu,\qquad 0<r<\diam (\spt
\mu),
$$
and $C>1$ is some fixed constant. A measure $\mu$ is called
\textit{uniformly rectifiable} if it is AD-regular and $\spt \mu$ is
contained in an AD-regular curve. One can summarise all up-to-date
results characterising uniformly rectifiable measures via
$L^2(\mu)$-bounded SIOs $T_{k_t}$ as follows.
\begin{corollary}
\label{paper3-corollary1} Let $\mu$ be an AD-regular
measure and $t\in(-\infty, -\sqrt{2})\cup (-t_0,\infty]$. The measure $\mu$ is uniformly rectifiable
if and only if the SIO $T_{k_t}$ is $L^2(\mu)$-bounded.
\end{corollary}
The part of Corollary~\ref{paper3-corollary1} for $t=\infty$, i.e. for the Cauchy transform, was
proved in \cite{MMV}; for $t=0$ in \cite{CMPT}; and for $t \in (-\infty, -\sqrt{2})\cup (0,\infty)$ in \cite{Chunaev2016,ChMT2016}.

\medskip

Furthermore, one can
formulate the following general result.
\begin{corollary}
\label{paper3-corollary2} 
Let $\mu$ be a measure
with linear growth and  $t\in (-\infty, -\sqrt{2})\cup (-t_0,\infty]$. If the SIO $T_{k_t}$  is $L^2(\mu)$-bounded, then
so are all $1$-dimensional SIOs associated with a wide class of sufficiently smooth kernels
kernels.
\end{corollary}

We refer the reader to \cite[Sections 1~and~12]{Tolsa2004} and
\cite[Theorem A]{Girela} for the more precise description of what is
meant by ``sufficiently smooth kernels''.

The part of
Corollary~\ref{paper3-corollary2} for $t=\infty$, i.e. for the
Cauchy transform, was proved in
\cite{Tolsa2004} (see also
\cite{Girela}) and for $t \in (-\infty, -\sqrt{2})\cup
(0,\infty)$ in \cite{Chunaev2016,ChMT2016}.

\section{Plan of the proof of Main Lemma}
\label{chap3-plan}

To prove Main Lemma, we will use a corona decomposition that is
similar, for example, to the ones from \cite{Tolsa-delta} and
\cite{AT}: it splits the David-Mattila dyadic lattice into some
collections of cubes, which we will call ``trees'', where the
density of $\mu$ does not oscillate too much and most of the measure
is concentrated close to a graph of a Lipschitz function. To
construct this function we will use a variant of the Whitney extension
theorem adapted to the David-Mattila dyadic lattice. Further, we
will show that the family of trees of the corona decomposition
satisfies a packing condition by arguments inspired by some of the techniques used in
\cite{AT} and  earlier in
\cite{Tolsa-bilip} to prove the bilipschitz ``invariance'' of
analytic capacity. More precisely, we will deduce Main Lemma from the two-sided estimate
\begin{equation}
\label{paper3-top-ineq}
p_{k_\infty}(\mu)\lesssim \sum_{R\in {\sf Top}}\Theta_\mu(2B_R)^2\mu(R)\lesssim p_{k_0}(\mu)+C_*^2\mu(\mathbb{C}),
\end{equation}
where ${\sf Top}$ is the family of top cubes for the above-mentioned trees. Note that the left hand side inequality in (\ref{paper3-top-ineq}) in essentially contained in \cite{Tolsa-delta} and verifying the right hand side inequality is actually the main objective in the proof.

It is worth mentioning that the structure of our trees is more
complicated than in \cite{AT}. This
is because we deal with permutations which are not comparable to
curvature in some cases and this leads to additional technical
difficulties. What is more, we are not able to use a nice
 theorem by David and Toro \cite{DT1} which  shortens the proof in \cite{AT} considerably. Indeed, this theorem would be useful to construct a chordal curve such that most of the measure $\mu$ is concentrated close to it. However, in our situation we need to control slope and therefore we have to deal with and to construct a graph of a Lipschitz function with well-controlled Lipschitz constant instead.

The plan of the proof of Main Lemma is the following. In Section
\ref{paper3-secdyad} we recall the properties of the David-Mattila dyadic
lattice.  We construct the trees and establish their properties in
Sections~\ref{paper3-secbalan}--\ref{paper3-sec_small_BS}. The main properties are
summarized in Section~\ref{paper3-sec_pack_Top}, where they are further used
 for constructing the corona type decomposition. The end of the proof of
Main Lemma is given in Section \ref{paper3-sec-the-end}.

Finally, in Section~\ref{paper3-sec_pack_Tree}  we show how one can
slightly change the proof of Main Lemma in order to give another
proof of a certain result from \cite{AT}  and obtain an alternative proof of the bi-Lipschitz invariance of the $L^2$-boundedness of the Cauchy transform.

\begin{remark}
The measure $\mu$ considered below is under assumptions of Main Lemma, i.e. $\mu$ is a finite measure with $C_*$-linear growth. Moreover, without loss of generality we additionally suppose that $\mu$ has \textit{compact support}.
\end{remark}

\section{The David-Mattila lattice}
\label{paper3-secdyad}

We use the dyadic lattice of cubes with small boundaries
constructed by David and Mattila \cite{David-Mattila}. The
properties of this lattice are summarized in the next lemma (for the case of $\mathbb{C}$).

\begin{lemma}[Theorem 3.2 in \cite{David-Mattila}]
\label{paper3-lemcubs} Let $\mu$ be a measure,
$E=\spt\mu$, and consider two constants $C_0>1$ and
$A_0>5000\,C_0$. Then there exists a sequence of partitions of $E$
into Borel subsets $Q$, $Q\in \mathcal{D}_k$, with the following properties:
\begin{itemize}
\item For each integer $k\ge0$, $E$ is the disjoint union of the ``cubes'' $Q$, $Q\in\mathcal{D}_k$, and
if $k<l$, $Q\in\mathcal{D}_l$, and $R\in\mathcal{D}_k$, then either $Q\cap R=\varnothing$ or else $Q\subset R$.

\item The general position of the cubes $Q$ can be described as follows. For each $k\ge0$ and each cube $Q\in\mathcal{D}_k$, there is a ball $B(Q)=B(z_Q,r(Q))$ such that
$$z_Q\in Q, \qquad A_0^{-k}\le r(Q)\le C_0\,A_0^{-k},$$
$$E\cap B(Q)\subset Q\subset E\cap 28\,B(Q)=E \cap B(z_Q,28r(Q)),$$
and
$$\mbox{the balls $5B(Q)$, $Q\in\mathcal{D}_k$, are disjoint.}$$

\item The cubes $Q\in\mathcal{D}_k$ have small boundaries. That is, for each $Q\in\mathcal{D}_k$ and each
integer $l\ge0$, set
$$N_l^{ext}(Q)= \{x\in E\setminus Q:\,\dist(x,Q)< A_0^{-k-l}\},$$
$$N_l^{int}(Q)= \{x\in Q:\,\dist(x,E\setminus Q)< A_0^{-k-l}\},$$
and
$$N_l(Q)= N_l^{ext}(Q) \cup N_l^{int}(Q).$$
Then
$$
\mu(N_l(Q))\le (C^{-1}C_0^{-7}A_0)^{-l}\,\mu(90B(Q)).
$$

\item Denote by $\mathcal{D}_k^{db}$ the family of cubes $Q\in\mathcal{D}_k$ for which
\begin{equation}\label{paper3-eqdob22}
\mu(100B(Q))\le C_0\,\mu(B(Q)).
\end{equation}
If $Q\in\mathcal{D}_k\setminus \mathcal{D}_k^{db}$, then $r(Q)=A_0^{-k}$ and
$$
\mu(100B(Q))\le C_0^{-l}\,\mu(100^{l+1}B(Q))\quad
\mbox{for all $l\ge1$ such that $100^l\le C_0$.}
$$
\end{itemize}
\end{lemma}

We use the notation $\mathcal{D}=\bigcup_{k\ge0}\mathcal{D}_k$. For $Q\in\mathcal{D}$, we set $\mathcal{D}(Q) =
\{P\in\mathcal{D}:P\subset Q\}$.
Observe that
 \begin{equation*}
 \label{paper3-comp_r_l}
 r(Q)\approx\diam(Q).
 \end{equation*}
 Also we call $z_Q$
the center of $Q$.  We set
$B_Q=28 \,B(Q)=B(z_Q,28\,r(Q))$, so that
$$E\cap \tfrac1{28}B_Q\subset Q\subset B_Q.$$

We denote $\mathcal{D}^{db}=\bigcup_{k\ge0}\mathcal{D}_k^{db}$ and
$\mathcal{D}^{db}(Q) = \mathcal{D}^{db}\cap \mathcal{D}(Q)$. Note
that, in particular, from (\ref{paper3-eqdob22}) it follows that
\begin{equation}
\label{paper3-doubling_property}
 \mu(100B(Q))\le
C_0\,\mu(2B_Q)\qquad\text{if } Q\in\mathcal{D}^{db}.
\end{equation}
For this reason we will call the cubes
from $\mathcal{D}^{db}$ \textit{doubling}.

As shown in \cite{David-Mattila}, any cube $Q\in\mathcal{D}$ can be
covered $\mu$-a.e. by doubling cubes.
\begin{lemma}[Lemma 5.28 in \cite{David-Mattila}]\label{paper3-lemcobdob}
Let $Q\in\mathcal{D}$. Suppose that the constants $A_0$ and $C_0$ in Lemma
\ref{paper3-lemcubs} are chosen suitably. Then there exists a family of
doubling cubes $\{Q_i\}_{i\in I}\subset \mathcal{D}^{db}$, with $Q_i\subset
Q$ for all $i$, such that their union covers $\mu$-almost all $Q$.
\end{lemma}
We denote by $J(Q)$ the number $k$ such that $Q\in \mathcal{D}_k$.
\begin{lemma}[Lemma 5.31 in \cite{David-Mattila}]
\label{paper3-lemcad22} Let $P\in\mathcal{D}$ and let $Q\subsetneq P$ be a
cube such that all the intermediate cubes $S$, $Q\subsetneq
S\subsetneq P$, are non-doubling $($i.e.  not in
$\mathcal{D}^{db}$$)$. Then
\begin{equation*}
\label{paper3-eqdk88}
\mu(100B(Q))\le A_0^{-20(J(Q)-J(P)-1)}\mu(100B(P)).
\end{equation*}
\end{lemma}

Recall that $\Theta_\mu(B)=\mu(B(x,r))/r$.
From Lemma~\ref{paper3-lemcad22} one can easily deduce\footnote{Note that there is an inaccuracy with constants in the original Lemma 2.4 in \cite{AT}.}
\begin{lemma}[Lemma 2.4 in \cite{AT}]\label{paper3-lemcad23}
Let $Q,P\in\mathcal{D}$ be as in Lemma \ref{paper3-lemcad22}. Then
$$
\Theta_\mu(100B(Q))\le
C_0 A_0^{-19(J(Q)-J(P)-1)+1}\,\Theta_\mu(100B(P))\le
C_0 A_0 \,\Theta_\mu(100B(P))
$$
and
$$
\sum_{S\in\mathcal{D}:Q\subset S\subset P}\Theta_\mu(100B(S))\le c\,\Theta_\mu(100B(P)),\qquad c=c(C_0,A_0).
$$
\end{lemma}

We will assume that all implicit constants in the inequalities that follow may depend on $C_{0}$ and $A_{0}$. Moreover, we will assume that $C_0$ and $A_0$ are some big \textit{fixed} constants so that the
results stated in the lemmas below hold.

\section{Balanced cubes and control on beta numbers through permutations}
\label{paper3-secbalan}

We first recall the properties of the so called balanced balls
introduced in \cite{AT}.

\begin{lemma}[Lemma 3.3 and Remark 3.2 in \cite{AT}]
\label{paper3-lemma_balanced_balls} Let $\mu$ be a measure 
and consider the dyadic lattice $\mathcal{D}$ associated with $\mu$
from Lemma \ref{paper3-lemcubs}. Let ${0<\gamma<1}$ be small enough $($with
respect to some absolute constant$)$, then there exist
$\rho'=\rho'(\gamma)>0$ and $\rho''=\rho''(\gamma)>0$ such that one
of the following alternatives holds for every
$Q\in\mathcal{D}^{db}$:
\begin{itemize}
\item[$(a)$] There are balls $B_k=B(\xi_k,\rho'\,r(Q))$, $k=1,2$, where $\xi_1,\xi_2\in B(Q)$, such that
\begin{equation*}
\label{paper3-bal_balls_mu}
 \mu\left(B_k\cap
B(Q)\right)\ge \rho''\,\mu(Q), \qquad k=1,2,
\end{equation*}
and for any $y_k\in B_k\cap Q$, $k=1,2$,
\begin{equation*}
\label{paper3-bal_balls_d} \dist(y_1,y_2) \ge \gamma\,r(B_Q).
\end{equation*}

\item[$(b)$] There exists a family of pairwise disjoint cubes $\{P\}_{P\in I_Q}\subset\mathcal{D}^{db}(Q)$ so that $\diam(P) \gtrsim\gamma\diam(Q)$  and
$\Theta_\mu(2B_{P})\gtrsim \gamma^{-1}\,\Theta_\mu(2B_Q)$ for each $P\in I_Q$,
and
\begin{equation}
\label{paper3-eqsgk32}
\sum_{P\in I_Q} \Theta_\mu(2B_{P})^2\,\mu(P)\gtrsim
\gamma^{-2}\,\Theta_\mu(2B_Q)^2\,\mu(Q).
\end{equation}
\end{itemize}
\end{lemma}

Let us mention that the densities in the latter inequality in the original Lemma~3.3  in \cite{AT} are not squared. However, a slight variation of the proof of  \cite[Lemma~3.3]{AT} gives (\ref{paper3-eqsgk32}) as stated.

Moreover, notice that in Lemma~\ref{paper3-lemma_balanced_balls} the cubes $Q$ and $P$, with $P\in
I_Q$, are doubling. If the alternative $(a)$ holds for a doubling cube
$Q$ with some $\gamma$, $\rho'(\gamma)$ and $\rho''(\gamma)$, then the corresponding
ball $B(Q)$ is called \textit{$\gamma$-balanced}. Otherwise, it is
called \textit{$\gamma$-unbalanced}. If $B(Q)$ is $\gamma$-balanced, then the cube $Q$ is also called \textit{$\gamma$-balanced}.

\bigskip

We are going to show now that the beta numbers $\beta_{\mu,2}(2B_Q)$
(see (\ref{paper3-beta_ball})) for $\gamma$-balanced cubes $Q$ are
controlled by a truncated version of the permutations
$p_0(\mu\lfloor 2B_Q)$. To do so, we introduce some additional
notation.

Given two distinct points $z,w\in \mathbb{C}$, we denote by
$L_{z,w}$ the line passing through $z$ and $w$. Given three pairwise
distinct points $z_1,z_2,z_3\in \mathbb{C}$, we denote by
$\measuredangle(z_1,z_2,z_3)$ the smallest angle formed by the lines
$L_{z_1,z_2}$ and $L_{z_1,z_3}$ and belonging to $[0,\pi/2]$. If $L$
and $L'$ are lines, let $\measuredangle(L,L')$ be the smallest angle
between them. This angle belongs to $[0,\pi/2]$, too. Also, we set
$\theta_V(L)=\measuredangle(L,V)$, where $V$ is the vertical line.

First we recall the following result of Chousionis and Prat
\cite{Chousionis-Prat}. We say that a triple $(z_1,z_2,z_3)\in
\mathbb{C}^3$ is \textit{in the class ${\rm V}_{\sf Far}(\theta)$}
if it satisfies
\begin{equation}
\label{paper3-far_from_vert}
\theta_V(L_{z_1,z_2})+\theta_V(L_{z_1,z_3})+\theta_V(L_{z_2,z_3})\ge
\theta>0.
\end{equation}
\begin{lemma}[Proposition 3.3 in \cite{Chousionis-Prat}]
\label{paper3-lemma_Chous_Prat} If $(z_1,z_2,z_3)\in
{\rm V}_{\sf Far}(\theta)$, then
\begin{equation}
\label{paper3-perm-curv} p_0(z_1,z_2,z_3)\ge {\sf c_1}(\theta)\cdot
p_\infty(z_1,z_2,z_3),\quad\text{where}\quad 0<{\sf c_1}(\theta)\le 2.
\end{equation}
\end{lemma}
Note that the inequality  ${\sf c_1}(\theta)\le 2$ follows from
(\ref{paper2-Main_inequality}) that was proved in \cite{ChMT2016}.

For measures $\mu_1$, $\mu_2$ and $\mu_3$ and a cube $Q$ we
set
\begin{equation*}
\label{paper3-perm_delta}
p^{[\delta,Q]}_0(\mu_1,\mu_2,\mu_3):=\iiint\limits_{\delta r(Q)\le |z_1-z_2|\le
\delta^{-1} r(Q)} p_0(z_1,z_2,z_3)\;d\mu_1(z_1)d\mu_2(z_2)d\mu_3(z_3).
\end{equation*}
The parameter $\delta>0$ will be chosen later to be small enough for
our purposes. If $\mu_1=\mu_2=\mu_3=\mu$, then we write
$p^{[\delta,Q]}_0(\mu)$ instead of $p^{[\delta,Q]}_0(\mu,\mu,\mu)$, for short.

Now we are ready to state the above mentioned estimate of
$\beta_{\mu,2}(2B_Q)$  for $\gamma$-balanced cubes $Q$ via the
truncated version of  $p_0(\mu\lfloor 2B_Q)$. Pay attention that the
first
 term in the estimate is a ``non-summable'' part which makes a big difference with the
case of curvature or $p_\infty$ (see
Section~\ref{paper3-sec_pack_Tree}).

\begin{lemma}
\label{paper3-lemma_beta_perm_initial} If $Q$ is $\gamma$-balanced, then
for any $\varepsilon\in(0,1)$,
\begin{equation}
\label{paper3-beta-perm-initial} \beta_{\mu,2}(2B_Q)^2\Theta_\mu(2B_Q)\le
4\varepsilon^2
\Theta_\mu(2B_Q)^2+C(\varepsilon,\gamma)\frac{p^{[\delta,Q]}_0(\mu\lfloor
2B_Q)}{\mu(Q)},\qquad 0<\delta\le\gamma.
\end{equation}
Moreover, for any $\varepsilon_0>0$, there exist
$\varepsilon=\varepsilon(\varepsilon_0)>0$ and
$\tilde{\varepsilon}=\tilde{\varepsilon}(\varepsilon_0,\gamma)>0$
such that if
\begin{equation}
\label{paper3-beta_cond}
\frac{p^{[\delta,Q]}_0(\mu\lfloor 2B_Q)}{\Theta_\mu(2B_Q)^2\mu(Q)}\le
\tilde{\varepsilon},
\end{equation}
then
$$
\beta_{\mu,2}(2B_Q)^2\le \varepsilon_0^2\Theta_\mu(2B_Q).
$$
\end{lemma}
\begin{proof}
By Lemma~\ref{paper3-lemma_balanced_balls}, there exist  balls
$B_k=B(\xi_k,\rho'\,r(Q))$, $k=1,2$, where $\xi_k\in B(Q)$, such that $\mu(B_k\cap B(Q))\ge
\rho''\mu(Q)$ and $\dist(y_1,y_2)\ge \gamma r(B_Q)$ for any
$y_k\in B_k\cap Q$, $k=1,2$. From (\ref{paper3-beta_ball}) it follows that
\begin{align*}
\beta_{\mu,2}(2B_Q)^2
&\le    \frac{1}{2r(B_Q)}\int_{2B_Q}\bigg(\frac{\dist(w,L_{y_1,y_2})}{2r(B_Q)}\bigg)^2d\mu(w).
\end{align*}
We separate triples $(w,y_1,y_2)$ that are in and not in ${\rm V}_{\sf Far}(\varepsilon)$. Clearly,
$$
\dist(w,L_{y_1,y_2})\le \diam(2B_Q) \sin\varepsilon \le 4\,\varepsilon
\,r(B_Q) \qquad \text{if } (w,y_1,y_2)\notin
{\rm V}_{\sf Far}(\varepsilon).
$$
Thus
\begin{align*}
&\beta_{\mu,2}(2B_Q)^2\\
&\quad\le    \frac{4\varepsilon^2}{2r(B_Q)}\int_{2B_Q} d\mu(w)+\frac{1}{2r(B_Q)}\int_{2B_Q, \;(w,y_1,y_2)\in {\rm V}_{\sf Far}(\varepsilon)}\bigg(\frac{\dist(w,L_{y_1,y_2})}{2r(B_Q)}\bigg)^2d\mu(w)\\
&\quad\le   4\varepsilon^2 \Theta_\mu(2B_Q)+  8r(B_Q)\int_{2B_Q, \;(w,y_1,y_2)\in {\rm V}_{\sf Far}(\varepsilon)}\bigg(\frac{2\dist(w,L_{y_1,y_2})}{|w-y_1||w-y_2|}\bigg)^2d\mu(w)\\
&\quad=   4\varepsilon^2 \Theta_\mu(2B_Q)+ 8r(B_Q)\int_{\,2B_Q, \;(w,y_1,y_2)\in {\rm V}_{\sf Far}(\varepsilon)}c(w,y_1,y_2)^2d\mu(w).
\end{align*}
We used that $|w-y_k|\le \diam(2B_Q)=4r(B_Q)$ as $w,y_1,y_2\in 2B_Q$ and that
$$
c(w,y_1,y_2)=\frac{2\dist(w,L_{y_1,y_2})}{|w-y_1||w-y_2|}.
$$
Recall that $r(B_Q)=28r(Q)$ by definition.
By (\ref{curvature_pointwise}) and (\ref{paper3-perm-curv}),
\begin{align*}
&\int_{\,2B_Q, \;(w,y_1,y_2)\in {\rm V}_{\sf Far}(\varepsilon)}c(w,y_1,y_2)^2d\mu(w) \\
& \qquad\qquad\qquad\le \frac{2}{{\sf c_1}(\varepsilon)}
                     \int_{\,2B_Q, \;(w,y_1,y_2)\in {\rm V}_{\sf Far}(\varepsilon)}p_0(w,y_1,y_2)d\mu(w).
\end{align*}

Recall  that $|y_1-y_2|\ge \gamma r(Q)$ for any
$y_k\in B_k\cap Q$, $k=1,2$. Furthermore, for any $\delta$ such that $0<\delta\le \gamma$ we can find $y_1\in B_1$ and $y_2\in B_2$ so that
$$
\int_{2B_Q}p_0(w,y_1,y_2)d\mu(w)
\le \frac{p^{[\delta,Q]}_0(\mu\lfloor 2B_Q)}{\mu(B_1)\mu(B_2)} \le
\frac{p^{[\delta,Q]}_0(\mu\lfloor 2B_Q)}{(\rho'')^2\mu(Q)^2}.
$$
By (\ref{paper3-eqdob22}) and the fact that $E\cap B(Q)\subset Q$, we
deduce that
\begin{equation*}
\label{paper3-muQ_mu2BQ} \mu(Q)\ge C_0^{-1}\mu(100B(Q))\ge
C_0^{-1}\mu(56B(Q))=C_0^{-1}\mu(2B_Q).
\end{equation*}
Consequently,
\begin{align*}
\beta_{\mu,2}(2B_Q)^2&\le 4\varepsilon^2 \Theta_\mu(2B_Q)+\frac{16r(B_Q)p^{[\delta,Q]}_0(\mu\lfloor
2B_Q)}{{\sf
c_1}(\varepsilon)(\rho'')^2\mu(Q)C_0^{-1}\mu(2B_Q)}\\
&=4\varepsilon^2 \Theta_\mu(2B_Q)+C(\varepsilon,\gamma)\frac{p^{[\delta,Q]}_0(\mu\lfloor
2B_Q)}{\Theta_\mu(2B_Q)\mu(Q)}.
\end{align*}
Multiplying both sides by $\Theta_\mu(2B_Q)$ finishes the proof of (\ref{paper3-beta-perm-initial}).
Note that $\rho''=\rho''(\gamma)$.

Let us prove the second statement. By the assumption (\ref{paper3-beta_cond}),
$$
\beta_{\mu,2}(2B_Q)^2\le
(4\varepsilon^2+C(\varepsilon,\gamma)\tilde{\varepsilon})\Theta_\mu(2B_Q).
$$
For any $\varepsilon_0>0$, we put $\varepsilon=\tfrac{\sqrt{2}} {4}
\varepsilon_0 $ and choose $\tilde{\varepsilon}$ so that
$\tilde{\varepsilon}\le
\tfrac{1}{2}\varepsilon_0^2/C(\varepsilon,\gamma)$.
\end{proof}

\section{Parameters and thresholds}
\label{paper3-sec_param}

Recall that we work everywhere with the David-Mattila dyadic lattice $\mathcal{D}$ associated with the measure $\mu$.

In what follows we will use many parameters and thresholds. Some of them depend on each other, some are independent. Let us give a list of the parameters:
\begin{itemize}
  \item $\tau$ is the threshold for cubes with low density:
  $$
  0<\tau\ll 1.
  $$
  \item $A$ is the threshold for cubes with high density:
  $$
  0<A^{-1}\le \tau^2\ll 1,\quad \text{i.e.} \quad  A\gg 1.
  $$
  \item $\theta_0$ is the threshold for the angle between best approximating lines associated to some cubes:
  $$
  0<\theta_0 \ll 1.
  $$
  \item $\gamma$ is the parameter controlling unbalanced cubes:
  $$
  0<\gamma\le \tau^3\ll 1.
  $$
    \item $\varepsilon_0$ is the threshold controlling the $\beta_{2,\mu}$-numbers:
    $$
    0<\varepsilon_0=\varepsilon_0(\gamma, \tau,A, \theta_0)\ll 1.
    $$
    \item $\alpha$ is the threshold controlling permutations of intermediate cubes:
    $$
    0<\alpha=\alpha(\tau,A,\varepsilon_0,\gamma,\theta_0)\ll 1.
    $$
    \item $\delta$ is the parameter controlling the truncation of permutations:
    $$
   0<\delta=\delta(\gamma,\varepsilon_0,\tau,A)\ll 1.
    $$
\end{itemize}

All the parameters and thresholds are supposed to be chosen (and fixed at the very end) so that the forthcoming results hold true. In what follows, we will again indicate step by step how the choice should be made.

\section{Stopping cubes and trees}
\label{paper3-sec_stopping}
\subsection{Stopping cubes}
\label{paper3-sec_stopping_cubes}
Let $R\in\mathcal{D}^{db}$. We  use the parameters and
thresholds given in Section~\ref{paper3-sec_param}. We denote by ${\sf
Stop}(R)$ the family of the maximal cubes $Q\subset R$  for which
one of the following holds:
\begin{enumerate}
\item[(\textbf{S1})] $Q\in \HD(R)\cup \LD(R)\cup {\sf UB}(R)$, where
\begin{itemize}
\item $\HD(R)$ is the family of \textit{high density} doubling cubes $Q\in \mathcal{D}^{db}$ satisfying
$$
\Theta_\mu(2B_Q)> A\, \Theta_\mu(2B_R);
$$

\item $\LD(R)$ is the family of \textit{low density} cubes $Q$ satisfying
$$
\Theta_\mu(2B_Q)< \tau\, \Theta_\mu(2B_R);
$$

\item ${\sf UB}(R)$ is the family of \textit{unbalanced} cubes $Q\in\mathcal{D}^{db} \setminus (\HD(R)\cup\LD(R))
$ such that $Q$ is $\gamma$-unbalanced;
\end{itemize}
\item[(\textbf{S2})] $Q\in \mathsf{BP}(R)$ (``\textit{big permutations}''), meaning $Q\notin  \HD(R)\cup \LD(R)\cup {\sf UB}(R) $ and
\begin{equation*}
\label{paper3-big_curv_stop}
 \sum_{Q\subset \tilde{Q}\subset R}
\textsf{perm}(\tilde{Q})^2>\alpha^{2},\qquad
\textsf{perm}(\tilde{Q})^2:=\frac{p_0^{[\delta,\tilde{Q}]}(\mu\lfloor 2
B_{\tilde{Q}},\mu\lfloor 2B_R,\mu\lfloor 2B_R)}{\Theta_\mu(2B_{R})^2\mu(\tilde{Q})}.
\end{equation*}
\item[(\textbf{S3})] $Q\in \mathsf{BS}(R)$ (``\textit{big slope}''), meaning $Q\notin \HD(R)\cup \LD(R)\cup {\sf UB}(R)\cup \mathsf{BP}(R)$
 and $Q\in \mathcal{D}^{db}$ so that
$$
\measuredangle(L_Q,L_R)>\theta(R),
$$
where $\theta(R)$ depends on some geometric properties of $R$ and is comparable with the parameter $\theta_0>0$ mentioned in  Section~\ref{paper3-sec_param}. The more precise description will be given in Section~\ref{paper3-sec_small_BS}.
\item[(\textbf{S4})] $Q\in \textsf{F}(R)$ (``\textit{big part of $Q$ is far from best approximating lines for the doubling ancestors of $Q$}''),
meaning $Q\notin  \HD(R)\cup
\LD(R)\cup {\sf UB}(R)\cup \mathsf{BP}(R)\cup \mathsf{BS}(R)$ and
    $$
    \mu(Q\setminus 2B_Q^{\sf Cl})>\sqrt{\alpha}\,\mu(Q),
    $$
where
\begin{equation}
\label{sup_new}
\begin{split}
2B_Q^{\sf Cl}:=\{&x\in R\cap 2B_Q: \;\dist(x,L_{\tilde{Q}})\le 5\sqrt{\varepsilon_0} \,r(B_{\tilde{Q}})\quad \forall
\tilde{Q}\in \mathcal{D}^{db}(R)\,: \\
& 2B_Q\subset 2B_{\tilde{Q}} \text{ and } \tilde{Q}\text{ is not contained in any cube
from }\\
& \HD(R)\cup
\LD(R)\cup {\sf UB}(R)\cup \mathsf{BP}(R)\cup \mathsf{BS}(R) \}.
\end{split}
\end{equation}
\end{enumerate}

Let ${\sf Tree}(R)$ be the subfamily of the cubes from $\mathcal{D}(R)$ which are not \textit{strictly} contained in any cube
from ${\sf Stop}(R)$. We also set
\begin{equation*}
\label{paper3-DbTree}
{\sf DbTree}(R):=\mathcal{D}^{db}\cap({\sf Tree}(R)\setminus {\sf Stop}(R)).
\end{equation*}
Note that all cubes in ${\sf Stop}(R)$ are disjoint.
\begin{remark}
\label{paper3-remark3}
It may happen that ${\sf Stop}(R)$ is empty. In this case there is no need to estimate the measure of stopping cubes and we may immediately go to Section~\ref{section_LIP}. In the lemmas below related to estimating the measure of stopping cubes we naturally suppose that ${\sf Stop}(R)$ is not empty.

Generally speaking it is possible that $R \in {\sf Stop}(R)$ (and then ${\sf DbTree}(R)$ is empty). Clearly, $R\notin \textsf{HD}(R) \cup \textsf{LD}(R)\cup {\sf BS}(R)$ by definition but it may occur that $R \in \textsf{UB}(R)\cup \textsf{BP}(R)\cup \textsf{F}(R)$. Firstly, we will not work with the family ${\sf UB}(R)$ before Section~\ref{paper3-sec_pack_Top} so we may assume before that section that $R\notin {\sf UB}(R)$. Secondly, if $R\in {\sf BP}(R)$, then we may directly go to Lemma~\ref{paper3-measure_BP_F_stop_cubes} and use the same estimate for the measure of stopping cubes from ${\sf BP}(R)$. Thirdly, it will follow from Lemmas~\ref{paper3-lemma_R_far_1} and~\ref{paper3-lemma_R_far_2} (see Remark~\ref{p3_remrk_R_F}) that if $R\notin {\sf UB}(R)\cup {\sf BP}(R)$, then $R\notin {\sf F}(R)$, i.e. the case $R\in {\sf F}(R)$ may be skipped.

It is also worth mentioning that if $R\in {\sf Stop}(R)$, then the  Lipschitz function mentioned in Section~\ref{chap3-plan} may be chosen identically zero and its graph is just $L_R$.
\end{remark}

\subsection{Properties of cubes in trees}
\label{section_R_Far} Below, we will collect main properties of cubes from ${\sf
Tree}(R)$ that readily follow from the stopping conditions. Before it
we prove an additional result.
\begin{lemma}
\label{paper3-density_of_cubes_hd} For any $Q\in {\sf Tree}(R)$, we have
$$
\Theta_\mu(2B_Q) \lesssim A\,\Theta_\mu(2B_R).
$$
The implicit constant depends only on $C_0$ and $A_0$.
\begin{proof}
Let $Q\in {\sf Tree}(R)$. If $Q\in \mathcal{D}^{db}$, then there is
nothing to prove. If not, then denote by $\tilde{Q}\in
\mathcal{D}^{db}$ the first doubling ancestor of $Q$. Such a cube
$\tilde{Q}$ exists and $\tilde{Q}\subset R$ because
$R\in\mathcal{D}^{db}$ by construction. Since the intermediate cubes
$P$, $Q\subsetneq P\subsetneq\tilde{Q}$, do not belong to
$\mathcal{D}^{db}$, by Lemma \ref{paper3-lemcad23} we have
$$
\Theta_\mu(2B_Q)\lesssim \Theta_\mu(100B(Q))\lesssim
C_0A_0\Theta_\mu(100B(\tilde{Q})).
$$
Using that $\tilde{Q}\in \mathcal{D}^{db}$, namely, the inequality
(\ref{paper3-doubling_property}), we get
$$
\Theta_\mu(2B_Q) \lesssim C_0^2A_0\,\Theta_\mu(2B_{\tilde{Q}})
\lesssim  C_0^2A_0\, A\,\Theta_\mu(2B_{R}),
$$
and we are done.
\end{proof}
\end{lemma}
\begin{lemma}
\label{paper3-DbTree_properties} The following properties hold:
\begin{equation} \label{paper3-theta_Q_Tree}
\tau\,\Theta_\mu(2B_R)\le \Theta_\mu(2B_Q) \lesssim
A\,\Theta_\mu(2B_R), \quad \forall Q\in {\sf Tree}(R)\setminus ({\sf
LD}(R)\cup {\sf
HD}(R)).
\end{equation}
\begin{equation} \label{paper3-Q_Tree_bal} Q\in \mathcal{D}^{db}\cap ({\sf Tree}(R)\setminus ({\sf HD}(R)\cup {\sf LD}(R)\cup  {\sf UB}(R)))\quad \Longrightarrow \quad \text{$Q$
is $\gamma$-balanced}.
\end{equation}
\begin{equation}
\label{paper3-big_perm_prop}
\sum_{Q\subset \tilde{Q}\subset R}
{\sf perm}(\tilde{Q})^2< \alpha^{2}
\quad \forall Q\in {\sf Tree}(R)\setminus ({\sf HD}(R)\cup {\sf LD}(R)\cup  {\sf UB}(R)\cup {\sf BP}(R)).
\end{equation}
\begin{equation}
\label{paper3-DbTree_beta}
\begin{split}
&\beta_{\mu,2}(2B_Q)^2\le
\varepsilon_0^2\Theta_\mu(2B_Q)\quad \text{ if } \alpha=\alpha(\gamma,\tau,\varepsilon_0)\text{ is small enough and}\\
&Q\in \mathcal{D}^{db}\cap ({\sf Tree}(R)\setminus ({\sf HD}(R)\cup {\sf LD}(R)\cup  {\sf UB}(R)\cup {\sf BP}(R))).
\end{split}
\end{equation}
\begin{equation}
\label{paper3-DBTree_BS} \measuredangle(L_Q,L_R)\le \theta(R)\qquad \forall
Q\in{\sf DbTree}(R).
\end{equation}
\begin{equation}\label{paper3-DbTree_Far}
\mu(Q\setminus 2B_Q^{\sf Cl})\le\sqrt{\alpha}\,\mu(Q)\qquad \forall
Q\in {\sf Tree}(R)\setminus {\sf Stop}(R).
\end{equation}
\end{lemma}
\begin{proof}
The statement (\ref{paper3-theta_Q_Tree}) follows from
Lemma~\ref{paper3-density_of_cubes_hd} and the stopping condition
\textbf{(S1)}. The statements (\ref{paper3-Q_Tree_bal}), (\ref{paper3-big_perm_prop}), (\ref{paper3-DBTree_BS}) and
(\ref{paper3-DbTree_Far}) immediately follow from
the construction of ${\sf Stop}(R)$ and ${\sf Tree}(R)$, while
(\ref{paper3-DbTree_beta}) is implied by
Lemma~\ref{paper3-lemma_beta_perm_initial} and the stopping conditions
\textbf{(S1)} and \textbf{(S2)}.
\end{proof}

The following property of $\gamma$-balanced cubes will be used many times below.
\begin{lemma}
\label{paper3-lemma_gamma_cubes_Hausdorf_dist}
Let $\varepsilon_0=\varepsilon_0(\gamma)$ be chosen small enough. Then for any $Q\in \mathcal{D}^{db}\cap ({\sf Tree}(R)\setminus ({\sf HD}(R)\cup {\sf LD}(R)\cup  {\sf UB}(R)\cup {\sf BP}(R)))$ there exist  two sets $\mathcal{Z}_k\subset Q$, $k=1,2$, such that
$$
\mu(Q)\lesssim_{\gamma}\mu(\mathcal{Z}_k)\le \mu(Q)\quad\text{and}\quad\dist(\mathcal{Z}_1,\mathcal{Z}_2)\ge \gamma r(B_Q),
$$
and moreover for any  $z_1\in \mathcal{Z}_1$ and $z_2\in \mathcal{Z}_2 $ we have
$$
\dist_H(L_{z_1,z_2}\cap 2B_Q,L_Q\cap 2B_Q)\le \sqrt{\varepsilon_0} \,r(B_Q).
$$
\end{lemma}
\begin{proof}
Since $Q\in \mathcal{D}^{db}\cap ({\sf Tree}(R)\setminus ({\sf HD}(R)\cup {\sf LD}(R)\cup  {\sf UB}(R)))$, $Q$ is $\gamma$-balanced by (\ref{paper3-Q_Tree_bal}). Furthermore, by Lemma~\ref{paper3-lemma_balanced_balls} there exist balls $B_k=B(\xi_k,\rho'r(Q))$, $k=1,2$, where $\xi_k\in B(Q)$, such that
$$
\mu(B_k\cap B(Q))\ge \rho''\mu(Q)\quad \text{and} \quad \dist(y_1,y_2)\ge \gamma r(B_Q)\quad \text{for any } y_k\in B_k,\; k=1,2,
$$
where $\rho'$ and $\rho''$ depend on $\gamma$. Due to the estimate
$\beta_{\mu,2}(2B_Q)^2\le \varepsilon_0^2\Theta_\mu(2B_Q)$ (see
(\ref{paper3-DbTree_beta})),
 by Chebyshev's inequality there exist $\mathcal{Z}_k\subset B_k\cap Q$ such that
$$
\mu(Q)\lesssim_{\gamma}\mu(B_k)\lesssim\mu(\mathcal{Z}_k)\le \mu(Q)\quad \text{and}\quad \sup_{z\in
\mathcal{Z}_k}\dist(z,L_Q)\lesssim_{\gamma}   \varepsilon_0 \,r(B_Q),\quad k=1,2.
$$
Thus for any $z_1\in \mathcal{Z}_1$ and $z_2\in \mathcal{Z}_2$  we
have
$$
\dist(z_k,L_Q)\lesssim_{\gamma}  \varepsilon_0 \,r(B_Q),\quad k=1,2,\qquad
\dist(z_1,z_2)\gtrsim_\gamma r(B_Q).
$$
This implies that $\measuredangle(L_{z_1,z_2},L_Q)\lesssim_\gamma
 \varepsilon_0$ and therefore the following estimate for the Hausdorff distance holds:
$$
\dist_H(L_{z_1,z_2}\cap 2B_Q,L_Q\cap 2B_Q)\lesssim_{\gamma}
 \varepsilon_0 \,r(B_Q).
$$
Choosing $\varepsilon_0$ small enough with respect to the implicit constant depending on $\gamma$, we obtain the required result.
\end{proof}

Clearly, it may happen that not all cubes in ${\sf Tree}(R)$ are
$\gamma$-balanced as there may be undoubling cubes. However, for any
cube in ${\sf Tree}(R)$, there is always an ancestor in ${\sf
DbTree}(R)$ close by. Namely, the following result holds.
\begin{lemma}[Lemma 6.3 in \cite{AT}]
\label{paper3-lemdob32} For any cube $Q\in{\sf Tree}(R)$  there exists a
cube $\tilde{Q}\supset Q$ such that $\tilde{Q}\in{\sf DbTree}(R)$
and $\diam(\tilde{Q})\le \lambda\diam(Q)$ with some
$\lambda=\lambda(A,\tau)$.
\end{lemma}

Now we want to show that the measure of the set of points from $R$
which are far from the best approximation lines for cubes in $\{R\}\cup ({\sf
Tree}(R)\setminus {\sf Stop}(R))$ is small. Set
\begin{equation*}\label{paper3-perm_delta_x}
p^{[\delta,Q]}_0(x,\mu,\mu):=\iint_{\delta r(Q)\le |x-y|\le
\delta^{-1} r(Q)}p_0(x,y,z)\;d\mu(y)d\mu(z)
\end{equation*}
and consider
\begin{align*}
R_{\sf Far}:=\{x\in R:\quad &
\frac{p^{[\delta,Q]}_0(x,\mu\lfloor 2B_R,\mu\lfloor 2B_R)}{\Theta_\mu(2B_R)^2}\ge
{\sf c_2}(\gamma,\tau,\varepsilon_0)\\
& \text{for some } Q\in \{R\}\cup ({\sf
Tree}(R)\setminus {\sf Stop}(R)) \text{ such that } x\in
2B_Q\},
\end{align*}
where ${\sf c_2}(\gamma,\tau,\varepsilon_0)>0$ will be defined precisely in
the proof of Lemma~\ref{paper3-lemma_R_far_2}.

\begin{lemma}
\label{paper3-lemma_R_far_1} If $R\notin {\sf UB}(R)\cup {\sf BP}(R)$ and $\alpha=\alpha(\gamma,\tau,\varepsilon_0)$ is
chosen small enough, then
$$
\mu(R_{\sf Far})\le \alpha\mu(R).
$$
\end{lemma}
\begin{proof}
By Chebyshev's inequality,
\begin{align*}
{\sf c_2}(\gamma,\tau,\varepsilon_0)&\,\mu(R_{\sf Far})\\
&\le \int_R\sum_{Q\in \{R\}\cup({\sf
Tree}(R)\setminus {\sf Stop}(R)):\,x\in 2B_Q}\!\!\!\!\!\frac{p^{[\delta,Q]}_0(x,\mu\lfloor 2B_R,\mu\lfloor 2B_R)}{\Theta_\mu(2B_R)^2} \;d\mu(x)\\
&\le \sum_{Q\in \{R\}\cup({\sf
Tree}(R)\setminus {\sf Stop}(R))}\frac{p^{[\delta,Q]}_0(\mu\lfloor 2 B_Q,\mu\lfloor 2B_R,\mu\lfloor 2B_R)}{\Theta_\mu(2B_R)^2} \\
&=\sum_{Q\in \{R\}\cup({\sf
Tree}(R)\setminus {\sf Stop}(R))}\frac{p^{[\delta,Q]}_0(\mu\lfloor 2
B_Q,\mu\lfloor 2B_R,\mu\lfloor
2B_R)}{\Theta_\mu(2B_R)^2\mu(Q)} \int\chi_Q(x)\,d\mu(x).
\end{align*}
Changing the order of summation yields
\begin{align*}
&{\sf c_2}(\gamma,\tau,\varepsilon_0)\,\mu(R_{\sf Far})\\ &\le \int_R  \left(\frac{p^{[\delta,R]}_0(\mu\lfloor 2B_R)}{\Theta_\mu(2B_R)^2\mu(R)}+
\sum_{Q \in {\sf
Tree}(R)\setminus {\sf Stop}(R):\,x\in Q}\frac{p^{[\delta,Q]}_0(\mu\lfloor 2B_Q,\mu\lfloor 2B_R,\mu\lfloor 2B_R)}{\Theta_\mu(2B_R)^2\mu(Q)}\right)\,d\mu(x)\\
&\le  \int_R\left(\textsf{perm}(R)^2+
\sum_{Q \in {\sf Tree}(R)\setminus  {\sf Stop}(R):\,x\in Q}\textsf{perm}(Q)^2\right)\,d\mu(x)\\
&\le 2\alpha^2\,\mu(R).
\end{align*}
Supposing that $2\alpha\le {\sf c_2}(\gamma,\tau,\varepsilon_0)$
gives the required result.
\end{proof}
Recall the definition (\ref{sup_new}).
\begin{lemma}
\label{paper3-lemma_R_far_2} Let $\delta=\delta(\varepsilon_0)$ be chosen small enough. If $x\in (R\cap 2B_{\tilde{Q}})\setminus
2B_{\tilde{Q}}^{\sf Cl}$ for~some ${\tilde{Q}\in {\sf Tree}(R)}$,
i.e. in particular there exists $Q\in \mathcal{D}^{db}(R)$ such that
$2B_Q\supset 2B_{\tilde{Q}}$ and  $Q$ is not contained in any cube
from $\HD(R)\cup
\LD(R)\cup {\sf UB}(R)\cup \mathsf{BP}(R)\cup \mathsf{BS}(R)$, then $x\in R_{{\sf Far}}$.
\end{lemma}
\begin{proof}
Clearly, $x\in 2B_Q$ and $Q\in \mathcal{D}^{db}\cap ({\sf Tree}(R)\setminus ({\sf HD}(R)\cup {\sf LD}(R)\cup  {\sf UB}(R)\cup {\sf BP}(R)))$. Therefore, by Lemma~\ref{paper3-lemma_gamma_cubes_Hausdorf_dist} we can find $\mathcal{Z}_k \subset Q$, $k=1,2$, such that
 for any $z_1\in \mathcal{Z}_1$ and $z_2\in \mathcal{Z}_2$  we have
$$
\dist_H(L_{z_1,z_2}\cap 2B_Q,L_Q\cap 2B_Q)\le
 \sqrt{\varepsilon_0} \,r(B_Q).
$$
Consider
triangle $(x,z_1,z_2)$  which is wholly contained in $2B_Q$.
It is easily seen that
\begin{equation}
\label{paper3-est_angle}
\dist(x,L_{z_1,z_2})\ge \dist(x,L_Q)-\dist_H(L_{z_1,z_2}\cap 2B_Q,L_Q\cap 2B_Q)\ge 4\sqrt{\varepsilon_0}\,r(B_Q).
\end{equation}
This implies that one of the angle of the triangle $(x,z_1,z_2)$ is at least
$$
\frac{4\sqrt{\varepsilon_0}\,r(B_Q)}{\diam(2B_Q)}= \sqrt{\varepsilon_0},
$$
 and thus
$(x,z_1,z_2)\in {\rm V}_{\sf Far}(\sqrt{\varepsilon_0})$ for any $z_1\in \mathcal{Z}_1$ and $z_2\in \mathcal{Z}_2$. Note also that (\ref{paper3-est_angle}) implies that $|x-z_1|\ge\delta r(Q)$ if $\delta=\delta(\varepsilon_0)$ is chosen small enough.
Consequently, by the identity (\ref{curvature_pointwise})
and Lemma~\ref{paper3-lemma_Chous_Prat},
\begin{align*}
p^{[\delta,Q]}_0(x,\mu\lfloor 2B_R,\mu\lfloor 2B_R)
&\ge\iint_{z_1\in \mathcal{Z}_1,\,z_2\in \mathcal{Z}_2}p_0(x,z_1,z_2)\;d\mu(z_1)d\mu(z_2)\\
&\ge \tfrac{1}{2} {\sf c_1}(\sqrt{\varepsilon_0})\iint_{z_1\in \mathcal{Z}_1,\,z_2\in \mathcal{Z}_2}c(x,z_1,z_2)^2\;d\mu(z_1)d\mu(z_2)\\
&= \tfrac{1}{2}{\sf c_1}(\sqrt{\varepsilon_0})\iint_{z_1\in
\mathcal{Z}_1,\,z_2\in \mathcal{Z}_2}\left(\frac{2\dist(x,L_{z_1,z_2})}{|x-z_1||x-z_2|}\right)^2\;d\mu(z_1)d\mu(z_2),
\end{align*}
where the constant ${\sf c_1}$ is from Lemma~\ref{paper3-lemma_Chous_Prat}.
Furthermore, we apply (\ref{paper3-est_angle}) and the fact that
$|x-z_k|\le \diam (2B_Q)=4r(B_Q)$ for $k=1,2$ to obtain the following:
$$
p^{[\delta,Q]}_0(x,\mu\lfloor 2B_R,\mu\lfloor 2B_R)
\ge  \frac{\varepsilon_0\,{\sf
c_1}(\sqrt{\varepsilon_0})}{8r(B_Q)^2}\,\mu(\mathcal{Z}_1)\mu(\mathcal{Z}_2).
$$
Since $\mu(\mathcal{Z}_k)\gtrsim_{\gamma} \mu(Q)$ by Lemma~\ref{paper3-lemma_gamma_cubes_Hausdorf_dist},
$\mu(Q)\gtrsim \mu(2B_Q)$ as
 $Q\in \mathcal{D}^{db}$ and $\Theta_\mu(2B_Q)\ge \tau \Theta_\mu(2B_R)$
 by (\ref{paper3-theta_Q_Tree}), we finally get
\begin{equation*}
\label{paper3-sigma_varepsilon}
p^{[\delta,Q]}_0(x,\mu\lfloor 2B_R,\mu\lfloor 2B_R) \gtrsim_{\gamma}
\varepsilon_0\,{\sf
c_1}(\sqrt{\varepsilon_0})\,\tau^2\Theta_\mu(2B_R)^2={\sf
c_2}(\gamma,\tau,\varepsilon_0)\,\Theta_\mu(2B_R)^2.
\end{equation*}
Consequently, $x\in R_{{\sf Far}}$ by definition.
\end{proof}

\begin{remark}
\label{p3_remrk_R_F}
Suppose that $R\in {\sf F}(R)$ and thus $\mu(R\setminus 2B_R^{\sf Cl})>\sqrt{\alpha}\,\mu(R)$ by definition. Then it is clear that $R\notin {\sf UB}(R)\cup {\sf BP}(R)$ (and furthermore $R\notin \textsf{HD}(R) \cup \textsf{LD}(R)\cup {\sf BS}(R)$, see Remark~\ref{paper3-remark3}) and so $\mu(R_{\sf Far})\le \alpha\,\mu(R)$ by Lemma~\ref{paper3-lemma_R_far_1}. Furthermore, $R\setminus 2B_R^{\sf Cl}\subset R_{\sf Far}$ by Lemma~\ref{paper3-lemma_R_far_2} (where one takes $R$ for both $Q$ and $\tilde{Q}$) and thus
$\mu(R\setminus 2B_R^{\sf Cl})\le \alpha\mu(R)$ which contradicts the fact that $R\in {\sf F}(R)$ as $\alpha\ll 1$.
\end{remark}

\section{Measure of stopping cubes from ${\sf BP}(R)$ and ${\sf F}(R)$}

\begin{lemma}
\label{paper3-measure_BP_F_stop_cubes}
It holds that
$$
\sum_{Q\in {\sf BP}(R)}\mu(Q)\le \frac{1}{\alpha^{2}\,\Theta_\mu(2B_{R})^2}\sum_{\tilde{Q}\in
{\sf Tree}(R)} p_0^{[\delta,\tilde{Q}]}(\mu\lfloor 2 B_{\tilde{Q}},\mu\lfloor 2B_R,\mu\lfloor 2B_R).
$$
 What is more, if $\alpha=\alpha(\tau)$ is small enough, then
$$
\sum_{Q\in {\sf F}(R)}\mu(Q)\le \sqrt{\alpha}\,\mu(R)\le \tfrac{1}{3}\sqrt{\tau}\,\mu(R).
$$
\end{lemma}
\begin{proof}
All the cubes in ${\sf Stop}(R)$ are disjoint and so are the cubes
in $\mathsf{BP}(R)$ and ${\sf F}(R)$. From $\textbf{(S3)}$ we get
\begin{align*}
\sum_{Q\in \textsf{BP}(R)}\mu(Q) & \le \frac{1}{\alpha^{2}}
\sum_{Q\in \textsf{BP}(R)}\sum_{Q\subset \tilde{Q}\subset R} \!\!\!
\frac{p_0^{[\delta,\tilde{Q}]}(\mu\lfloor
2B_{\tilde{Q}},\mu\lfloor 2B_R,\mu\lfloor 2B_R)}{\Theta_\mu(2B_{R})^2\mu(\tilde{Q})}\,\mu(Q)\\
& =  \frac{1}{\alpha^{2}\,\Theta_\mu(2B_{R})^2}\sum_{\tilde{Q}\in
{\sf Tree}(R)} \!\!\!p_0^{[\delta,\tilde{Q}]}(\mu\lfloor 2
B_{\tilde{Q}},\mu\lfloor 2B_R,\mu\lfloor 2B_R)\!\!\!\!\sum_{Q\in
\textsf{BP}(R):\,Q\subset \tilde{Q}}\frac{\mu(Q)}{\mu(\tilde{Q})}\\
& \le  \frac{1}{\alpha^{2}\,\Theta_\mu(2B_{R})^2}\sum_{\tilde{Q}\in
{\sf Tree}(R)} \!\!\! p_0^{[\delta,\tilde{Q}]}(\mu\lfloor 2 B_{\tilde{Q}},\mu\lfloor 2B_R,\mu\lfloor 2B_R).
\end{align*}

By $\textbf{(S4)}$ and Lemmas~\ref{paper3-lemma_R_far_1} and
\ref{paper3-lemma_R_far_2}, we obtain
$$
\sum_{Q\in {\sf F}(R)}\mu(Q)\le \frac{1}{\sqrt{\alpha}}\sum_{Q\in {\sf F}(R)}\mu(Q\setminus 2B_Q^{\sf Cl})\le \frac{1}{\sqrt{\alpha}}\,\mu(R_{\sf Far})\le \sqrt{\alpha}\,\mu(R),
$$
which finishes the proof.
\end{proof}

\section{Construction of a Lipschitz function}
\label{section_LIP}
We aim to construct a Lipschitz  function $F: L_R\to L_R^\perp$  whose graph $\Gamma_R$ is
close to $R$ up to the scale of cubes from ${\sf Stop}(R)$. We will mostly use the properties mentioned in Lemma~\ref{paper3-DbTree_properties}. This task is quite technical and so we  start with a bunch of auxiliary results. Note that, although we follow some of the methods from \cite{L} and \cite[Chapter 7]{Tolsa_book} quite closely, we need to adapt the whole construction to the David-Mattila lattice used in the current chapter (instead of the balls with controlled density used in \cite{L} and \cite{Tolsa_book}). 

Let us mention again that we may suppose that $R\notin {\sf Stop}(R)$ as otherwise we choose $F\equiv 0$ and the graph $\Gamma_R$ of $F$ is just $L_R$.

\subsection{Auxiliary results} As before, we denote by $L_Q$ a best
approximating line for the ball $2B_Q$ in the
sense of the beta numbers (\ref{paper3-beta_ball}). We need now to estimate
the angles between the best approximating lines corresponding to
cubes that are near each other. This task is carried out in the next
two lemmas. The first one is a well known result from \cite[Section
5]{DS1}. We formulate it for lines in the complex plane.
\begin{lemma}[\cite{DS1}]
\label{paper3-angles} Let $L_{1},L_{2}\subset \mathbb{C}$ be lines and
$z_1,z_2\in Z\subset \mathbb{C}$ be points so that
\begin{enumerate}
\item[$(a)$] $d_1=\dist(z_1,z_2)/\diam(Z)\in (0,1)$,
\item[$(b)$] $\dist(z_{i},L_{j})<d_2\diam(Z)$ for $i=1,2$ and $j=1,2$, where $d_2<d_1/4$.
\end{enumerate}
Then for any $z\in L_2$,
\begin{equation}
\dist(z,L_{1}) \le d_2\left(\tfrac{4}{d_1}\dist(z,Z)+\diam(Z)\right).
\end{equation}
\end{lemma}

We will use the preceding lemma to prove the following result.
\begin{lemma}\label{paper3-distlqlemma}
Let $\varepsilon_0=\varepsilon_0(\gamma)$ be chosen small enough. If $Q_{1},Q_{2}\in {\sf DbTree}(R)$ are such that
$r(Q_1)\approx r(Q_2)$ and $\dist(Q_1,Q_2)\lesssim
r(Q_j)$ for $j=1,2$,  then
\begin{align}
\label{paper3-distlq1} \dist(w,L_{Q_2})\lesssim  \sqrt{\varepsilon_0}\,(\dist(w,Q_1)+r(Q_1)),&\qquad w\in L_{Q_1}, \\
\label{paper3-distlq2}  \dist(w,L_{Q_1})\lesssim  \sqrt{\varepsilon_0}\,(\dist(w,Q_2)+r(Q_2)),& \qquad w\in L_{Q_2},\\
\label{paper3-distlq3}\measuredangle(L_{Q_1},L_{Q_2})\lesssim
 \sqrt{\varepsilon_0}.
\end{align}
\end{lemma}
\begin{proof}
Let $Q\in {\sf DbTree}(R)$ be the smallest cube such that
$2B_{Q}\supset B_{Q_1}\cap B_{Q_2}$. Clearly, $r(Q)\gtrsim r(Q_j)$,
$j=1,2$. Moreover, we can also guarantee that
$$
r(Q)\lesssim \dist(Q_1,Q_2)+\sum_{j=1}^{2}r(Q_j)\lesssim r(Q_j).
$$

Now we use arguments similar to those in Lemma~\ref{paper3-lemma_gamma_cubes_Hausdorf_dist}.
Since $Q_j\in {\sf
DbTree}(R)$ for $j=1,2$, by (\ref{paper3-Q_Tree_bal}) and
Lemma~\ref{paper3-lemma_balanced_balls} there are  balls
$B_{k,j}=B(\xi_{k,j},\rho'\,r(Q_j))$, $k=1,2$, where $\xi_{k,j}\in B(Q_j)$, such that $\mu(B_{k,j}\cap B(Q_j))\ge
\rho''\mu(Q_j)$ and $\dist(y_{1,j},y_{2,j})\ge \gamma r(B_{Q_j})\ge \gamma r(Q_j)$ for
all $y_{k,j}\in B_{k,j}\cap Q_j$, where $\rho'$ and $\rho''$ depend on $\gamma$.
Consequently, by (\ref{paper3-DbTree_beta}) and the fact that $r(B_{k,j})\approx_{\gamma}r(Q_j)$ we get
$$
\frac{1}{r(B_{k,j})}
\int_{B_{k,j}}\ps{\frac{\dist(w,L_{Q_j})}{r(B_{k,j})}}^{2}d\mu(w)\lesssim_\gamma
\beta_{\mu,2}(2B_{Q_j})^2\lesssim_\gamma \varepsilon_0^2\,\Theta_\mu(2B_{Q_j}).
$$
Since  $r(Q)\approx
r(Q_{j})$, we analogously obtain
$$
\frac{1}{r(B_{k,j})}
\int_{B_{k,j}}\ps{\frac{\dist(w,L_{Q})}{r(B_{k,j})}}^{2}d\mu(w)
\lesssim_\gamma  \beta_{\mu,2}(2B_{Q})^2  \lesssim_\gamma
\varepsilon_0^2\,\Theta_\mu(2B_Q).
$$
Therefore, using Chebyshev's inequality and again the relation $r(Q)\approx
r(Q_{j})$, we can find $z_{k,j}\in B_{k,j}\cap
Q_j$ such that
$$
\max\{\dist(z_{k,j},L_{Q_j}),\dist(z_{k,j},L_{Q})\}\lesssim_\gamma
 \varepsilon_0\,r(Q).
$$
Since $\dist(z_{1,j},z_{2,j})\ge \gamma r(Q_j)\gtrsim \gamma r(Q)$, it follows by Lemma \ref{paper3-angles} that
$$
\dist(w,L_Q)\lesssim_\gamma  \varepsilon_0 (\dist(w,Q_j)+r(Q_j))
\text{ for all }w\in L_{Q_j},\qquad j=1,2,
$$
and
$$
\dist(w,L_{Q_j})\lesssim_\gamma
\varepsilon_0(\dist(w,Q_j)+r(Q_j)) \text{ for all }w\in L_Q,\qquad j=1,2.
$$
From this, by the triangle inequality, choosing $\varepsilon_0$ small enough with respect to the implicit constant depending on $\gamma$, we obtain (\ref{paper3-distlq1}) and
(\ref{paper3-distlq2}).

The inequality (\ref{paper3-distlq3}) follows from (\ref{paper3-distlq1}) and (\ref{paper3-distlq2}) by elementary geometry.
\end{proof}
\begin{lemma}
\label{paper3-lemma_approx_line_Q1inQ2}
Let $\alpha=\alpha(\gamma)$ and $\varepsilon_0=\varepsilon_0(\gamma)$
be chosen small enough. If $Q_{1},Q_{2}\in {\sf
DbTree}(R)$ are such that $2B_{Q_1}\subset 2B_{Q_2}$  and $x\in L_{Q_1}\cap 2B_{Q_1}$, then
$$
\dist(x,L_{Q_2})\lesssim   \varepsilon_0^{1/3}\,r(Q_2).
$$
\end{lemma}
\begin{proof}
By Lemma~\ref{paper3-lemma_balanced_balls} there exists a family of balls
$B_k=B(\xi_k,\rho'\,r(Q_1))$, where $\xi_k\in B(Q_1)$, such that $\mu(B_k\cap B(Q_1))\ge
\rho''\mu(Q_1)$ and $\dist(y_1,y_2)\ge \gamma r(B_{Q_1})\ge \gamma r(Q_1)$
for any $y_k\in B_k\cap Q_1$, $k=1,2$. Recall that $\rho'$ and $\rho''$
depend on $\gamma$. Furthermore, we can choose
$\alpha=\alpha(\gamma)$ in (\ref{paper3-DbTree_Far}) small enough to guarantee that $B_k\cap B(Q_1)\cap 2B_{Q_1}^{\sf Cl} \neq \varnothing$. This and the definition of $2B_{Q_1}^{\sf Cl}$ imply that there exist  $z_k\in B_k\cap B(Q_1)\cap 2B_{Q_1}^{\sf Cl} $, $k=1,2$, such that
$$
\dist(z_k,L_{Q_j})\lesssim \sqrt{\varepsilon_0}\,r(Q_j),\quad
k=1,2,\quad j=1,2.
$$
Let $z_k'$ be the orthogonal projection of $z_k$ onto $L_{Q_1}$. We easily get from the previous inequality that
\begin{equation}
\label{paper3-eq_dist} \dist(z_k',L_{Q_2})\lesssim
\sqrt{\varepsilon_0}\,r(Q_2),\quad k=1,2.
\end{equation}

Moreover, $\dist(z_1,z_2)\gtrsim_\gamma r(Q_1)$ implies that
$\dist(z_1',z_2')\gtrsim_\gamma r(Q_1)$ and $z_k'\in 2B_{Q_1}$, if
$\varepsilon_0=\varepsilon_0(\gamma)$ is small enough. Having this
and (\ref{paper3-eq_dist}) in mind and taking into account that $x\in
L_{Q_1}\cap 2B_{Q_1}$, by elementary geometry we get the required
estimate for $\dist(x,L_{Q_2})$, assuming again that $\varepsilon_0=\varepsilon_0(\gamma)$ is small enough.
\end{proof}

\subsection{Lipschitz function $F$ for the good part of $R$}

For each given $R\in \mathcal{D}^{db}$, we first construct the required
function $F$ on the projection of the ``good part'' of $R$ onto $L_R$ and then extend it onto the whole $L_R$. In what follows, we will work a lot with the function
\begin{equation}
\label{paper3-def_d} d(z):=\inf_{Q\in {\sf DbTree}(R)}\{\dist(z,Q)+\diam(Q)\},\qquad z\in \mathbb{C}.
\end{equation}

Let us mention that $\theta(R)$ is supposed to be comparable with the parameter $\theta_0$, i.e. ${\theta(R)\approx \theta_0}$, where the implicit constants  will be defined in Section~\ref{paper3-sec_small_BS}.
\begin{lemma}
\label{paper3-Lip_func_good_R}
Let $\varepsilon_0=\varepsilon_0(\tau,A,\theta_0)$ and $\theta_0$ be small enough. For any $z_1,z_2\in cB_R$ we have
\begin{equation*}
\label{paper3-Lip_2} |\Pi^\perp(z_1)-\Pi^\perp(z_2)|\lesssim
\theta(R)|\Pi(z_1)-\Pi(z_2)|+c(\tau,A)\left(d(z_1)+d(z_2)\right),
\end{equation*}
where $\Pi(z)$ and $\Pi^\perp(z)$ are the projections of $z$ onto $L_R$ and $L_R^\perp$, correspondingly, and $c(\tau,A)>0$.
\end{lemma}
\begin{proof} Everywhere in the proof $k=1,2$.
For a fixed $h>0$ and any $z_k\in cB_R$ one can always find $Q_k\in {\sf DbTree}(R)$ such that
$$
\dist(z_k,Q_k)+\diam(Q_k)\le d(z_k)+h,\qquad k=1,2.
$$
Choose $z_k'\in Q_k$. Clearly, $|z_k-z_k'|\le d(z_k)+h$.

Let $\tilde{Q}_k\in {\sf DbTree}(R)$ be the smallest cube  such that $2B_{\tilde{Q}_k}\supset B_{Q_k}$ and
$$
r(\tilde{Q}_k)\approx_{\tau,A} \varepsilon_0|z_1-z_2|+\sum\nolimits_k\diam(Q_k).
$$
Now let  $\tilde{Q}\in {\sf DbTree}(R)$ be the smallest cube such that
$2B_{\tilde{Q}}\supset B_{\tilde{Q}_1}\cup B_{\tilde{Q}_2}$  and
$$
r(\tilde{Q})\approx_{\tau,A} |z_1-z_2|+\sum\nolimits_k\diam(Q_k).
$$
Note that $|z_1-z_2|\lesssim r(R)$ as $z_k\in cB_R$ and thus the cubes $\tilde{Q}_k$ and $\tilde{Q}$ are well defined.
Furthermore,  we easily get that $\varepsilon_0\, r(\tilde{Q})\lesssim_{\tau,A} r(\tilde{Q}_k)$. Consequently, the way how $\tilde{Q}_k$ and $\tilde{Q}$ are chosen and
the inequalities (\ref{paper3-theta_Q_Tree}) and (\ref{paper3-DbTree_beta}) in Lemma~\ref{paper3-DbTree_properties} imply that
\begin{align*}
\frac{1}{\mu(B_{\tilde{Q}_k})}
\int_{B_{\tilde{Q}_k}}\left(\frac{\dist(w,L_{\tilde{Q}})}{r(\tilde{Q})}\right)^2\,d\mu(w)&\lesssim_{\tau,A} \frac{r(\tilde{Q})\,\beta_{\mu,2}(2B_{\tilde{Q}})^2}{\mu(2B_{\tilde{Q}_k})} \lesssim_{\tau,A} \varepsilon_0^2\,\frac{r(\tilde{Q})\Theta_{\mu}(2B_{\tilde{Q}})}{\mu(2B_{\tilde{Q}_k})} \\
&\lesssim_{\tau,A} \varepsilon_0\,\frac{\Theta_\mu(2B_{\tilde{Q}})}{\Theta_\mu(2B_{\tilde{Q}_k})}\lesssim_{\tau,A} \varepsilon_0 \lesssim
\varepsilon_0^{3/4},
\end{align*}
if $\varepsilon_0=\varepsilon_0(\tau,A)$ is chosen properly. Recall again  that $r(B_Q)=28r(Q)$ by definition.

From the inequality just obtained  we deduce by Chebyshev's inequality that
there exist $z_k''\in R\cap B_{\tilde{Q}_k}$, $k=1,2$, such that
$$
|z_k''-\pi(z_k'')|\lesssim  \varepsilon_0^{3/8}\,r(\tilde{Q})\lesssim \sqrt[4]{\varepsilon_0}\left(|z_1-z_2|+\sum\nolimits_k\diam(Q_k)\right),
$$
where $\pi(z_k'')$ stands for the orthogonal projection of $z_k''$ onto $L_{\tilde{Q}}$ and  $\varepsilon_0=\varepsilon_0(\tau,A)$ is small enough.
Note also that
$$
|z_k'-z_k''| \lesssim r(\tilde{Q}_k)\lesssim \sqrt[4]{\varepsilon_0} |z_1-z_2|+c(\tau,A)\sum\nolimits_k\diam(Q_k),
$$
if $\varepsilon_0=\varepsilon_0(\tau,A)$ is small enough.
Summarizing, we obtain the inequality
$$
|z_k'-\pi(z_k'')|\le |z_k'-z_k''|+|z_k''-\pi(z_k'')|\lesssim \sqrt[4]{\varepsilon_0}|z_1-z_2|+c(\tau,A)\sum\nolimits_k\diam(Q_k).
$$

Furthermore, the triangle inequality yields
\begin{align*}
|\Pi^\perp (z_1')-\Pi^\perp (z_2')| &\le |\Pi^\perp (\pi(z_1''))-\Pi^\perp (\pi(z_2''))| +\sum\nolimits_k|\Pi^\perp (z_k')-\Pi^\perp (\pi(z_k''))|\\
& \le |\Pi^\perp (\pi(z_1''))-\Pi^\perp (\pi(z_2''))| +\sum\nolimits_k|z_k'-\pi(z_k'')|,
\end{align*}
and therefore we  immediately obtain
$$
|\Pi^\perp (z_1')-\Pi^\perp (z_2')|
\lesssim |\Pi^\perp (\pi(z_1''))-\Pi^\perp (\pi(z_2''))|+ \sqrt[4]{\varepsilon_0}|z_1-z_2|+c(\tau,A)\sum\nolimits_k\diam(Q_k).
$$
From (\ref{paper3-DBTree_BS}) in Lemma~\ref{paper3-DbTree_properties} applied to $\tilde{Q}$ and the triangle inequality we deduce that
\begin{align*}
|\Pi^\perp &(\pi(z_1''))-\Pi^\perp (\pi(z_2''))|\\
&\lesssim \theta(R) |\Pi(\pi(z_1''))-\Pi (\pi(z_2''))|\\
&\lesssim \theta(R) \left(|\Pi(z_1)-\Pi (z_2)|+\sum\nolimits_k|\Pi(z_k)-\Pi(\pi(z_k''))|\right)\\
&\lesssim \theta(R) \left(|\Pi(z_1)-\Pi (z_2)|+\sum\nolimits_k|z_k-\pi(z_k'')|\right)\\
&\lesssim \theta(R) \left(|\Pi(z_1)-\Pi (z_2)|+\sum\nolimits_k\left(|z_k-z_k'|+|z_k'-\pi(z_k'')|\right)\right).
\end{align*}
Recall the estimates for  $|z_k-z_k'|$ and $|z_k'-\pi(z_k'')|$ and take into account that $\diam(Q_k)\le d(z_k)+h$ and that $\varepsilon_0$ and $\theta_0$ (and thus $\theta(R)$) are small. Consequently,
$$|\Pi^\perp (z_1')-\Pi^\perp (z_2')|
\lesssim \theta(R) |\Pi(z_1)-\Pi (z_2)|+\sqrt[4]{\varepsilon_0}|z_1-z_2|+c(\tau,A)\sum\nolimits_k(d(z_k)+h).
$$
Additionally, the triangle inequality and the estimate for $|z_k-z_k'|$ lead to
$$
|\Pi^\perp (z_1)-\Pi^\perp (z_2)|\le |\Pi^\perp (z_1')-\Pi^\perp (z_2')|+\sum\nolimits_k(d(z_k)+h),
$$
and thus
$$
|\Pi^\perp (z_1)-\Pi^\perp (z_2)|\lesssim \theta(R) |\Pi(z_1)-\Pi (z_2)|+\sqrt[4]{\varepsilon_0}|z_1-z_2|+c(\tau,A)\sum\nolimits_k(d(z_k)+h).
$$
Take into account that $|z_1-z_2|\le |\Pi (z_1)-\Pi(z_2)|+|\Pi^\perp (z_1)-\Pi^\perp (z_2)|$ and choose $\varepsilon_0$ small enough with respect to $\theta_0$ (and thus to $\theta(R)$) and to the implicit absolute constant in the latter inequality. Finally,
$$
|\Pi^\perp (z_1)-\Pi^\perp (z_2)|\lesssim \theta(R) |\Pi (z_1)-\Pi(z_2)|+c(\tau,A)\sum\nolimits_k(d(z_k)+h).
$$
Letting $h\to 0$ finishes the proof.
\end{proof}

We will also use the following notation:
\begin{equation}
\label{paper3-good_R}
G_R=\{x\in \mathbb{C}: d(x)=0\}.
\end{equation}
Lemma~\ref{paper3-Lip_func_good_R} implies that the map $\Pi: G_R\to L_R$  is injective and we can define the function $F$ on $\Pi(G_R)$ by setting
\begin{equation}
\label{paper3-F_on_G_R}
F(\Pi(x))=\Pi^\perp(x), \qquad x\in G_R.
\end{equation}
Moreover, this $F$ is Lipschitz with constant $\lesssim\theta(R)$.

We are now aimed to extend $F$ onto the whole line $L_R$ using a
variant of the Whitney extension theorem. This approach is quite
standard and is used, for example, in \cite[Section 8]{DS1},
\cite[Section 3.2]{L} and \cite[Section 7.5]{Tolsa_book}. Therefore we will skip some details and mostly give the results related to the adaptation of the scheme to the
David-Mattila lattice that we use. These results will then
imply the extension of $F$ onto the whole $L_R$ by repeating the ``partition of unity''
arguments presented in \cite[Section
7.5]{Tolsa_book}.

Let us define the function
\begin{equation}
\label{paper3-def_D}
D(z):=\inf_{x\in\Pi^{-1}(z)}d(x)=\inf_{Q\in {\sf DbTree}(R)}\{\dist(z,\Pi(Q))+\diam(Q)\},\qquad z\in L_R.
\end{equation}

For each $z\in L_R$ such that $D(z)>0$, i.e. $z\in L_R\setminus
\Pi(G_R)$, we call $J_z$ the largest dyadic interval from $L_R$
containing $z$ such that
$$
\ell(J_z)\le \tfrac{1}{20}\inf_{u\in J_z} D(u).
$$
Let $J_i$, $i\in I$, be a relabelling of the set of all these intervals
$J_z$, $z\in L_R\setminus \Pi(G_R)$, without repetition. Some
properties of $\{J_i\}$ are summarized in the following lemma.
\begin{lemma}[Analogue of Lemma 7.20 in \cite{Tolsa_book}]
\label{paper3-Lip_3} The intervals in $\{J_i\}_{i\in I}$ have disjoint
interiors in $L_R$ and satisfy the properties:
\begin{enumerate}[label=$($\alph*$)$]
  \item If $z\in 15J_i$, then $5\ell(J_i)\le D(z)\le 50 \ell(J_i)$.
  \item There exists an absolute constant $c>1$ such that if $15J_{i}\cap 15 J_{i'}\neq \varnothing$, then
  $$
  c^{-1}\ell(J_i)\le  \ell(J_{i'})\le  c \,\ell(J_i).
  $$
  \item For each $i\in I$, there are at most $N$ intervals $J_{i'}$ such that $15J_{i}\cap 15 J_{i'}\neq \varnothing$, where $N$ is some absolute constant.
  \item $L_R\setminus \Pi(G_R)=\bigcup_{i\in I}J_i=\bigcup_{i\in I}15 J_i$.
\end{enumerate}
\end{lemma}

Now we construct the function $F$ on
$$
U_0=L_R\cap B_0,\qquad B_0=B(\Pi(x_0),10\diam(R)),
$$
where $x_0\in R$ is such that
$$
\dist(x_0,\Pi(x_0))=\dist(x_0,L_R)\le\diam(R).
$$
This $x_0$ exists due to the inequality (\ref{paper3-DbTree_beta}) in Lemma~\ref{paper3-DbTree_properties}. Note that by construction
\begin{equation}
\label{paper3-RinB}
R\subset
B(\Pi(x_0),2\diam(R)) \quad \text{and}\quad \Pi(R)\subset L_R\cap B(\Pi(x_0),2\diam(R)).
\end{equation}
We also define the following set of indexes:
$$
I_0=\{i\in I: J_i\cap U_0\neq \varnothing\}.
$$
\begin{lemma}
\label{paper3-lemma_J_0}
The following holds.

$(a)$ If $i\in I_0$, then $\ell(J_i)\le \diam(R)$ and $3J_i\subset L_R\cap B(\Pi(x_0),12\diam(R))$.

$(b)$ If $J_i\cap B(\Pi(x_0),3\diam(R))=\varnothing$ $($in particular if $i\notin I_0$$)$, then
      $$
      \ell(J_i)\approx \dist(\Pi(x_0),J_i)\approx |\Pi(x_0)-z|\quad \text{for all}\quad z\in J_i.
      $$
\end{lemma}
\begin{proof}
For $(a)$, take $J_i$ with $i\in I_0$ so that $J_i\cap U_0\neq \varnothing$. Then we have
$$
3J_i\subset L_R\cap B(\Pi(x_0),10\diam(R)+2\ell(J_i)).
$$
It is necessary to estimate $\ell(J_i)$. Recall that
$$
\ell(J_i)\le \tfrac{1}{20}\inf_{u\in J_i}D(u).
$$
Definitely, $\inf_{u\in J_i}D(u)\le \max_{u\in U_0} D(u)$ in our case so we will estimate this maximum instead.
To do so, we first notice that the definition (\ref{paper3-def_d}) of $d$ and the inequality (\ref{paper3-RinB}) give
$$
d(x)\le  \dist(x,R)+\diam(R) \le
13\diam(R),\qquad x\in B_0.
$$
This yields
$$
\max_{u\in U_0} D(u)\le \max_{x\in B_0}d(x)\le 13\diam(R),
$$
if we take into account the connection between $d$ and $D$ in (\ref{paper3-def_D}). Thus
$$
\ell(J_i)\le \tfrac{13}{20}\diam(R)
$$
and therefore
$$
3J_i\subset L_R\cap B(\Pi(x_0),(10+\tfrac{13}{10})\diam(R)).
$$

Now let us prove $(b)$. Let $z\in J_i$. Clearly, $\diam(R)\le \frac{1}{3}|\Pi(x_0)-z|$. Furthermore, we infer from this and the definition (\ref{paper3-def_D}) that
$$
D(z)\le(|\Pi(x_0)-z|+2\diam(R))+\diam(R)\le 2|\Pi(x_0)-z|.
$$
From another side, by (\ref{paper3-def_D}) and (\ref{paper3-RinB}),
$$
  D(z) \ge \dist(z,\Pi(R))\ge |\Pi(x_0)-z|-2\diam(R) \ge
  \tfrac{1}{3}|\Pi(x_0)-z|.
$$
Thus
$$
\tfrac{1}{3}|\Pi(x_0)-z|\le D(z)\le 2|\Pi(x_0)-z|,\qquad z\in J_i.
$$
Together with Lemma~\ref{paper3-Lip_3}$(a)$ this gives
$$
\tfrac{5}{2}\ell(J_i)\le|\Pi(x_0)-z|\le 150\ell(J_i).
$$
Moreover, since
$$
|\Pi(x_0)-z|-\ell(J_i)\le\dist(\Pi(x_0),J_i)\le |\Pi(x_0)-z|,\qquad z\in J_i,
$$
we get
$$
\tfrac{3}{2}\ell(J_i)\le\dist(\Pi(x_0),J_i)\le 150\ell(J_i),
$$
which finishes the proof.
\end{proof}
\begin{lemma}
\label{paper3-lemma_J_1} Given $i\in I_0$, there exists a cube $Q_i\in
{\sf DbTree}(R)$  such that
\begin{enumerate}[label=$($\alph*$)$]
  \item $\ell(J_i)\lesssim \diam(Q_i)\lesssim_{\tau,A} \ell (J_i)$;
  \item $\dist(J_i,\Pi(Q_i))\lesssim \,\ell(J_i)$.
\end{enumerate}
\end{lemma}
\begin{proof}
From the definition (\ref{paper3-def_D}) of $D$ it follows that there exists a cube $Q\in
{\sf DbTree}(R)$ such that
$$
\dist(z,\Pi(Q))+\diam(Q)\le 2D(z)\approx \ell(J_i),\qquad z\in J_i,
$$
where the comparability is due to Lemma~\ref{paper3-Lip_3}$(a)$. This
immediately gives $(b)$ and the right hand side inequality in $(a)$
for $Q_i=Q$. If the left hand side inequality in $(a)$ does not
hold, we can replace $Q$ by its smallest doubling ancestor $Q'$
satisfying $\diam(Q')\gtrsim \ell(J_i)$ so that all other inequalities
are valid (recall Lemma~\ref{paper3-lemdob32}). We rename $Q'$ by $Q_i$ then.
\end{proof}

For $i\in I_0$, let $F_i$ be the affine function $L_R\to L_R^\perp$ whose graph
is the line $L_{Q_i}$. Moreover, $F_i$ are Lipschitz functions with constant $\le \theta(R)$ as $\measuredangle(L_{Q_i},L_R)\le \theta(R)$ by (\ref{paper3-DBTree_BS}) in Lemma~\ref{paper3-DbTree_properties} taking into account that all $Q_i \in {\sf DbTree}(R)$. On the other hand, for
$i\notin I_0$, we set $F_i\equiv 0$, i.e. the graph of $F_i$ is just
 $L_R$ in this case.

\begin{lemma}
\label{paper3-lemma_J_3} If ${10J_i\cap 10J_{i'}}\neq\varnothing$ for some $i,i'\in I$, then
\begin{enumerate}[label=$($\alph*$)$]
  \item $\dist (Q_i,Q_{i'})\lesssim_{\tau,A} \ell (J_i)$ if moreover $i,i'\in
  I_0$;
  \item $|F_i(z)-F_{i'}(z)|\lesssim   \varepsilon_0^{1/3}\,\ell(J_i)$ for
  $z\in 100 J_i$;
  \item $|F'_i-F'_{i'}|\lesssim  \varepsilon_0^{1/3}$.
\end{enumerate}
\end{lemma}
\begin{proof}
For $i,i'\in I_0$, Lemmas~\ref{paper3-Lip_3}$(b)$ and~\ref{paper3-lemma_J_1}$(b)$ ensure that $\ell(Q_i)\approx \ell(Q_{i'})$ and
\begin{align*}
\dist(\Pi(Q_i),&\Pi(Q_{i'}))\\
 &\le \dist(\Pi(Q_i),J_{i})+\ell(J_i)
+\dist(J_i,J_{i'})+\ell(J_{i'})+\dist(J_{i'},\Pi(Q_{i'}))\\
&\lesssim \ell(J_i).
\end{align*}

Keeping this in mind, we continue. For any $z_1\in Q_i$ and $z_2\in Q_{i'}$ by the triangle inequality and Lemma~\ref{paper3-Lip_func_good_R} we have
\begin{align*}
  \dist(Q_i,Q_{i'}) & \le \dist(z_1,z_2) \le |\Pi^\perp(z_1)-\Pi^\perp(z_2)|+|\Pi(z_1)-\Pi(z_2)| \\
   & \lesssim  |\Pi(z_1)-\Pi(z_2)|+c(\tau,A)(d(z_1)+d(z_2)).
\end{align*}
Since $z_1\in Q_i$ and $z_2\in Q_{i'}$, we have $d(z_1)\le \diam(Q_i)$ and $d(z_2)\le \diam(Q_{i'})$. Moreover,
if $z_1$ and $z_2$ are chosen so that
$$
|\Pi(z_1)-\Pi(z_2)|\le 2\dist(\Pi(Q_i),\Pi(Q_{i'})),
$$
then
$\dist(Q_i,Q_{i'}) \lesssim  \dist(\Pi(Q_i),\Pi(Q_{i'}))+\diam(Q_{i})+\diam(Q_{i'})\lesssim_{\tau,A} \ell(J_i)$ as in
$(a)$.

For $i,i'\in I_0$ the properties $(b)$ and $(c)$ follow from $(a)$ and Lemma~\ref{paper3-distlqlemma}. Indeed, in this case
$$
\diam(Q_i)\approx \diam(Q_{i'})\approx_{\tau,A} \ell(J_i)\approx \ell(J_{i'})\quad\text{and}\quad \dist(Q_i,Q_{i'}) \lesssim_{\tau,A} \ell(J_i).
$$
Taking into account that $L_{Q_i}$ and $L_{Q_{i'}}$ are the graphs of $F_i$ and $F_{i'}$, correspondingly, by Lemma~\ref{paper3-distlqlemma} we have
$$
|F_i(z)-F_{i'}(z)|\lesssim_{\tau,A} \sqrt{\varepsilon_0}\,\ell(J_i)\lesssim \varepsilon_0^{1/3}\,\ell(J_i),\qquad z\in 100J_i,
$$
if $\varepsilon_0=\varepsilon_0(\tau,A)$ is chosen small enough.
Moreover, by the same lemma we have $\measuredangle(L_{Q_i},L_{Q_{i'}})\lesssim_{\tau,A}  \sqrt{\varepsilon_0}$ and thus
\begin{align*}
  |F_i'-F_{i'}'| & =|\arctan\measuredangle(L_{Q_i},L_{R})-\arctan \measuredangle(L_{Q_{i'}},L_R)| \\
    & =
|\arctan\measuredangle(L_{Q_i},L_{R})-\arctan (\measuredangle(L_{Q_i},L_{R})\pm\measuredangle(L_{Q_i},L_{Q_{i'}}))|\\
& \lesssim |\arctan\measuredangle(L_{Q_i},L_{Q_{i'}})|\\
&\lesssim_{\tau,A} \sqrt{\varepsilon_0}\\
&\lesssim \varepsilon_0^{1/3},
\end{align*}
if $\varepsilon_0=\varepsilon_0(\tau,A)$ is small enough.

For $i,i'\notin I_0$, $F_i\equiv F_{i'}\equiv 0$, and so $(b)$ and $(c)$ are trivial.

Finally, let $i\in I_0$ and $i'\notin I_0$.
From the assumption ${10J_i\cap 10J_{i'}}\neq\varnothing$ and Lemma~\ref{paper3-Lip_3}$(b)$ we know that $\ell(J_i)\approx\ell(J_{i'})$.
Moreover, by Lemma~\ref{paper3-lemma_J_0}$(a)$ we have $\ell(J_i)\le \diam(R)$ as $i\in I_0$. From another side, by Lemma~\ref{paper3-lemma_J_0}$(b)$
$$
\ell(J_{i'})\approx\dist(\Pi(x_0),J_{i'})
$$
and additionally $\dist(\Pi(x_0),J_{i'})\ge 10\diam(R)$ as $i'\notin I_0$, i.e. $J_{i'}\cap U_0= \varnothing$.
From all these facts we conclude that
$$
\ell(J_i)\approx\ell(J_{i'})\approx_{\tau,A} \diam(R)\quad\text{and}\quad \dist(J_i,J_{i'})\lesssim_{\tau,A} \diam(R).
$$
Recall that $F_{i'}\equiv 0$ and $J_{i'}\subset L_R$.
Then, using Lemma~\ref{paper3-lemma_J_1} and arguments close to those in the proof of Lemmas~\ref{paper3-lemma_gamma_cubes_Hausdorf_dist} and~\ref{paper3-distlqlemma}, one can show that $L_{Q_i}$ is very close to $L_R$ in $cB_0$, which yields $(b)$ and $(c)$ in this case if $\varepsilon_0=\varepsilon_0(\tau,A)$ is chosen small enough.
\end{proof}

\subsection{Extension of  $F$ to the whole $L_R$}

We are now ready to finish the definition of $F$ on the whole $L_R$.
Recall that $F$ has already been defined on $\Pi(G_R)$ (see (\ref{paper3-F_on_G_R})).
So it remains to define it only on $L_R\setminus\Pi(G_R)$. To this end, we first introduce a
partition of unity on $L_R\setminus \Pi(G_R)$. For each $i\in I$, we can find a function $\tilde{\varphi}_i\in
C^\infty(L_R)$ such that $\chi_{2J_i}\le\tilde{\varphi}_i\le\chi_{3J_i}$, with
$$
|\tilde{\varphi}_i \!\!'|\le\frac{c}{\ell(J_i)}\qquad\text{and} \qquad |\tilde{\varphi}_i \!\!''|\le\frac{c}{\ell(J_i)^2}.
$$
Then, for each $i\in I$, we set
\begin{equation}\label{paper3-varphi_i}
 \varphi_i = \frac{\tilde{\varphi}_i}{\sum_{j\in I}\tilde{\varphi}_j}.
\end{equation}
It is clear that the family $\{\varphi_i\}_{i\in I}$ is a partition of unity subordinated to the sets
$\{3J_i\}_{i\in I}$, and each function $\varphi_i$ satisfies
$$
|\varphi_i\,\!\!'|\le\frac{c}{\ell(J_i)}\qquad\text{and} \qquad |\varphi_i\,\!\!''|\le\frac{c}{\ell(J_i)^2},
$$
taking into account Lemma~\ref{paper3-Lip_3}.

Recall that $L_R\setminus\Pi(G_R)=\bigcup_{i\in I} J_i = \bigcup_{i\in I} 3J_i$.
For $z\in L_R\setminus\Pi(G_R)$, we define
$$
F(z) := \sum_{i\in I_0}\varphi_i(z)F_i(z).
$$
Observe that in the preceding sum we can replace $I_0$ by $I$ as $F_i\equiv 0$ for $i\in I\setminus I_0$.

We denote by $\Gamma_R$ the graph $\{(z,F(z)):z\in L_R\}$.

Using the lemmas proved above, one can undeviatingly follow the
``partition of unity'' arguments  in \cite[Section 7.5]{Tolsa_book} to prove the following.
\begin{lemma}
\label{paper3-lemma_Lip_A}
The function $F:L_R\to L_R^\perp$ is supported on  $L_R\cap B(\Pi(x_0),12\diam(R))$ and is $C_F\,\theta(R)$-Lipschitz, where $C_F>0$ is absolute. Also, if $z\in 15J_i$, $i\in I$, then
\begin{equation*}\label{paper3-estimate_A''}
|F''(z)|\lesssim \frac{\sqrt[4]{\varepsilon_0}}{\ell(J_i)}.
\end{equation*}
\end{lemma}

Recall that we suppose of course that the parameters and thresholds mentioned in Section~\ref{paper3-sec_param} are chosen properly.

\subsection{$\Gamma_R$ and $R$ are close to each other}

\begin{lemma}
\label{paper3-lemma_gamma_R_1}
 There exists a constant ${\sf
c_3}(\tau,A)>0$ such
that
\begin{equation}
\label{paper3-lemma_gamma_R_1.1} \dist(x,\Gamma_R)\le {\sf
c_3}(\tau,A)\cdot d(x)\quad
\text{for any }x\in B_0.
\end{equation}
\end{lemma}
\begin{proof}
Let $y=(\Pi(x),F(\Pi(x)))$. By Lemma~\ref{paper3-Lip_func_good_R},
\begin{equation}
\label{paper3-dist(x,Gamma)}
\dist(x,\Gamma_R)\le |x-y|=|\Pi^\perp(x)-\Pi^\perp(y)|\lesssim_{\tau,A} d(x)+d(y).
\end{equation}

If $\Pi(x)\in \Pi(G_R)$, then  $y\in G_R$ and thus $d(y)\equiv 0$, which proves the lemma.

If $\Pi(x)\notin \Pi(G_R)$, let $J_i$, $i\in I$, be such that $\Pi(x)\in J_i$. Since $\Pi(x)\in J_i\cap B_0\neq \varnothing$, $i\in I_0$ and therefore there exists a cube $Q_i\in {\sf DbTree}(R)$ described in Lemma~\ref{paper3-lemma_J_1}. This gives
$$
d(y)\le \dist(y,Q_i)+\diam(Q_i)\lesssim_{\tau,A} \dist(y,Q_i)+\ell(J_i).
$$

Let us estimate $\dist(y,Q_i)$. One can deduce from the definition of $F$ that there exist $y'\in L_{Q_i}$ such that $\Pi(y')=\Pi(y)$ and $\dist(y,y')\lesssim \ell(J_i)$ (recall that $L_{Q_i}$ is the graph of $F_i$ and $\Pi(y)\in J_i$, see some details in \cite[Proof of Lemma~7.24]{Tolsa_book}). Moreover, it follows in a similar way as in the proof of  Lemmas~\ref{paper3-lemma_gamma_cubes_Hausdorf_dist} and~\ref{paper3-distlqlemma} that there exist  $\zeta\in Q_i$ and $\zeta'\in L_{Q_i}$ such that $\dist(\zeta,\zeta')\lesssim \sqrt{\varepsilon_0}\diam(Q_i)$. We know from Lemma~\ref{paper3-lemma_J_1} that $\dist(\Pi(y'),\Pi(\zeta))\lesssim \ell(J_i)$. Furthermore,  it holds that    $\measuredangle(L_{Q_i},L_R)\le \theta(R)$  by (\ref{paper3-DBTree_BS}) in Lemma~\ref{paper3-DbTree_properties} taking into account that all $Q_i \in {\sf DbTree}(R)$. These facts imply  that $\dist(y',\zeta')\lesssim \ell(J_i)$. Summarizing, we obtain
$$
\dist(y,Q_i)\le \dist(y,y')+\dist(y',\zeta')+\dist(\zeta',\zeta)\lesssim \ell(J_i).
$$
From this by Lemma~\ref{paper3-Lip_3}$(a)$ and the definition of $D$ (see (\ref{paper3-def_D})), we conclude that
$$
d(y)\lesssim_{\tau,A} \ell(J_i)\lesssim_{\tau,A} D(\Pi(x))\lesssim_{\tau,A} d(x).
$$
This fact together with (\ref{paper3-dist(x,Gamma)}) proves the lemma.
\end{proof}
\begin{lemma}
\label{paper3-lemma_gamma_R_2}
Let $\varepsilon_0=\varepsilon_0(A,\tau)$ be small enough. If $Q\in {\sf DbTree}(R)$ and $z\in \Gamma_R\cap 2B_Q$, then
\begin{equation}
\label{paper3-lemma_gamma_R_2.1}
\dist(z,L_Q)\lesssim \sqrt[4]{\varepsilon_0}\,r(Q).
\end{equation}
\end{lemma}
\begin{proof}
Let $z\in G_R$. Then there exists $Q'\in {\sf DbTree}(R)$ such that $z\in Q'$, $Q'\subset Q$ and $r(Q')\le  \varepsilon_0^{1/3}\,r(Q)$. By Lemma~\ref{paper3-lemma_gamma_cubes_Hausdorf_dist} there is $z'\in Q'$ such that $\dist (z',z'')\lesssim \sqrt{\varepsilon_0}\,r(Q')$, where $z''\in L_{Q'}\cap 2B_{Q'}$. Furthermore, it is clear that $\dist(z,z')\lesssim r(Q')\lesssim  \varepsilon_0^{1/3}\,r(Q)$. Using that $Q'\subset Q$, by Lemma~\ref{paper3-lemma_approx_line_Q1inQ2} we get $\dist(z'',L_Q)\lesssim  \varepsilon_0^{1/3}\,r(Q)$. Consequently,
$$
\dist(z,L_Q)\le \dist (z,z')+\dist(z',z'')+
\dist(z'',L_Q)\lesssim \varepsilon_0^{1/3}\,r(Q).
$$

Now let $z\notin G_R$ and $\zeta=\Pi(z)$. In this case
$$
F(\zeta)=\sum_{i\in I_0}\varphi_i(\zeta)F_i(\zeta).
$$

Now take into account (\ref{paper3-varphi_i}) and distinguish two cases. Suppose first that
$$
\sum_{i\in I_0}\varphi_i(\zeta)=1.
$$
In this case $(\zeta,F(\zeta))$ is a convex combination of the points $(\zeta,F_i(\zeta))$ for $i$ such that $\varphi_i(\zeta)\neq 0$ (we will write $i\in \tilde{I}_0$ for these $i$\,s, $\tilde{I}_0\subset I_0$). Therefore (\ref{paper3-lemma_gamma_R_2.1}) follows if
\begin{equation}\label{paper3-lemma_gamma_R_2.1(1)}
\dist((\zeta,F_i(\zeta)),L_Q)\lesssim
 \varepsilon_0^{1/3}\,r(Q)\qquad \text{for all } i\in \tilde{I}_0.
\end{equation}
To prove this estimate, notice that since $z\in 2B_Q$,
$$
D(\zeta)\le d(z)\lesssim r(Q).
$$
Let $J_{i'}$, where $i'\in I_0$, be the interval that contains $\zeta$. Then
\begin{equation}
\label{paper3-eq1}
\ell(J_{i'})\le \tfrac{1}{20}D(\zeta)\lesssim r(Q).
\end{equation}
Recall that $\varphi_i$ is supported on $3J_i$. Consequently, we necessarily have $3J_i\cap J_{i'}\neq\varnothing$ if $i\in \tilde{I}_0$. Therefore by Lemma~\ref{paper3-Lip_3}$(b)$ and~\ref{paper3-lemma_J_1}$(a)$,
$$
\ell(J_i)\approx_{\tau,A}\diam(Q_i)\approx_{\tau,A}\diam(Q_{i'})\approx_{\tau,A} \ell(J_{i'})\lesssim_{\tau,A} r(Q).
$$
Moreover, by Lemma~\ref{paper3-lemma_J_3}$(a)$,
$$
\dist(\Pi(Q_i),\Pi(Q_{i'}))\le \dist(Q_i,Q_{i'})\lesssim_{\tau,A} \ell(J_i).
$$
Taking into account that
\begin{align*}
\dist&(\Pi(Q_{i'}),\Pi(Q))\\
&\le \dist(\Pi(Q_{i'}),J_{i'})+\diam(J_{i'}) +\dist(J_{i'},\Pi(Q))\lesssim \ell(J_{i'})\lesssim_{\tau,A}r(Q),
\end{align*}
we get
\begin{align*}
\dist&(\Pi(Q_i),\Pi(Q))\\
& \le \dist(\Pi(Q_i),\Pi(Q_{i'}))+\diam(\Pi(Q_{i'}))+\dist(\Pi(Q_{i'}),\Pi(Q))\lesssim_{\tau,A} r(Q).
\end{align*}
From Lemma~\ref{paper3-Lip_func_good_R}, applied for $z_1\in Q_i$ and $z_2\in Q$, we deduce that
$$
\dist(Q_i,Q)\lesssim \dist(\Pi(Q_i),\Pi(Q))+\diam(Q_i)+\diam(Q)\lesssim_{\tau,A} r(Q).
$$
This means that $2B_{Q_i}\subset cB_Q$ with some $c=c(\tau,A)>1$. Consequently, by Lemmas~\ref{paper3-lemdob32} and \ref{paper3-distlqlemma}, we can find $Q'\in {\sf DbTree}(R)$ such that $cB_Q\subset 2B_{Q'}$, $\diam(Q')\approx_{A,\tau} \diam (Q)$ and
$$
\dist(w,L_{Q})\lesssim_{A,\tau} \sqrt{\varepsilon_0}(\dist(w,Q')+\diam(Q')),\qquad  w\in L_{Q'}.
$$
Choosing $\varepsilon_0=\varepsilon_0(A,\tau)$ small enough, we get
\begin{equation}\label{paper3-eq2}
\dist(w,L_{Q})\lesssim \varepsilon_0^{1/3}(\dist(w,Q')+\diam(Q')),\qquad  w\in L_{Q'}.
\end{equation}
Recall that $(\zeta,F_i(\zeta))\in L_{Q_i}\cap cB_{Q_i}$ and $2B_{Q_i}\subset 2B_{Q'}$ so Lemma~\ref{paper3-lemma_approx_line_Q1inQ2} gives
$$
\dist((\zeta,F_i(\zeta)),L_{Q'})\lesssim  \varepsilon_0^{1/3}\,r(Q').
$$
Note that the parameters and thresholds in Lemma~\ref{paper3-lemma_approx_line_Q1inQ2} are also supposed to be properly chosen.
Together with (\ref{paper3-eq2}) applied to $w=\textsf{proj}_{L_{Q'}}(\zeta,F_i(\zeta))$, this yields (\ref{paper3-lemma_gamma_R_2.1(1)}) as required.

Suppose now that
$$
\sum_{i\in I_0}\varphi_i(\zeta)<1.
$$
In this case, there exists some $J_{i'}$ with $i'\notin I_0$ such that $\zeta\in 3J_{i'}$ (as from (\ref{paper3-varphi_i}) it follows that $\sum_{i\in I\setminus I_0}\varphi_i(\zeta)>0$) and by Lemma~\ref{paper3-lemma_J_0}$(b)$,
$$
\diam(R)\lesssim\ell(J_{i'})\approx \dist (\Pi(x_0),J_{i'}).
$$
Moreover, if $J_i$ is the interval that contains $\zeta=\Pi(z)$, $z\in 2B_Q$, then
$$
\ell(J_i)\lesssim D(\Pi(z))\lesssim d(z)\lesssim \dist(z,Q)+\diam(Q)\lesssim \diam(R),
$$
where we used the definition of $D$, see (\ref{paper3-def_D}).

By Lemma~\ref{paper3-Lip_3}$(b)$, $\ell(J_i)\approx \ell(J_{i'})$ as $J_i\cap 3J_{i'}\neq\varnothing$. That is why $\ell(J_{i'})\approx \diam(R)$. This also implies that $\ell(J_m)\approx \diam(R)$ for any $m\in I_0$ such that $\zeta\in 3J_m$. By Lemma~\ref{paper3-lemma_J_1}$(a)$, it means that $\diam(Q_m)\approx_{\tau,A} \diam(R)$. Furthermore, it is clear that $\dist(Q_m,R)\equiv 0$ and so the assumptions of Lemma~\ref{paper3-distlqlemma} are satisfied for $Q_m$ and $R$. Consequently, $L_{Q_m}$ and $L_R$ are very close in $cB_R$ for some $c>1$ if the corresponding parameters are chosen properly, namely,
\begin{equation}\label{paper3-dist_H_lemma_gamma_R}
\dist_H(L_{Q_m}\cap cB_R,L_R\cap cB_R)\lesssim_{\tau,A}  \sqrt{\varepsilon_0}\diam(R).
\end{equation}
On the other hand, arguing as in (\ref{paper3-eq1}), one deduces that $\ell(J_m)\lesssim_{\tau,A} r(Q)$, and from this we conclude that $r(Q)\approx_{\tau,A} \diam(R)$. By (\ref{paper3-dist_H_lemma_gamma_R}) then we get
$$
|F_m(\zeta)|=\dist((\zeta,F_m(\zeta)),L_R)\lesssim_{\tau,A} \sqrt{\varepsilon_0}\diam(R)\lesssim_{\tau,A} \sqrt{\varepsilon_0}\,r(Q)\lesssim \varepsilon_0^{1/3}\,r(Q)
$$
for all above-mentioned $m$s is $\varepsilon_0=\varepsilon_0(\tau,A)$ is chosen small enough. Recall that we only need to sum up $i\in I_0$ such that $\zeta\in 3J_i$ and these are our $m\in I_0$. Thus
\begin{align*}
\dist((\zeta,F(\zeta)),L_R) & \le \sum_{i\in I_0}\varphi_i(\zeta)|F_i(\zeta)|= \sum_{m\in I_0}\varphi_m(\zeta)|F_m(\zeta)| \\
 & \le \max_{m\in I_0} |F_m(\zeta)|\sum_{m\in I_0}\varphi_m(\zeta) \lesssim  \varepsilon_0^{1/3}\,r(Q).
\end{align*}
Due to the fact that $r(Q)\approx\diam(R)$, by Lemma~\ref{paper3-distlqlemma} lines $L_R$ and $L_Q$ are very close to each other in $2B_Q$, and thus
$$
\dist((\zeta,F(\zeta)),L_Q)\lesssim   \varepsilon_0^{1/3}\,r(Q)
$$
as desired.
\end{proof}
\begin{lemma}
\label{paper3-lemma_gamma_R_3} For all $x\in R\setminus R_{\sf Far}$,
\begin{equation}
\label{paper3-lemma_gamma_R_3.1} \dist(x,\Gamma_R)\lesssim
\sqrt[4]{\varepsilon_0}\;d(x).
\end{equation}
\end{lemma}
\begin{proof}
Recall that if $d(x)=0$, then $x\in\Gamma_R$ and we are done.

By Lemmas~\ref{paper3-lemma_R_far_2} and~\ref{paper3-lemma_gamma_R_2} any point $x\in R\setminus R_{\sf Far}$ is very close to $L_R$ and (\ref{paper3-lemma_gamma_R_3.1}) clearly holds if $d(x)\approx \diam(R)$. Hence, we may suppose below that $d(x)$ is small with respect to $\diam(R)$, say, $d(x)\ll ({\sf
c_3}(\tau,A)+2)\diam(R)$, where ${\sf
c_3}(\tau,A)>0$ is from Lemma~\ref{paper3-lemma_gamma_R_1}. 

Given $x\in R\setminus R_{\sf Far}$ with $d(x)>0$, take a cube $Q\in {\sf DbTree}(R)$ such that
$$
\dist(x,Q)+\diam(Q)\le 2d(x).
$$
Take any $z\in Q$ (note that $\dist(z,x)\le 2d(x)$) and find $Q'\in {\sf DbTree}(R)$  such that
$$
B(z,2({\sf
c_3}(\tau,A)+2)d(x))\subset \tfrac{3}{2}B_{Q'}.
$$
 Recall that $d(x)$ is  small with respect to $\diam(R)$ and thus $Q'$ can be found. We can also guarantee that
$r(Q')\approx_{\tau,A} d(x)$.
Furthermore, it is clear that $x\in B(z,2({\sf
c_3}(\tau,A)+2)d(x))$ and thus $x\in \tfrac{3}{2}B_{Q'}$. Moreover, Lemma~\ref{paper3-lemma_gamma_R_1} gives
$$
\dist(z,\Gamma_R)\le \dist(z,x)+\dist(x,\Gamma_R)\le (2+{\sf
c_3}(\tau,A))d(x),
$$
which yields that $B(z,2({\sf
c_3}(\tau,A)+2)d(x))\cap \Gamma_R \neq\varnothing$ and therefore
$$
\tfrac{3}{2}B_{Q'}\cap \Gamma_R\neq\varnothing.
$$
Take into account that $x\in \tfrac{3}{2}B_{Q'}\cap R\setminus R_{\sf Far}\subset 2B_{Q'}^{\sf Cl}$, i.e. $\dist(x,L_{Q'})\lesssim \sqrt{\varepsilon_0}\,r(Q')$ and thus there is $x'\in L_{Q'}\cap 2B_{Q'}$ such that
$\dist(x,x')\lesssim \sqrt{\varepsilon_0}\;r(Q')$. Furthermore, Lemma~\ref{paper3-lemma_gamma_R_2} says that $\dist(y,L_{Q'})\le c {\varepsilon_0}^{1/3}\,r(Q')$ for any $y\in \Gamma_R\cap 2B_{Q'}$ and some $c>0$. In other words,
$$
\Gamma_R\cap 2B_{Q'}\subset \mathcal{U}_{c {\varepsilon_0}^{1/3}\,r(Q')}(L_{Q'}),
$$
and thus $\dist(x',\Gamma_R)\lesssim  {\varepsilon_0}^{1/3}\;r(Q')$. Summarising, we get
$$
\dist(x,\Gamma_R)\le \dist(x,x')+\dist(x',\Gamma_R)\lesssim  {\varepsilon_0}^{1/3}\;r(Q').
$$
It is left to remember that $r(Q')\approx_{\tau,A} d(x)$ by construction and to choose $\varepsilon_0=\varepsilon_0(\tau,A)$ small enough.
\end{proof}

\begin{lemma}
\label{paper3-lemma_Q_i_proj_of_J_i}
For each $i\in I_0$,
$$
\dist(Q_i,\Gamma_R\cap\Pi^{-1}(J_i))\lesssim_{\tau,A} \ell(J_i).
$$
\end{lemma}
\begin{proof}
Let $x \in Q_i\subset B_0$. Then by Lemmas~\ref{paper3-lemma_gamma_R_1} and~\ref{paper3-lemma_J_1}$(a)$ we have
$$
\dist(Q_i,\Gamma_R)\le \dist(x,\Gamma_R)\lesssim_{\tau,A} d(x)\lesssim_{\tau,A} \diam(Q_i)\approx_{\tau,A} \ell(J_i).
$$
Moreover, $\dist(J_i,\Pi(Q_i))\lesssim \ell(J_i)$ by Lemma~\ref{paper3-lemma_J_1}$(b)$. From these two inequalities and Lemma~\ref{paper3-lemma_Lip_A}, the required result follows.
\end{proof}

We finish this section with one more result which can be easily deduced from Lemmas~\ref{paper3-lemma_Lip_A} (look at $\spt F$) and~\ref{paper3-lemma_gamma_R_2}.
\begin{lemma}
\label{paper3-lemma_Gamma_R_to_R}
For any $z\in \Gamma_R$, it holds that
$$
\dist(z,L_R)\lesssim \sqrt[4]{\varepsilon_0}\,r(R).
$$
\end{lemma}

\section{Small measure of the cubes from $\LD(R)$}

In what follows we show that the measure of the low-density cubes is small.
\begin{lemma}
\label{paper3-lemma_ld_small}
If $\varepsilon_0=\varepsilon_0(\tau,A)$ and $\tau$ are  small enough, then
\begin{equation}\label{paper3-sumld}
\sum_{Q\in \LD(R)}\mu(Q)\le\tfrac{1}{3}\sqrt{\tau}\,\mu(R).
\end{equation}
\end{lemma}
\begin{proof}
Recall that the the parameters and thresholds from Section~\ref{paper3-sec_param} are supposed to be chosen  so that all above-stated results hold true. Taking this into account, note that by Lemma~\ref{paper3-lemma_R_far_1} with $\alpha=\alpha(\tau)$, being small enough,  we have
$$
\mu(R_{\sf Far})\le\tfrac{1}{6}\,\sqrt{\tau}\mu(R),
$$
thus  for obtaining (\ref{paper3-sumld}) it suffices to show that
\begin{equation}
\label{paper3-ldnotfar_new}
\mu(\mathcal{S}_{\sf LD})\le\tfrac{1}{6}\,\sqrt{\tau}\mu(R), \quad \text{where } \mathcal{S}_{\sf LD}:=\bigcup_{Q\in \textsf{LD}(R)}Q \setminus R_{\sf Far}.
\end{equation}

By the Besicovitch covering theorem, there exist a countable collection of points $x_i\in \mathcal{S}_{\sf LD}$ such that
$$
\mathcal{S}_{\sf LD}\subset \bigcup\nolimits_i B(x_i,r(Q_i)) \quad \text{and}\quad \sum\nolimits_i \chi_{B(x_i,r(Q_i))}\le N,
$$
where $Q_i\in \textsf{LD}(R)$ is such that $x_i\in Q_i$, and $N$ is some fixed constant. Note that $B(x_i,r(Q_i))\subset 2B_{Q_i}$. From this it follows that
$$
\mu(\mathcal{S}_{\sf LD}) \le\sum\nolimits_i \mu(B(x_i,r(Q_i)))\\
\le \sum\nolimits_i \mu(2B_{Q_i})\lesssim \sum\nolimits_i \Theta_\mu(2B_{Q_i})r(Q_i).
$$
Since $Q_i\in \textsf{LD}(R)$, we have $\Theta_\mu(2B_{Q_i})<\tau \Theta_\mu(2B_R)$ by definition. Furthermore,
each $x_i\in \mathcal{S}_{\sf LD}$ satisfies Lemma~\ref{paper3-lemma_gamma_R_3} and moreover $d(x_i)\lesssim_{\tau,A} \diam(Q_i)$ (as $x_i$ also belongs to the first doubling ancestor of $Q_i$ with a comparable diameter with comparability constant $\lambda=\lambda(\tau,A)$, see Lemma~\ref{paper3-lemdob32}) so that
$$
\dist(x_i,\Gamma_R)\lesssim_{\tau,A}  \sqrt[4]{\varepsilon_0}\;r(Q_i)\lesssim \sqrt[8]{\varepsilon_0}\;r(Q_i),
$$
if $\varepsilon_0=\varepsilon_0(\tau,A)$ is small enough. This means that $\Gamma_R$ passes very close to the center of $B(x_i,r(Q_i))$ in terms of $r(Q_i)$.
Consequently,
$$
r(Q_i)\lesssim \mathcal{H}^1(\Gamma_R\cap B(x_i,r(Q_i)))
$$
as $\Gamma_R$ is a connected graph of a Lipschitz function. Thus we get
$$
\mu(\mathcal{S}_{\sf LD})\lesssim \tau \Theta_\mu(2B_R) \sum\nolimits_i \mathcal{H}^1(\Gamma_R\cap B(x_i,r(Q_i))).
$$
Since $\sum_i \chi_{B(x_i,r(Q_i))}\le N$ with an absolute constant $N$, we get by Lemma~\ref{paper3-lemma_Lip_A} that
$$
\sum\nolimits_i \mathcal{H}^1(\Gamma_R\cap B(x_i,r(Q_i)))\lesssim \mathcal{H}^1\left(\Gamma_R\cap \bigcup\nolimits_iB(x_i,r(Q_i))\right)\lesssim  \mathcal{H}^1(\Gamma_R\cap 2B_R)\lesssim r(B_R).
$$
From this we deduce that
$$
\mu(\mathcal{S}_{\sf LD})
\lesssim \tau \Theta_\mu(2B_R) r(B_R)\lesssim \tau \mu(2B_R)  \lesssim\tau \mu(R),
$$
where the latter inequality is due to the fact that $R\in \mathcal{D}^{db}$ by construction.
Finally, we obtain (\ref{paper3-ldnotfar_new}) if $\tau$ is chosen small enough.
\end{proof}

\section{Small measure of the cubes from ${\sf BS}(R)$ for $R$ whose best approximation line is far from the vertical}

\label{paper3-sec_small_BS}
\subsection{Auxiliaries and the key estimate for the measure of cubes from ${\sf BS}(R)$}
\label{paper3-sec_small_BS_1}
Given some $\theta_0>0$, we say that
\begin{eqnarray*}
    \nonumber & R \in  {\sf T}_{VF}(\theta_0) \text{ and } \theta(R)=\theta_0, \quad &\text{ if } \quad \theta_V(L_R)\ge (1+C_F)\,\theta_0; \\
    \nonumber & R \notin  {\sf T}_{VF}(\theta_0) \text{ and } \theta(R)=2(1+C_F)\,\theta_0, \quad &\text{ if } \quad \theta_V(L_R)< (1+C_F)\,\theta_0.
  \end{eqnarray*}
Note that $C_F>0$ is an absolute constant from Lemma~\ref{paper3-lemma_Lip_A} where it is stated that the function $F$ is $C_F\theta(R)$-Lipschitz. Recall that $\theta_0$ and $\theta(R)$ were first introduced and used in Sections~\ref{paper3-sec_param} and~\ref{paper3-sec_stopping_cubes}.

Let $R\in {\sf T}_{VF}(\theta_0)$. From the definition of the family ${\sf BS}(R)$ it follows that in this case we have
\begin{equation}
\label{paper3-BS1}
\measuredangle(L_Q,L_R)> \theta_0\qquad \forall Q\in {\sf BS}(R).
\end{equation}
On the other hand, if $Q\in {\sf DbTree}(R)$, then $\measuredangle(L_Q,L_R)\le\theta_0$ and thus
\begin{equation*}
\label{paper3-BS2}
\theta_V(L_Q)\ge (1+C_F)\,\theta_0-\measuredangle(L_Q,L_R)\ge C_F\,\theta_0\qquad \forall Q\in {\sf DbTree}(R).
\end{equation*}

In this section we are going to deal with $R\in {\sf T}_{VF}(\theta_0)$ only. Our  aim  is to prove the following assertion.
\begin{lemma}
\label{paper3-lemma_bs_small}
For any $R\in {\sf T}_{VF}(\theta_0)$, if $\varepsilon_0=\varepsilon_0(\tau)$ is chosen small enough, then
\begin{equation*}
\label{paper3-BS3}
\sum_{Q\in {\sf BS}(R)}\mu(Q) \le \tfrac{1}{3}\sqrt{\tau} \mu(R).
\end{equation*}
\end{lemma}

The rest of this section is devoted to the proof of this lemma.

\begin{remark}
\label{remark_BS_R}
It is natural to suppose in this section that ${\sf BS}(R)$ is not empty. This and Remark~\ref{paper3-remark3} imply that $R\notin {\sf Stop}(R)$ and thus ${\sf Tree}(R)\setminus {\sf Stop}(R)$ is not empty.
\end{remark}

\subsection{The measure of cubes from ${\sf BS}(R)$ is controlled by the permutations of the Hausdorff measure restricted to $\Gamma_R$}

Recall that the the parameters and thresholds from Section~\ref{paper3-sec_param} are supposed to be chosen  so that all above-stated results hold. Taking this into account, note that by Lemma~\ref{paper3-lemma_R_far_1} with $\alpha=\alpha(\tau)$, being small enough,  we have
$$
\mu(R_{\sf Far})\le\tfrac{1}{6}\,\sqrt{\tau}\mu(R),
$$
thus to prove Lemma~\ref{paper3-lemma_bs_small}, it suffices to show that
\begin{equation}
\label{paper3-bsnotfar_new}
\mu(\mathcal{S}_{{\sf BS}})\le\tfrac{1}{6}\,\sqrt{\tau}\mu(R),\qquad \text{where }  \mathcal{S}_{{\sf BS}}:=\bigcup_{Q\in \textsf{BS}(R)}Q \setminus R_{\sf Far}.
\end{equation}

The following results is the first step in proving (\ref{paper3-bsnotfar_new}). (Recall the identity (\ref{curvature_pointwise}).)

\begin{lemma}
\label{paper3-bs_lemma_1}
If $\theta_0$ and $\varepsilon_0=\varepsilon_0(\theta_0,\tau,A)$ are chosen small enough, then
\begin{equation*}
\label{paper3-eqBS} \mu(\mathcal{S}_{{\sf BS}})\lesssim_A
\frac{\,p_\infty(\Theta_{\mu}(2B_R)\,\mathcal{H}^1_{\Gamma_R})}{\theta_0^2\,\Theta_\mu(2B_R)^2}.
\end{equation*}
\end{lemma}
\begin{proof}
For every $x\in \mathcal{S}_{{\sf BS}}$ take the ball $B(x,r(Q_x))$, where $Q_x\in {\sf BS}(R)$ and is such that $x\in Q_x$. By the $5r$-covering theorem there exists a subfamily of pairwise disjoint balls $\{B(x_i,r(Q_i))\}_{i\in \hat{I}}$, where $Q_i=Q_{x_i}$, such that
$$
\mathcal{S}_{{\sf BS}}\subset R\cap \bigcup\nolimits_{i\in \hat{I}}B(x_i,5r(Q_i)).
$$

Let $B_i=B(x_i,\tfrac{1}{2}r(Q_i))$, $i\in \hat{I}$. Clearly, $B_i\subset B_{Q_i}$. Moreover, take into
account that  $Q_i \in \mathcal{D}^{db}$   by the
stopping condition \textbf{(S3)} and that $\mathcal{S}_{{\sf BS}}\cap R_{\sf Far}=\varnothing$ by definition.
Therefore, by Lemma~\ref{paper3-lemma_gamma_R_3},
$$
\dist(x_i, \Gamma_R)\lesssim \sqrt[4]{\varepsilon_0}\,d(x_i)\lesssim \sqrt[4]{\varepsilon_0}\,r(Q_i) <\tfrac{1}{4}r(Q_i),
$$
if $\varepsilon_0$ is small enough.  Thus $\Gamma_R\cap \tfrac{1}{2} B_i\neq \varnothing$ and therefore there exist $y_1,y_2\in \Gamma_R\cap B_i$ such that
$$
cr(Q_i)\le |y_1-y_2|\lesssim |\Pi(y_1)-\Pi(y_2)|
$$
with some small fixed constant $c>0$, where in the latter inequality we took into account that $\Gamma_R$ is a graph of a Lipschitz function $F$ (see Lemma~\ref{paper3-lemma_Lip_A}).

Now, by Lemma~\ref{paper3-lemdob32}, there exists $\tilde{Q}_i\in {\sf DbTree}(R)$ such that $Q_i\subset \tilde{Q}_i$ and moreover $\diam(Q_i)\approx_{\tau,A} \diam(\tilde{Q}_i)$. By Lemma~\ref{paper3-lemma_gamma_R_2},
$$
\dist(y_k,L_{\tilde{Q}_i})\lesssim  {\varepsilon_0}^{1/3}\,r(\tilde{Q}_i),\qquad y_k\in \Gamma_R\cap B_i,\quad k=1,2.
$$
At the same time, $\angle(L_{Q_i},L_{\tilde{Q}_i})\lesssim_{\tau,A} \sqrt[4]{\varepsilon_0}$ by arguments similar to those in the proof of Lemma~\ref{paper3-distlqlemma} (this lemma cannot be applied directly as $Q_i\notin {\sf DbTree}(R)$ but the arguments can still be adapted if one of the cubes is in ${\sf BS}(R)$). Therefore, if $\varepsilon_0=\varepsilon_0(\tau,A)$ is small enough, then one can show that
$$
\dist(y_k,L_{Q_i})\lesssim \sqrt[8]{\varepsilon_0}\,r(Q_i),\qquad y_k\in \Gamma_R\cap B_i\subset \Gamma_R\cap 2B_{Q_i},\quad k=1,2.
$$
Consequently, denoting by $y_k'$ the orthogonal projections of $y_k$ onto $L_{Q_i}$, we get
$$
|y_k-y_k'|\lesssim \sqrt[8]{\varepsilon_0}\,r(Q_i),\qquad k=1,2.
$$
Since $\measuredangle(L_{Q_i},L_R)>\theta_0$ by (\ref{paper3-BS1}) and $\varepsilon_0=\varepsilon_0(\theta_0)$ is small enough, it holds that
\begin{align*}
&|F(\Pi(y_1))-F(\Pi(y_2))|\\
&\quad =|\Pi^\perp(y_1)-\Pi^\perp(y_2)|  \ge |\Pi^\perp(y_1')-\Pi^\perp(y_2')|-\sum\nolimits_k|y_k-y_k'|\\
   & \quad \gtrsim \theta_0|\Pi(y_1')-\Pi(y_2')|-\sum\nolimits_k|y_k-y_k'| \gtrsim \theta_0|\Pi(y_1)-\Pi(y_2)|-2\sum\nolimits_k|y_k-y_k'| \\
   & \quad \gtrsim \theta_0 r(Q_i)-\sqrt[8]{\varepsilon_0} r(Q_i)  \gtrsim \theta_0 r(Q_i),
\end{align*}
where $k=1,2$. Thus,
$$
\int_{\Pi(B_i)}|F'(z)|dz\ge \left|\int_{\Pi(y_1)}^{\Pi(y_2)}F'(z)dz\right|=|F(\Pi(y_1))-F(\Pi(y_2))|\gtrsim \theta_0 r(Q_i).
$$
This and H\"{o}lder's inequality yield
$$
\theta_0 r(Q_i)\lesssim \sqrt{r(B_i)}\|F'\|_{2,\Pi(B_i)}\approx \sqrt{r(Q_i)}\|F'\|_{2,\Pi(B_i)},
$$
and finally
$$
r(Q_i) \lesssim \theta_0^{-2}\|F'\|_{2,\Pi(B_i)}^2.
$$

Since the balls $2B_i$, $i\in \hat{I}$, are pairwise disjoint by construction, so are the intervals $\Pi(B_i)\subset L_R$, $i\in \hat{I}$, if $\theta_0$ is chosen small enough. 
This is a consequence of the fact that $x_i$, the centres of $B_i$, lie very close to $\Gamma_R$, namely, $\dist(x_i, \Gamma_R) \lesssim \sqrt[4]{\varepsilon_0}\,r(B_i)$, and moreover $\measuredangle(L_{x_i,x_j},L_R)\lesssim \theta_0$ for all $i,j\in \hat{I}$ as $\Gamma_R$ is Lipschits with constant $\lesssim\theta_0$, see Lemma~\ref{paper3-lemma_Lip_A}. By this reason we have
\begin{align*}
   \mu(\mathcal{S}_{\sf BS}) &\le \sum\nolimits_{i\in\hat{I}}\mu(B(x_i,5r(Q_i)))\lesssim \sum\nolimits_{i\in\hat{I}}\Theta_\mu(2B_{{Q}_i}) r(Q_i)\\
   & \lesssim_A \theta_0^{-2}\Theta_\mu(2B_R)\sum\nolimits_{i\in\hat{I}}\|F'\|_{2,\Pi(B_i)}^2 \lesssim_A \theta_0^{-2}\Theta_\mu(2B_R)\|F'\|_{2}^2.
 \end{align*}

Now take into account that under the assumption that $\|F'\|_\infty\le 1/10$ (which is satisfied if $\theta_0$ is sufficiently small) by \cite[Lemma 3.9]{Tolsa_book} we have
$$
\|F'\|_{2}^2 \approx p_\infty(\mathcal{H}^1_{\Gamma_R})\approx \Theta_\mu(2B_R)^{-3}\,p_\infty(\Theta_\mu(2B_R)\,\mathcal{H}^1_{\Gamma_R})
$$
with some absolute constants.
\end{proof}

We claim that $p_\infty(x,y,z)$ is well controlled by $p_0(x,y,z)$ for any $x,y\in \Gamma_R$ if $R\in {\sf T}_{VF}(\theta_0)$.
\begin{lemma}
\label{paper3-lemma_compar_p_0_p_infty}
If $R\in {\sf T}_{VF}(\theta_0)$, then
$$
p_\infty(x,y,z)\lesssim_{\theta_0}p_0(x,y,z)\qquad \text{for any}\quad x,y\in \Gamma_R.
$$
\end{lemma}
\begin{proof}
The fact that the function $F$ (whose graph is $\Gamma_R$) is $C_F\theta(R)$-Lipschitz
by Lemma~\ref{paper3-lemma_Lip_A} and the definitions at the beginning of Subsection~\ref{paper3-sec_small_BS_1} yield
$$
\measuredangle(L_{xy},L_R)\le C_F\,\theta_0\qquad \text{ and } \qquad \theta_V(L_R)\ge (1+C_F)\,\theta_0.
$$
Consequently,
$$
\theta_V(L_{xy})\ge \theta_V(L_R)-\measuredangle(L_{xy},L_R)\ge (1+C_F)\,\theta_0-C_F\,\theta_0=\theta_0.
$$
Therefore $(x,y,z)\in {\rm V}_{\sf Far}(\theta_0)$ and it is left to use Lemma~\ref{paper3-lemma_Chous_Prat}.
\end{proof}
For $x\in \mathbb{C}$ such that $\Pi(x)\notin \Pi(G_R)$, set
$$
J_x=J_i,\quad i\in I,\quad\text{such that } \Pi(x)\in J_i,
$$
and
$$
\ell_x=\ell(J_x).
$$
If $\Pi(x)\in \Pi(G_R)$, we write
$$
J_x=\Pi(x) \qquad\text{and}\qquad \ell_x=0,
$$
i.e. one should think that in this case the point $\Pi(x)$ is a degenerate interval $J_x$ with zero side length.
To simplify notation, throughout this section we also write
$$
x_1=\Pi(x)\qquad \text{and}\qquad x_2=\Pi^\perp(x).
$$

Recall that the intervals $\{J_i\}$, $i\in I_0$, are the ones from $\{J_i\}$, $i\in I$, that intersect the ball $B_0=B(\Pi(x_0),10\diam(R))$, where $x_0\in R$ is such that $\dist(x_0,L_R)\lesssim r(R)$ (see (\ref{paper3-RinB})).
Observe that if $z\in U_0=L_R\cap B_0$, then $D(z)\lesssim r(R)$. Thus $\ell(J_i)\lesssim r(R)$ for all $i\in I_0$. Thus, setting
$$
\Gamma_{B_0}=G_R\cup \bigcup_{i\in I_0}\Gamma_R\cap \Pi^{-1}(J_i),
$$
we deduce that $\Gamma_{B_0}\subset c'B_0$ with some fixed $c'>0$. It is also true that $B_0\subset c''B_R$ with some $c''>0$ and thus
$$
\Gamma_{B_0}\subset cB_R\quad \text{with some } \quad c>0.
$$
One can actually tune constants to guarantee that
$$
\Gamma_{B_0}\subset \bigcup_{Q\in {\sf Tree}(R)}2B_Q\subset 2B_R,
$$
so we will suppose this in what follows.

 Clearly, $\Pi(\Gamma_{B_0})$ is an interval on $L_R$ and therefore $\Gamma_{B_0}$ is a connected subset of $\Gamma_R$. We also set
$$
\Gamma_{\mathbf{Ext}(B_0)}=\Gamma_R\setminus \Gamma_{B_0}.
$$

First we will show that the part of the permutations of $\mathcal{H}^1_{\Gamma_R}$ that involves $\Gamma_{\mathbf{Ext}(B_0)}$ is very small.
\begin{lemma}
\label{paper3-lemma_Ext_B_0}
  We have
  $$
 p_\infty(\Theta_\mu(2B_R)\,\mathcal{H}^1_{\Gamma_{\mathbf{Ext}(B_0)}},\Theta_\mu(2B_R)\,\mathcal{H}^1_{\Gamma_R},\Theta_\mu(2B_R)\,\mathcal{H}^1_{\Gamma_R})\lesssim \sqrt[8]{\varepsilon_0}\,\Theta_\mu(2B_R)^2\,\mu(R).
  $$
\end{lemma}
\begin{proof}
The proof is analogous (up to constants) to the proof of \cite[Lemma~7.36]{Tolsa_book}, where we should use our Lemmas~\ref{paper3-lemma_Lip_A} and~\ref{paper3-lemma_Gamma_R_to_R} instead of \cite[Lemma~7.27 and Lemma~7.32]{Tolsa_book}.
\end{proof}

What is more, it can be easily seen that
\begin{equation}
\label{add_eq1}
\begin{split}
p_\infty(\Theta_\mu(2B_R)\,\mathcal{H}^1_{\Gamma_R})\le & p_\infty(\Theta_\mu(2B_R)\,\mathcal{H}^1_{\Gamma_{B_0}})\\
&+3p_\infty(\Theta_\mu(2B_R)\,\mathcal{H}^1_{\Gamma_{\mathbf{Ext}(B_0)}},\Theta_\mu(2B_R)\,\mathcal{H}^1_{\Gamma_R},\Theta_\mu(2B_R)\,\mathcal{H}^1_{\Gamma_R}).
\end{split}
\end{equation}

Consequently, taking into account Lemmas~\ref{paper3-bs_lemma_1} and \ref{paper3-lemma_Ext_B_0}, we are now able to reduce the proof of Lemma~\ref{paper3-lemma_bs_small} to the proof of a proper estimate for $p_\infty(\Theta_\mu(2B_R)\,\mathcal{H}^1_{\Gamma_{B_0}})$, where $\Gamma_{B_0}\subset cB_R$ with some  $c>0$. For short, we will write
$$
\sigma:=\Theta_\mu(2B_R)\,\mathcal{H}^1_{\Gamma_{B_0}}.
$$
Thus, using this notation, we are aimed to prove the following lemma in the forthcoming subsections.

\begin{lemma}
\label{lemma_add_eq2}
It holds that
\begin{equation*}
\label{add_eq2}
p_\infty(\sigma)\lesssim \varepsilon_0^{1/40}\,\Theta_\mu(2B_R)^2\;\mu(R).
\end{equation*}
\end{lemma}
\subsection{Estimates for the permutations of the Hausdorff measure restricted to $\Gamma_R$ }

Recall that, for $x\in \mathbb{C}$, we set $\ell_x=\ell(J_x)$. Let $x,y\in \Gamma_R$. We say that $x$ and $y$ are
\begin{itemize}
  \item \textbf{very close} and write
$$
(x,y)\in \textbf{VC},\qquad \text{if } |x_1-y_1|\le \ell_x+\ell_y;
$$
  \item \textbf{close} and write
$$
(x,y)\in \textbf{C},\qquad \text{if } |x_1-y_1|\le \varepsilon_0^{-1/20}(\ell_x+\ell_y);
$$
  \item \textbf{far} and write
$$
(x,y)\in \textbf{F},\qquad \text{if } |x_1-y_1|> \varepsilon_0^{-1/20}(\ell_x+\ell_y).
$$
\end{itemize}
Notice that the relations are symmetric with respect to $x$ and $y$.

Given $(x,y,z)\in\Gamma_{B_0}^3$,
there are three possibilities: either two of the points in the triple are very close, or no pair of points is very close but there is at least one pair that is close, or all the pairs of points are far.
So we can split $p_\infty(\sigma)$ as follows:
\begin{equation}
\label{add_eq3}
\begin{split}
p_\infty(\sigma) &\le
3 \iiint_{(x,y)\in\mathbf{VC}}
p_\infty(x,y,z)\,d\sigma(x)\,d\sigma(y)\,d\sigma(z)\\
&\quad +
3
\iiint_{\,
\begin{subarray}{l} (x,y)\in\mathbf{C}\setminus\mathbf{VC} \\(x,z)\not\in\mathbf{VC}\\
(y,z)\not\in\mathbf{VC}
\end{subarray}}
p_\infty(x,y,z)\,d\sigma(x)\,d\sigma(y)\,d\sigma(z)\\
&\quad+
\iiint_{\,
\begin{subarray}{l} (x,y)\in \mathbf{F} \\ (x,z)\in \mathbf{F}\\
(y,z)\in \mathbf{F}
\end{subarray}} \,
p_\infty(x,y,z)\,d\sigma(x)\,d\sigma(y)\,d\sigma(z)\\
&=: p_{\infty,\mathbf{VC}}(\sigma) + p_{\infty,\mathbf{C}\setminus\mathbf{VC}}(\sigma)
+p_{\infty,\mathbf{F}}(\sigma).
\end{split}
\end{equation}

A straightforward adaptation of the arguments from \cite[Section 7.8.2, Lemmas 7.38 and 7.39]{Tolsa_book} to our settings gives the following.
\begin{lemma}
\label{paper3-lemma_34} If $\varepsilon_0=\varepsilon_0(\tau,A)$ and $\alpha=\alpha(\theta_0,\varepsilon_0,\tau,A)$ are chosen small enough, then
$$
p_{\infty,\mathbf{VC}}(\sigma)+p_{\infty,\mathbf{C}\setminus\mathbf{VC}}(\sigma)\lesssim\, \varepsilon_0^{1/40}\,\Theta_\mu(2B_R)^2\;\mu(R).
$$
\end{lemma}

Now we are going to prove the following result that actually finishes the proof of Lemma~\ref{lemma_add_eq2} and therefore Lemma~\ref{paper3-lemma_bs_small} taking into account (\ref{add_eq3}) and Lemma~\ref{paper3-lemma_34}.
\begin{lemma}
\label{paper3-lemma_35}
If $\varepsilon_0=\varepsilon_0(\tau,A)$, $\alpha=\alpha(\theta_0,\varepsilon_0,\tau,A)$ and $\delta=\delta(\varepsilon_0)$ are small enough, then
\begin{equation}
\label{paper3-eqnu**23}
p_{\infty,\mathbf{F}}(\sigma)\lesssim \varepsilon_0^{1/40}\,\Theta_\mu(2B_R)^2\;\mu(R).
\end{equation}
\end{lemma}

The proof of Lemma~\ref{paper3-lemma_35} is similar to the one  of \cite[Lemma 7.40]{Tolsa_book} but necessary changes are not straightforward so we give details. First we need to approximate the measure $\sigma$ by another measure absolutely continuous with respect to $\mu$, of the form $g\mu$, with some $g\in L^\infty(\mu)$. This is carried out by the next lemma, where we say that 
\begin{equation}
\label{Q_i_Not_Far}
i\in I_0'\quad\text{if}\quad i\in I_0 \quad\text{and}\quad
\mu(Q_i\setminus R_{\sf Far})\ge \tfrac{3}{4}\mu(Q_i),
\end{equation}
for the cubes $Q_i\in {\sf DbTree}(R)$  from Lemma~\ref{paper3-lemma_J_1} associated with the intervals $J_i$, $i\in I_0$. Recall the definition of $R_{\sf Far}$ in Section~\ref{section_R_Far} and Lemmas~\ref{paper3-lemma_R_far_1} and~\ref{paper3-lemma_R_far_2}. In what follows we will also write
$$
\widehat{J}_i:= \Gamma_R\cap \Pi^{-1}(J_i).
$$
\begin{lemma}
\label{paper3-lemma_g_i}
  For each $i\in I_0'$ there exists a non-negative function $g_i\in L^\infty(\mu)$, supported on $A_i\subset Q_i\setminus R_{\sf Far}$, where $Q_i\in {\sf DbTree}(R)$ are associated with the intervals $J_i$ by Lemma~\ref{paper3-lemma_J_1}, and such that
  \begin{equation}\label{paper3-lemma_g_i_1}
  \int g_i\,d\mu=\Theta_\mu(2B_R)\mathcal{H}^1(\widehat{J}_i)=\sigma(\widehat{J}_i),
  \end{equation}
  and
  \begin{equation}\label{paper3-lemma_g_i_2}
    \sum\nolimits_{i\in I_0'}g_i\lesssim_{\tau,A} 1.
  \end{equation}
\end{lemma}
\begin{proof}
Assume first that the family $\{J_i\}_{i\in I_0'}$ is finite. Suppose also that $\ell(J_i)\le \ell(J_{i+1})$ for all $i\in I_0'$. We will construct
  $$
  g_i=\alpha_i\chi_{A_i},\qquad \text{where} \quad \alpha_i\ge 0 \quad\text{and} \quad A_i\subset Q_i\setminus R_{\sf Far}.
  $$
   We set
  $$
  \alpha_1=\frac{\sigma(\widehat{J}_1)}{\mu(A_1)}
  \qquad \text{and}\qquad A_1=Q_1\setminus R_{\sf Far},
  $$
so that $\int g_1\,d\mu=\sigma(\widehat{J}_1)$. Furthermore, by (\ref{paper3-theta_Q_Tree}) in Lemma~\ref{paper3-DbTree_properties}, Lemmas~\ref{paper3-lemma_J_1} and~\ref{paper3-lemma_Lip_A} and the condition (\ref{Q_i_Not_Far}) we get
  $$
  \|g_1\|_\infty=\alpha_1\lesssim_{\tau,A} \frac{\,\Theta_\mu(2B_R)\ell(J_1)}{\mu(Q_1)}\approx_{\tau,A} \frac{\,\Theta_\mu(2B_R)\,\diam(Q_1)}{\mu(2B_{Q_1})}\le b'
  $$
  with some $b'=b'(\tau,A)>0$.
Furthermore, we define $g_k$, $k\ge 2$, by induction. Suppose that $g_1,\ldots,g_{k-1}$ have been constructed, satisfy (\ref{paper3-lemma_g_i_1}) and the inequality $\sum_{i=1}^{k-1}g_i\le b$ with some $b=b(\tau,A)>0$ to be chosen later.

  If $Q_k$ is such that $Q_k\cap\bigcup_{i=1}^{k-1}Q_i =\varnothing$, then we set
  $$
  \alpha_k=\frac{\sigma(\widehat{J}_k)}{\mu(A_k)}
  \qquad \text{and}\qquad A_k=Q_k\setminus R_{\sf Far},
  $$
  so that $\int g_k\,d\mu=\sigma(\widehat{J}_k)$. Moreover, similarly to the case of $\alpha_1$, we have
  $$
  \|g_k\|_\infty=\alpha_k\le b',
  $$
  where $b'=b'(\tau,A)$ is obviously independent of $k$.  Since  $A_k\cap\bigcup_{i=1}^{k-1}A_i =\varnothing$, we have
  $$
  g_k + \sum\nolimits_{i=1}^{k-1} g_i \le \max\{b,b'\}.
  $$
  We choose $b=b'(\tau,A)$ in order to have (\ref{paper3-lemma_g_i_2}).

  Now suppose that $Q_k\cap\bigcup_{i=1}^{k-1}Q_i\neq \varnothing$ and let $Q_{s_1},\ldots,Q_{s_m}$ be the subfamily of $Q_1,\ldots,Q_{k-1}$ such that $Q_{s_j}\cap Q_k\neq \varnothing$. Since $\ell(J_{s_j})\le
\ell(J_k)$ (because of the non-decreasing sizes of $\ell(J_i)$, $i\in I_0'$), we deduce that
 $\dist(J_{s_j},J_k)\lesssim\ell(J_k)$, and thus
 $J_{s_j}\subset c'J_k$, for some  constant $c'>0$. Using
(\ref{paper3-lemma_g_i_1}) for $i=s_j$, we get by (\ref{paper3-theta_Q_Tree}) in Lemma~\ref{paper3-DbTree_properties}, Lemmas~\ref{paper3-lemma_J_1} and~\ref{paper3-lemma_Lip_A} that
\begin{align*}
\sum\nolimits_j \int g_{s_j}\,d\mu & =  \sum\nolimits_j
\sigma(\widehat{J}_{s_j})  \le  \sigma(\Pi^{-1}(c'J_k))\\
& \lesssim \Theta_\mu(2B_R)
 \ell(J_k)\lesssim \Theta_\mu(2B_R)\diam(Q_k) \le c''\,\mu(Q_k)
\end{align*}
with some $c''=c''(\tau,A)>0$. Therefore, by Chebyshev's inequality,
$$
\mu\left(\left\{\sum\nolimits_j g_{s_j} >
2c''\right\}\right) \le \frac{1}{2}\, \mu(Q_k).
$$
So we set
$$
A_k = \left(Q_k\cap \left\{\sum\nolimits_j g_{s_j} \le
2c''\right\}\right)\setminus R_{\sf Far},
$$
and then $\mu(A_k) \ge \tfrac{1}{4}\mu(Q_k)$.
As above, we put $\alpha_k=\sigma(\widehat{J}_k)/\mu(A_k)$ so that $g_k=\alpha_k\chi_{A_k}$ satisfies $\int
g_k\,d\mu = \sigma(\widehat{J}_k)$.
Consequently,
\begin{equation*}
\alpha_k \le  \frac{\sigma(\widehat{J}_k)}{\tfrac{1}{4}\mu(Q_k)}
\le b''\qquad \text{with some } b''=b''(\tau,A)>0,
\end{equation*}
which yields
$$
g_k + \sum\nolimits_j g_{s_j} \le b'' +
2c''.
$$
Recall that $s_j$ are such that $Q_{s_j}\cap Q_k\neq \varnothing$. The latter inequality implies that
$$
  g_k + \sum\nolimits_{i=1}^{k-1} g_i \le \max\{b,b''+2c''\}.
  $$
In this case, we choose $b=b''+2c''$ and
(\ref{paper3-lemma_g_i_2}) follows. Clearly, this bound is independent of the number of functions.

Suppose now that $\{J_i\}_{i\in I_0'}$ is not finite. For each fixed $M$ we
consider a family of intervals $\{J_i\}_{1\le i\le M}$. As above,
we construct functions $g_1^M,\ldots,g_M^M$ with
$\spt(g_i^M)\subset Q_i\setminus R_{\sf Far}$ satisfying
$$
\int g^M_i d\mu = \sigma(\widehat{J}_i) \quad \text{and}\quad
\sum\nolimits_{i=1}^M
g_i^M \le b=b(\tau,A).
$$
Then there is a subsequence $\{g_1^k\}_{k\in I_1}$ which is convergent in
the weak $*$ topology of $L^\infty(\mu)$ to some function $g_1\in
L^\infty(\mu)$. Now we take another convergent subsequence $\{g_2^k\}_{k\in
I_2}$, $I_2\subset I_1$, in the weak $*$ topology of
$L^\infty(\mu)$ to another function $g_2\in L^\infty(\mu)$, etc.
We have $\spt(g_i)\subset Q_i\setminus R_{\sf Far}$. Furthermore, (\ref{paper3-lemma_g_i_1}) and (\ref{paper3-lemma_g_i_2})
 also hold due to the weak $*$ convergence.
\end{proof}

Recall that $G_R=\{z\in \mathbb{C}:d(z)=0\}$ (see (\ref{paper3-good_R})) and clearly $G_R\subset R$. We will need the following result which can be proved analogously to \cite[Lemma 7.18]{Tolsa_book} taking into account that the density $\Theta_\mu(2B_R)$ is involved in our case.
\begin{lemma}
\label{paper3-lemma_mu_G_R}
We have
$$
\mu\lfloor G_R=\rho_{G_R}\Theta_\mu(2B_R)\,\mathcal{H}^1_{G_R}=\rho_{G_R}\sigma\lfloor G_R,
$$
where $\rho_{G_R}$ is a function such that $c\le \rho_{G_R}\le c^{-1}$ with some constant $c=c(\tau,A)>0$.
\end{lemma}

Let us mention now the following technical result proved in \cite[Subsection 4.6.1]{Tolsa_book}.
\begin{lemma} \label{paper3-lemcompar**}
Let $x,y,z\in \mathbb{C}$ be pairwise distinct points, and let
$x'\in \mathbb{C}$ be such that
$$
a^{-1} |x-y| \le |x'-y| \le a |x-y|,
$$
where $a>0$ is some constant. Then
$$
|c(x,y,z) - c(x',y,z)| \le (4+2a) {\frac{|x-x'|}{|x-y| |x-z|}}.
$$
\end{lemma}
Take into account that $p_\infty(x,y,z)=\frac{1}{2}c(x,y,z)^2$ by (\ref{curvature_pointwise}).

Recall that
$$
\Gamma_{B_0}= G_R\cup \bigcup\nolimits_{i\in I_0}\widehat{J}_i\qquad\text{and}\qquad \widehat{J}_i= \Gamma_R\cap \Pi^{-1}(J_i).
$$
In Lemma~\ref{paper3-lemma_g_i} we showed how $\sigma\lfloor \widehat{J}_i$ can be approximated by a measure supported on $Q_i\setminus R_{\sf Far}$, for each $i\in I_0'$, where $I_0'$ is defined in (\ref{Q_i_Not_Far}). Notice that, by  Lemma~\ref{paper3-lemma_Q_i_proj_of_J_i},
\begin{equation}
\label{paper3-eqd5198}
\dist(Q_i, \widehat{J}_i)\lesssim_{\tau,A} \ell(J_i),\qquad i\in I_0.
\end{equation}

Now we consider the measures
$$
\nu_i:=g_i\,\mu
, \qquad i\in I_0',
$$
 with $g_i$ as in Lemma~\ref{paper3-lemma_g_i}, and set
\begin{equation}
\label{measure_nu}
\nu: =\sigma\lfloor G_R + \sum\nolimits_{i\in I_0'} \nu_i= \rho_{G_R}^{-1}\,\mu\lfloor G_R + \sum\nolimits_{i\in I_0'} g_i\,\mu.
\end{equation}
This measure should be understood as an approximation of $\sigma=\Theta_\mu(2B_R)\mathcal{H}^1_{\Gamma_{B_0}}$, which
coincides with $\sigma$ on $G_R$ due to Lemma~\ref{paper3-lemma_mu_G_R} ($g_i\equiv 0$ in this case).

Using the measure $\nu$,  we will actually prove the inequality (\ref{paper3-eqnu**23}) in Lemma~\ref{paper3-lemma_35}. This will be done in the forthcoming subsection.

\subsection{Estimates for the permutations of the Hausdorff measure restricted to $\Gamma_R$ in the case when points are far from each other}

To proceed, we need to introduce some additional notation.
Given measures $\tau_1,\tau_2,\tau_3$,  set
$$
p_t(\tau_1,\tau_2,\tau_3):= \iiint p_t(x,y,z)\,d\tau_1(x)\,d\tau_2(y)\,d\tau_3(z),\qquad \text{where }t=0\text{ or }t=\infty.
$$
We denote by $p_{t, \mathbf{F}}(\tau_1,\tau_2,\tau_3)$ the triple integral above restricted to $(x,y,z)$ such that
\begin{equation}
\label{paper3-eqfar**4}
\begin{array}{l}
  |x_1-y_1|\ge \varepsilon_0^{-1/20}(\ell_x+\ell_y), \\
  |x_1-z_1|\ge \varepsilon_0^{-1/20} (\ell_x+\ell_z), \\
  |y_1-z_1|\ge\varepsilon_0^{-1/20}
(\ell_y+\ell_z).
\end{array}
\end{equation}
So we have
\begin{equation}
\label{paper3-eqfi**4}
\begin{array}{rl}
p_{\infty, \mathbf{F}}(\sigma) = & \phantom{+} p_{\infty, \mathbf{F}}(\sigma\lfloor {G_R}) \\
  & +p_{\infty, \mathbf{F}}(\sigma\lfloor \Gamma_{B_0}\setminus {G_R}) \\
  & +3\,p_{\infty, \mathbf{F}}(\sigma\lfloor {G_R},\sigma\lfloor \Gamma_{B_0}\setminus {G_R},\sigma\lfloor \Gamma_{B_0}\setminus {G_R}) \\
  & +3\,p_{\infty, \mathbf{F}}(\sigma\lfloor {G_R},\sigma\lfloor {G_R},\sigma\lfloor \Gamma_{B_0}\setminus {G_R}).
\end{array}
\end{equation}

\bigskip

\textbf{1.} Consider the term $p_{\infty, \mathbf{F}}(\sigma\lfloor {G_R})$. In this case $\ell_x=\ell_y=\ell_z \equiv 0$ and the subscript $\mathbf{F}$ may be skipped. Moreover, using Lemmas~\ref{paper3-lemma_compar_p_0_p_infty} and~\ref{paper3-lemma_mu_G_R}, we get
$$
p_{\infty, \mathbf{F}}(\sigma\lfloor {G_R})
\lesssim_{\theta_0} p_{0}(\sigma\lfloor {G_R})
\approx_{\theta_0,\tau,A} p_{0}(\mu\lfloor {G_R}).
$$

Now we proceed very similarly to the proof of Lemma~\ref{paper3-lemma_R_far_1}. 
For $\delta>0$ from Lemma~\ref{paper3-lemma_beta_perm_initial}  (see also Section~\ref{paper3-sec_param}), taking into account Remark~\ref{remark_BS_R},  we get
\begin{align*}
p_{0}(\mu\lfloor {G_R})
&\le \int_{G_R}\sum_{Q\in {\sf Tree}(R)\setminus {\sf Stop}(R):\,x\in 2B_Q}p^{[\delta,Q]}_0(x,\mu\lfloor {G_R},\mu\lfloor {G_R}) \;d\mu(x)\\
&\le \sum_{Q\in {\sf Tree}(R)\setminus {\sf Stop}(R)}p^{[\delta,Q]}_0(\mu\lfloor 2B_Q,\mu\lfloor 2B_R,\mu\lfloor 2B_R) \\
&= \sum_{Q\in
{\sf Tree}(R)\setminus {\sf Stop}(R)}\frac{p^{[\delta,Q]}_0(\mu\lfloor 2
B_Q,\mu\lfloor 2B_R,\mu\lfloor 2B_R)}{\mu(Q)}\int \chi_Q(x)\,d\mu(x).
\end{align*}
Changing the order of summation and the inequality (\ref{paper3-big_perm_prop}) yield
\begin{align*}
\frac{p_{0}(\mu\lfloor {G_R})}{\Theta_\mu(2B_R)^2} &\le \int_R
\sum_{Q\in
{\sf Tree}(R)\setminus {\sf Stop}(R):\,x\in Q}\frac{p^{[\delta,Q]}_0(\mu\lfloor 2B_Q,\mu\lfloor 2B_R,\mu\lfloor 2B_R)}{\Theta_\mu(2B_R)^2\mu(Q)}\,d\mu(x)\\
&= \int_R
\sum_{Q\in
{\sf Tree}(R)\setminus {\sf Stop}(R):\,x\in Q}\textsf{perm}(Q)^2\,d\mu(x).
\end{align*}
From this and the inequality (\ref{paper3-big_perm_prop}) in Lemma~\ref{paper3-DbTree_properties} we deduce that
$$
p_{0}(\mu\lfloor {G_R})
\le \alpha^2\,\Theta_\mu(2B_R)^2  \mu(R).
$$
Finally, if $\alpha=\alpha(\theta_0,\varepsilon_0,\tau,A)$ is chosen small enough, then
\begin{equation*}\label{paper3-p_infty_F_sigma_1}
p_{\infty, \mathbf{F}}(\sigma\lfloor {G_R})
\lesssim \varepsilon_0^{1/40} \,\Theta_\mu(2B_R)^2\,\mu(R).
\end{equation*}

\bigskip

\textbf{2.} Let us study $p_{\infty, \mathbf{F}}(\sigma\lfloor \Gamma_{B_0}\setminus {G_R})$. In this case $\ell_x$, $\ell_y$ and $\ell_z$ are strictly positive and so are the lengths of the associated doubling cubes from Lemma~\ref{paper3-lemma_J_1}. We set
$$
p_{\infty, \mathbf{F}}(\sigma\lfloor \Gamma_{B_0}\setminus {G_R}) = \sum_{i,j,k\in I_0}
p_{\infty, \mathbf{F}}\left(\sigma\lfloor \widehat{J}_i, \sigma\lfloor\widehat{J}_j,\sigma\lfloor \widehat{J}_k\right).
$$

First let us consider the case when at least one of the indices $i$, $j$ or $k$ is in $I_0\setminus I_0'$, i.e.
$\mu(Q_h\cap R_{\sf Far})> \tfrac{3}{4}\mu(Q_h)$ 
for $h$ being $i$, $j$ or $k$,  according to (\ref{Q_i_Not_Far}). By symmetry, we may consider just the case $i\in I_0\setminus I_0'$. Moreover, then the required estimate will follow from a proper one for
$$
p_{\infty}\left(\sigma\lfloor \widehat{J'}, \sigma\lfloor\Gamma_{B_0},\sigma\lfloor\Gamma_{B_0}\right),\qquad\text{where } \widehat{J'}:=\bigcup_{i\in I_0\setminus I_0'}\widehat{J}_i.
$$

Recall that 
$$
\sigma=\Theta_\mu(2B_R)\mathcal{H}^1_{\Gamma_{B_0}}\quad\text{and}\quad \Gamma_{B_0}= G_R\cup \bigcup\nolimits_{i\in I_0}\widehat{J}_i.
$$
\begin{lemma}
\label{new1}
We have
$$
\mathcal{H}^1(\widehat{J'})\le \sqrt \alpha\,\diam(R).
$$
\end{lemma}
\begin{proof}
Notice that for $i\in I_0\setminus I_0'$ we have
\begin{align*}
\Theta_\mu(2B_R)\mathcal{H}^1(\wh J_i)&\lesssim
\Theta_\mu(2B_R)\ell(J_i)\lesssim_{\tau,A}\Theta_\mu(2B_R)\diam(Q_i)\\
&\approx_{\tau,A}
\mu(Q_i)\lesssim_{\tau,A}\mu(Q_i\cap R_{\sf Far}),
\end{align*}
where $Q_i\in {\sf DbTree}(R)$ is the cube associated to the interval $J_i$ by Lemma 3.21.
By Vitali's covering lemma, there exists a subfamily of balls $2B_{Q_i}$, $i\in J\subset I_0\setminus I_0'$, such that 
\begin{itemize}
\item the balls $2B_{Q_i}$, $i\in J$, are disjoint,
\item $\bigcup_{i\in I_0\setminus I_0'} 2B_{Q_i}\subset \bigcup_{i\in J} 10B_{Q_i}.$
\end{itemize}
Then, taking into account that $\mu(10B_{Q_i}\cap R)\approx_{\tau,A}  \mu(2B_{Q_i})\approx \mu(Q_i)$, we get
\begin{align*}
\Theta_\mu(2B_R)\mathcal{H}^1(\widehat{J'})& \lesssim \sum_{i\in I_0\setminus I_0'}
\mu(Q_i)
\lesssim \sum_{i\in J}
\mu(10 B_{Q_i}\cap R)\\
&\lesssim_{\tau,A} \sum_{i\in J}
\mu(Q_i)\lesssim _{\tau,A}
\sum_{i\in J}\mu(Q_i\cap R_{\sf Far})\lesssim_{\tau,A} \mu(R_{\sf Far}),
\end{align*}
because the cubes $Q_i$ from the family $J$ are disjoint.
Since $\mu(R_{\sf Far})\le \alpha \mu(R)$ by Lemma~\ref{paper3-lemma_R_far_1}, the lemma follows if $\alpha=\alpha(\tau,A)$ is chosen small enough.
\end{proof}
To continue, we need the following result from \cite{Tolsa_book}.
\begin{lemma}[Lemma 3.4 in \cite{Tolsa_book}]
\label{new2}
Let $\mu_1$, $\mu_2$ and $\mu_3$ be finite measures. Then
$$
\sum_{s\in \mathfrak{S}_3} \int \mathcal{C}_{\varepsilon}(\mu_{s_2})\overline{\mathcal{C}_{\varepsilon}(\mu_{s_3})}\,d\mu_{s_1}=c^2_\varepsilon(\mu_1,\mu_2,\mu_3)+\mathcal{R},\quad \mathcal{R}\le C\, 
\sum_{s\in \mathfrak{S}_3}\int M_{\sf R}\mu_{s_2} M_{\sf R} \mu_{s_3}\,d\mu_{s_1},
$$
where $\mathfrak{S}_3$ is the group of permutations of the three
elements $\{1,2,3\}$, $\mathcal{C}_\varepsilon$ the truncated Cauchy integral, $c^2_\varepsilon$ the truncated curvature of measure $($see $(\ref{p_K(mu)})$ and below$)$ and
 $M_{\sf R}$ the
$1$-dimensional radial maximal operator.
\end{lemma}
\begin{lemma}
\label{new3}
For $E\subset \Gamma_{B_0}$, we have
$$
c^2(\mathcal{H}^1_E,\mathcal{H}^1_{\Gamma_{B_0}},\mathcal{H}^1_{\Gamma_{B_0}})\lesssim \mathcal{H}^1(E)^{1/2}\,\diam(R)^{1/2}.
$$
\end{lemma}
\begin{proof}
By Lemma~\ref{new2}, we have
\begin{align*}
c^2(\mathcal{H}^1_E,\mathcal{H}^1_{ \Gamma_{B_0}},\mathcal{H}^1_{\Gamma_{B_0}}) & \lesssim \limsup_{\varepsilon\to0}
\int_{\Gamma_{B_0}}|\mathcal{C}_\varepsilon( \mathcal{H}^1_E)\,\mathcal{C}_\varepsilon (\mathcal{H}^1_{\Gamma_{B_0}})|\,d\mathcal{H}^1\\
& \quad + \limsup_{\varepsilon \to 0}\int_{E}|\mathcal{C}_\varepsilon (\mathcal{H}^1_{\Gamma_{B_0}})|^2\,d\mathcal{H}^1 \\
& \quad + \int_{ \Gamma_{B_0}} |M_{\sf R}(\mathcal{H}^1_E)\,M_{\sf R}(\mathcal{H}^1_{\Gamma_{B_0}})|^2\,d\mathcal{H}^1 \\
& \quad + \int_{E} |M_{\sf R}(\mathcal{H}^1_{\Gamma_{B_0}})|^2\,d\mathcal{H}^1\\
& := I_1+I_2+I_3+I_4.
\end{align*}

Regarding $I_1$, by the $L^2$-boundedness of the Cauchy transform on Lipschitz graphs (with respect to
$\mathcal{H}^1_{\Gamma_R}$) we have
\begin{equation*}
\begin{split}
I_1&\le \limsup_{\varepsilon\to0}
\|\mathcal{C}_\varepsilon( \mathcal{H}^1_E)\|_{L^2(\mathcal{H}^1_{\Gamma_R})}\|\mathcal{C}_\varepsilon( \mathcal{H}^1_{\Gamma_{B_0}})\|_{L^2(\mathcal{H}^1_{\Gamma_R})}\\
&\lesssim \mathcal{H}^1(E)^{1/2}\,\mathcal{H}^1( \Gamma_{B_0})^{1/2}\\
&\lesssim \mathcal{H}^1(E)^{1/2}\,\diam(R)^{1/2}.
\end{split}
\end{equation*}
For $I_2$ we use the $L^4$-boundedness of the Cauchy transform:
$$
I_2 \le \limsup_{\varepsilon \to 0}
\mathcal{H}^1(E)^{1/2}\,\|\mathcal{C}_\varepsilon(\mathcal{H}^1_{\Gamma_{B_0}})\|_{L^4(\mathcal{H}^1_{\Gamma_R})}^{2}\lesssim \mathcal{H}^1(E)^{1/2}\,\diam(R)^{1/2}.
$$

Using the fact that
$M_{\sf R}(\mathcal{H}^1_{\Gamma_{B_0}})\lesssim1$, we derive 
$$
I_4\le \mathcal{H}^1(E)\lesssim \mathcal{H}^1(E)^{1/2}\,\diam(R)^{1/2},
$$
and also
$$I_3\lesssim \int_{\Gamma_{B_0}} |M_{\sf R}(\mathcal{H}^1_E)|\,d\mathcal{H}^1.$$
Since the operator $M_{\sf R}( \mathcal{H}^1_{\Gamma_R})$ is bounded in $L^2(\mathcal{H}^1_{\Gamma_R})$
(as it is comparable to the Hardy-Littlewood operator with respect to the measure $\mathcal{H}^1_{\Gamma_R}$),
we deduce
$$I_3\lesssim \|M_{\sf R}(\chi_E\mathcal{H}^1_{\Gamma_R})\|_{L^2(\mathcal{H}^1_{\Gamma_R})}\mathcal{H}^1(\Gamma_{B_0})^{1/2}
\lesssim \mathcal{H}^1(E)^{1/2}\,\diam(R)^{1/2}.$$
So the lemma follows.
\end{proof}

By Lemma~\ref{new3} for 
$E =\wh J'$
and Lemma~\ref{new1} we derive that
$$
c^2(\mathcal{H}^1_{\wh J'},\mathcal{H}^1_{\Gamma_{B_0}},\mathcal{H}^1_{\Gamma_{B_0}})\lesssim\mathcal{H}^1(\wh J')^{1/2}\,\diam(R)^{1/2}
\lesssim  \alpha^{1/4}\,\diam(R).
$$
Therefore, recalling that $p_\infty(x,y,z)=\frac{1}{2}c(x,y,z)^2$ (see  (\ref{curvature_pointwise})),
$$
p_{\infty}\left(\sigma\lfloor \widehat{J'}, \sigma\lfloor\Gamma_{B_0},\sigma\lfloor\Gamma_{B_0}\right)\lesssim \alpha^{1/4}\Theta_\mu(2B_R)^3  \diam(R)
\approx \alpha^{1/4}\,\Theta_\mu(2B_R)^2 \,\mu(R).
$$
Furthermore, choosing  $\alpha=\alpha(\varepsilon_0)$ small enough, we get from the latter estimate that
\begin{equation}
\label{eqq}
\sum_{i\in I_0\setminus I_0',\,j,k\in I_0}
p_{\infty, \mathbf{F}}\left(\sigma\lfloor \widehat{J}_i, \sigma\lfloor\widehat{J}_j,\sigma\lfloor \widehat{J}_k\right)\lesssim \varepsilon_0^{1/40} \,\Theta_\mu(2B_R)^2\,\mu(R),
\end{equation}
and we are done with the case when at least one of the indices $i$, $j$ or $k$ is in $I_0\setminus I_0'$.

\medskip

Now let $(i,j,k)\in (I_0')^3$. By definition, if $p_{\infty, \mathbf{F}}\left(\sigma\lfloor\widehat{J}_i, \sigma\lfloor\widehat{J}_j,\sigma\lfloor \widehat{J}_k\right)\neq 0$, then there exist $x\in \widehat{J}_i$, $y\in \widehat{J}_j$ and $z\in \widehat{J}_k$ satisfying
(\ref{paper3-eqfar**4}). Then it follows easily that
\begin{equation}
\label{paper3-eqfar41630}
\begin{array}{l}
  \dist(\widehat{J}_i,\widehat{J}_j)\ge \frac{1}{2}\varepsilon_0^{-1/20}\left(\ell(J_i)+ \ell(J_j)\right), \\
  \dist(\widehat{J}_i,\widehat{J}_k)\ge \frac{1}{2}\varepsilon_0^{-1/20}\left(\ell(J_i)+ \ell(J_k)\right), \\
  \dist(\widehat{J}_j,\widehat{J}_k)\ge \frac{1}{2}\varepsilon_0^{-1/20}\left(\ell(J_j)+ \ell(J_k)\right).
\end{array}
\end{equation}
We denote by $J_{\bf F}$ the set of those indices $(i,j,k)\in (I_0')^3$ such that the  inequalities (\ref{paper3-eqfar41630}) hold, so that
$$
p_{\infty, \mathbf{F}}(\sigma\lfloor \Gamma_{B_0}\setminus {G_R}) \le  \sum_{(i,j,k)\in J_{\bf F}}
p_\infty\left(\sigma\lfloor \widehat{J}_i, \sigma\lfloor \widehat{J}_j,\sigma\lfloor \widehat{J}_k\right).
$$

Consider $(i,j,k)\in J_{\bf F}$ and
$$
x,x'\in \widehat{J}_i\cup Q_i, \qquad
y,y'\in \widehat{J}_j\cup Q_j \qquad \text{and} \qquad z,z'\in \widehat{J}_k\cup Q_k.
$$
Due to (\ref{paper3-eqfar41630}) and (\ref{paper3-eqd5198}),
taking into account that $\ell(J_h)\approx_{\tau,A} \diam(\widehat{J}_h)\approx_{\tau,A}\diam(Q_h)$ for each $h\in I$ by Lemma~\ref{paper3-lemma_J_1},
the sets
$\widehat{J}_i\cup Q_i$, $\widehat{J}_j\cup Q_j$ and $\widehat{J}_k\cup Q_k$ are far to each other in the sense that
\begin{equation}
\label{paper3-eqfar}
\begin{array}{l}
  \dist(\widehat{J}_i\cup Q_i,\widehat{J}_j\cup Q_j)\gtrsim \varepsilon_0^{-1/20}\left(\ell(J_i)+ \ell(J_j)\right), \\
  \dist(\widehat{J}_i\cup Q_i,\widehat{J}_k\cup Q_k) \gtrsim \varepsilon_0^{-1/20}\left(\ell(J_i)+ \ell(J_k)\right), \\
  \dist(\widehat{J}_j\cup Q_j,\widehat{J}_k\cup Q_k)\gtrsim \varepsilon_0^{-1/20}\left(\ell(J_j)+ \ell(J_k)\right),
\end{array}
\end{equation}
where $\varepsilon_0$ is chosen small enough. Furthermore, applying Lemma~\ref{paper3-lemcompar**} three times gives
\begin{align*}
p_\infty(x,y,z) & \le 2\,p_\infty(x',y',z')+c\left(T_x(y,z) +T_y(x,z)+T_z(x,y)\right),
\end{align*}
where
$$
T_{z_1}(z_2,z_3):=\frac{\ell_{z_1}^2}{|z_1-z_2|^2\,|z_1-z_3|^2}\qquad\text{for}\quad z_1,z_2,z_3 \in \mathbb{C}.
$$
Then, integrating on $x\in \widehat{J}_i$, $y\in \widehat{J}_j$, and $z\in \widehat{J}_k$ with respect to $\sigma$, we get
\begin{align*}
&p_\infty\left(\sigma\lfloor \widehat{J}_i,\sigma\lfloor\widehat{J}_j,\sigma\lfloor \widehat{J}_k\right)  \le 2\,p_\infty(x',y',z') \,\sigma(\widehat{J}_i)\,
\sigma(\widehat{J}_j)\,\sigma(\widehat{J}_k)
\\
&\qquad+
c\iiint_{\,
\begin{subarray}{l} x\in \widehat{J}_i\\y\in \widehat{J}_j\\ z\in \widehat{J}_k\end{subarray}}
\left[T_x(y,z) +T_y(x,z)+T_z(x,y)\right]\,d\sigma(x)\,d\sigma(y)\,d\sigma(z).
\end{align*}
On the other hand, by analogous arguments, we have
\begin{align*}
p_\infty(x',y',z') \,\|\nu_i\|& \,
\|\nu_j\|\,\|\nu_k\| \le 2\,p_\infty(\nu_i,\nu_j,\nu_k)
\\
&+ c\iiint
\left[
T_x(y,z) +T_y(x,z)+T_z(x,y)\right]\,d\nu_i(x)\,d\nu_j(y)\,d\nu_k(z).
\end{align*}
Thus, recalling that $\|\nu_h\|=
\sigma(\widehat{J}_h)$ for any $h\in I_0'$, from the preceding  inequalities we get
\begin{equation}
\label{add_eq4}
\begin{split}
p_\infty&\left(\sigma\lfloor \widehat{J}_i,\sigma\lfloor\widehat{J}_j,\sigma\lfloor \widehat{J}_k\right)\lesssim\,\,p_\infty(\nu_i,\nu_j,\nu_k)
\\
&+ \iiint
\left[
T_x(y,z) +T_y(x,z)+T_z(x,y)\right]\,d\nu_i(x)\,d\nu_j(y)\,d\nu_k(z)\\
&\qquad+
\iiint_{\,
\begin{subarray}{l} x\in \widehat{J}_i\\y\in \widehat{J}_j\\ z\in \widehat{J}_k\end{subarray}}
\left[T_x(y,z) +T_y(x,z)+T_z(x,y)\right]\,d\sigma(x)\,d\sigma(y)\,d\sigma(z).
\end{split}
\end{equation}

Now recall that $A_h=\spt \nu_h\subset Q_h$ for any $h\in I_0'$. This and  Lemma~\ref{paper3-lemma_gamma_R_1} imply that for each $x\in Q_i$ and $y\in Q_j$  there exist $\tilde{x}\in\Gamma_R$ and $\tilde{y}\in\Gamma_R$, correspondingly, such that $\dist(x,\tilde{x})\lesssim_{\tau,A} \ell(J_i)$ and $\dist(y,\tilde{y})\lesssim_{\tau,A}  \ell(J_j)$. Due to this fact and (\ref{paper3-eqfar}), it holds that
$$
\measuredangle (L_{x\tilde{y}},L_{xy})\lesssim \frac{|y-\tilde{y}|}{|x-y|}\lesssim_{\tau,A}  \frac{\ell(J_j)}{\varepsilon_0^{-1/20}\left(\ell(J_i)+ \ell(J_j)\right)}\lesssim_{\tau,A}  \varepsilon_0^{1/20}
$$
and
$$
\measuredangle (L_{\tilde{x}y},L_{xy})\lesssim   \frac{|x-\tilde{x}|}{|x-y|}\lesssim_{\tau,A}  \frac{\ell(J_i)}{\varepsilon_0^{-1/20}\left(\ell(J_i)+ \ell(J_j)\right)}\lesssim_{\tau,A}  \varepsilon_0^{1/20}.
$$
So it follows that $\measuredangle (L_{\tilde{x}\tilde{y}},L_{xy})\lesssim_{\tau,A} \varepsilon_0^{1/20}$.
By Lemma~\ref{paper3-lemma_Lip_A} and the definitions at the beginning of Subsection~\ref{paper3-sec_small_BS_1},
$$
\measuredangle(L_{xy},L_R)\le C_F\,\theta_0\qquad \text{ and } \qquad \theta_V(L_R)\ge (1+C_F)\,\theta_0.
$$
Consequently,
$$
\theta_V(L_{\tilde{x}\tilde{y}})\ge \theta_V(L_R)-\measuredangle(L_{xy},L_R)-\measuredangle (L_{\tilde{x}\tilde{y}},L_{xy})\ge \tfrac{1}{2}\theta_0,
$$
if $\varepsilon_0=\varepsilon_0(\theta_0,{\tau,A} )$ is chosen small enough. Now use Lemma~\ref{paper3-lemma_Chous_Prat} to conclude that
\begin{equation}
\label{add_eq5}
p_\infty(\nu_i,\nu_j,\nu_k)\lesssim_{\theta_0}p_0(\nu_i,\nu_j,\nu_k).
\end{equation}
Moreover, from (\ref{paper3-eqfar}) and the fact that $\ell(J_h)\approx_{\tau,A} \diam(Q_h)$ for any $h$ we conclude that
\begin{align*}
p_0(\nu_i,\nu_j,\nu_k)&\lesssim  \int\sum_{Q\in {\sf Tree}(R)\setminus {\sf Stop}(R):\,x\in 2B_Q}p^{[\delta,Q]}_0(x,\nu_j,\nu_k) \;d\nu_i(x)\\
&\lesssim  \sum_{Q\in {\sf Tree}(R)\setminus {\sf Stop}(R)}p^{[\delta,Q]}_0(\nu_i\lfloor 2B_Q,\nu_j,\nu_k),
\end{align*}
where $\delta=\delta(\varepsilon_0,\tau,A)$ is chosen small enough. Furthermore, using that $\nu=g\mu$ and arguing as in the case of $p_{\infty, \mathbf{F}}(\sigma\lfloor {G_R})$ we get
\begin{align*}
\sum_{(i,j,k)\in J_{\bf F}}p_0(\nu_i,\nu_j,\nu_k)
&\lesssim  \sum_{Q\in {\sf Tree}(R)\setminus {\sf Stop}(R)}p^{[\delta,Q]}_0(\nu\lfloor 2B_Q,\nu,\nu)\\
&\lesssim_{\tau,A}   \sum_{Q\in {\sf Tree}(R)\setminus {\sf Stop}(R)}p^{[\delta,Q]}_0(\mu\lfloor 2
B_Q,\mu\lfloor 2B_R,\mu\lfloor 2B_R)\\
& \lesssim_{\tau,A} \alpha^2\Theta_\mu(2B_R)^2\mu(R)\\
&\lesssim \varepsilon_0^{1/20}\Theta_\mu(2B_R)^2\mu(R),
\end{align*}
where $\alpha=\alpha(\theta_0,\varepsilon_0,\tau,A)$ is chosen small enough.
From this, (\ref{add_eq4}) and (\ref{add_eq5}) by summing on $(i,j,k)\in J_{\bf F}$ we deduce that
\begin{equation}
\label{paper3-eq**9212}
\begin{split}
&\sum_{(i,j,k)\in J_{\bf F}}p_\infty\left(\sigma\lfloor \widehat{J}_i,\sigma\lfloor\widehat{J}_j,\sigma\lfloor \widehat{J}_k\right)\\
&\quad \lesssim \varepsilon_0^{1/40}\Theta_\mu(2B_R)^2\mu(R)  \\
&
\qquad + \!\iiint_{\,
\begin{subarray}{l} |x-y|\ge \frac{1}{2}\varepsilon_0^{-1/20}(\ell_x+\ell_y) \\
|x-z|\ge \frac{1}{2}\varepsilon_0^{-1/20}(\ell_x+\ell_z)\\
|y-z|\ge \frac{1}{2}\varepsilon_0^{-1/20}(\ell_y+\ell_z)\end{subarray}}\left[
T_x(y,z) +T_y(x,z)+T_z(x,y)\right]\,d\sigma(x)\,d\sigma(y)\,d\sigma(z)\\
&
\qquad+ \!\iiint_{\,
\begin{subarray}{l} |x-y|\ge \frac{1}{2}\varepsilon_0^{-1/20}(\ell_x+\ell_y) \\
|x-z|\ge \frac{1}{2}\varepsilon_0^{-1/20}(\ell_x+\ell_z)\\
|y-z|\ge \frac{1}{2}\varepsilon_0^{-1/20}(\ell_y+\ell_z)\end{subarray}}\left[
T_x(y,z) +T_y(x,z)+T_z(x,y)\right]\,d\nu(x)\,d\nu(y)\,d\nu(z),
\end{split}
\end{equation}
where $\varepsilon_0=\varepsilon_0(\theta_0)$ was chosen small enough.
Recall the definition of $\nu$ in (\ref{measure_nu}).

To estimate the first triple integral in the right side of (\ref{paper3-eq**9212}), notice that
\begin{equation}
\label{paper3-eq**de3}
\begin{split}
&\iint_{\,
\begin{subarray}{l} |x-y|\ge \frac{1}{2}\varepsilon_0^{-1/20}(\ell_x+\ell_y) \\
|x-z|\ge \frac{1}{2}\varepsilon_0^{-1/20}(\ell_x+\ell_z)\\
\end{subarray}}
T_x(y,z)\,d\sigma(y)\,d\sigma(z)\\
&\quad \le \left(\int_{|x-y|\ge \frac{1}{2}\varepsilon_0^{-1/20}\ell_x}
 \frac{\ell_x}{|x-y|^2}\,d\sigma(y)\right)
\left(\int_{|x-z|\ge \frac{1}{2}\varepsilon_0^{-1/20}\ell_x}
 \frac{\ell_x}{|x-z|^2}\,d\sigma(z)\right)\\
 &\quad= \left(\int_{|x-y|\ge \frac{1}{2}\varepsilon_0^{-1/20}\ell_x}
 \frac{\ell_x}{|x-y|^2}\,d\sigma(y)\right)^2 \lesssim \varepsilon_0^{1/10}\,\Theta_\mu(2B_R)^2,
 \end{split}
\end{equation}
where the  last inequality follows from splitting the domain $\{y:|x-y|\ge \frac{1}{2}\varepsilon_0^{-1/20}\ell_x\}$ into annuli and the linear growth of $\sigma$ with constant $\lesssim\Theta_\mu(2B_R)$ (see (\ref{linear_growth})). Analogous estimates hold permuting $x,y,z$, and
also interchanging $\sigma$ by $\nu$ (the implicit constant in the analogue of (\ref{paper3-eq**de3}) for $\nu$ depends on $\tau$ and $A$ then). Indeed, this is a consequence of the following result. 
\begin{lemma}
It holds that
$$
\nu(B(x,r))\lesssim_{\tau,A} \Theta_\mu(2B_R)\,r,\quad \text{where } r\ge \ell_x>0\text{ and } x\in \spt\nu \subset R\setminus R_{\sf Far}.
$$
\end{lemma}
\begin{proof}
Recall that $\nu=g \mu$ with $g$ bounded  by a constant depending in $\tau$ and $A$, see (\ref{measure_nu}).

If $r>\diam R$, then $\spt \nu \subset B(x,r)$ and thus
$$
\nu(B(x,r))\lesssim_{\tau,A} \mu(2B_R) \approx_{\tau,A} \Theta_\mu(2B_R) \diam(R)  \lesssim_{\tau,A} \Theta_\mu(2B_R)\,r.
$$
Consequently, we may suppose below that
$\ell_x\le r\le \diam(R)$.

First let $d(x)\le C(\tau,A)\ell_x$, where $C(\tau,A)>0$ will be chosen later. Then there should exist $P\in {\sf DbTree}(R)$ such that $B(x,r)\subset 2B_P$ and $\diam(P)\approx_{\tau,A} r$ so that
$$
\nu(B(x,r))\lesssim_{\tau,A} \mu(B(x,r))\lesssim_{\tau,A} \mu(2B_P)\approx_{\tau,A} \Theta_\mu(2B_R) \diam(P)  \approx_{\tau,A} \Theta_\mu(2B_R) r.
$$

Now let $d(x)\ge C(\tau,A)\ell_x>0$. Set $y=(\Pi(x),F(\Pi(x)))\in \Gamma_R$. As shown in the proof of Lemma~\ref{paper3-lemma_gamma_R_1},
$d(y)\le c(\tau,A) \ell_x$ with some $c(\tau,A)>0$.
Choose $Q'\in {\sf DbTree}(R)$ so that
$$ 
\dist(y,Q')+\diam(Q')\le 2 d(y).
$$
Taking into account that $x\in R\setminus R_{\sf Far}$, from Lemma~\ref{paper3-lemma_gamma_R_3} and the properties of $\Gamma_R$ we deduce that
$\dist(x,y)\lesssim \sqrt[4]{\varepsilon_0}\,d(x)\le \sqrt[8]{\varepsilon_0}\,d(x) $ if $\varepsilon_0$ is chosen small enough. Thus
\begin{equation*}
\begin{split}
d(x)&\le \dist(x,Q')+\diam(Q') \le \dist(x,y)+ 2d(y) \le  \sqrt[8]{\varepsilon_0}\,d(x)+2c(\tau,A)\ell_x\\
&\le  \sqrt[8]{\varepsilon_0}\,d(x)+\frac{2c(\tau,A)}{C(\tau,A)}d(x) \le (\sqrt[8]{\varepsilon_0}+\tfrac{1}{2})d(x) <d(x),
\end{split}
\end{equation*}
if we choose $C(\tau,A)\ge 4c(\tau,A)$. Hence we get a contradiction if $d(x)\ge C(\tau,A)\ell_x>0$.
\end{proof}

By plugging the estimates obtained into (\ref{paper3-eq**9212}), choosing $\varepsilon_0=\varepsilon_0(\tau,A)$ small enough and recalling (\ref{eqq})  we get
\begin{equation*}
\label{paper3-eqd4*43}
p_{\infty, \mathbf{F}}(\sigma\lfloor \Gamma_{B_0}\setminus {G_R})  \lesssim \varepsilon_0^{1/40}\,\Theta_\mu(2B_R)^2\,\mu(R).
\end{equation*}

\medskip

Now it remains to estimate the last two terms of (\ref{paper3-eqfi**4}). The arguments are similar
to the preceding ones.

\medskip

\textbf{3.} Since $\sigma\lfloor G_R=\nu\lfloor G_R$, we have
\begin{equation*}\label{paper3-eqsigmanu01}
p_{\infty, \mathbf{F}}(\sigma\lfloor {G_R},\sigma\lfloor \Gamma_{B_0}\setminus {G_R},\sigma\lfloor \Gamma_{B_0}\setminus {G_R})
=p_{\infty, \mathbf{F}}(\nu\lfloor {G_R},\sigma\lfloor \Gamma_{B_0}\setminus {G_R},\sigma\lfloor \Gamma_{B_0}\setminus {G_R})
\end{equation*}
and
\begin{equation*}\label{paper3-eqsigmanu02}
p_{\infty, \mathbf{F}}(\sigma\lfloor {G_R},\sigma\lfloor {G_R},\sigma\lfloor \Gamma_{B_0}\setminus {G_R}) =
p_{\infty, \mathbf{F}}(\nu\lfloor {G_R},\nu\lfloor {G_R},\sigma\lfloor \Gamma_{B_0}\setminus {G_R}).
\end{equation*}
Concerning the term $p_{\infty, \mathbf{F}}(\sigma\lfloor {G_R},\sigma\lfloor \Gamma_{B_0}\setminus {G_R},\sigma\lfloor \Gamma_{B_0}\setminus {G_R})$, the main difference with respect to the estimates above for
$p_{\infty, \mathbf{F}}(\sigma\lfloor \Gamma_{B_0}\setminus {G_R})$
 is that
$T_x(y,z)$ equals zero in this case,
 and instead of integrating over $\sigma\lfloor \widehat{J}_i$
and $\nu_i$ and then summing on $i$, one integrates over $\sigma\lfloor {G_R}$. Then one obtains
\begin{align*}
\label{paper3-eq**9214}
p_{\infty, \mathbf{F}}(\sigma\lfloor {G_R},&\sigma\lfloor \Gamma_{B_0}\setminus{G_R},\sigma\lfloor \Gamma_{B_0}\setminus {G_R})\\
&\lesssim \varepsilon_0^{1/40}\Theta_\mu(2B_R)^2\,\mu(R) \\
&\quad
+ \iiint_{\,
\begin{subarray}{l} |x-y|\ge \frac{1}{2}\varepsilon_0^{-1/20}\ell_y \\
|x-z|\ge \frac{1}{2}\varepsilon_0^{-1/20}\ell_z\\
|y-z|\ge \frac{1}{2}\varepsilon_0^{-1/20}(\ell_y+\ell_z)\end{subarray}}\!\!\!\!\left[
T_y(x,z)+T_z(x,y)\right]\,d\sigma(x)\,d\sigma(y)\,d\sigma(z)\nonumber\\
&\quad
+ \iiint_{\,
\begin{subarray}{l} |x-y|\ge \frac{1}{2}\varepsilon_0^{-1/20}\ell_y \\
|x-z|\ge \frac{1}{2}\varepsilon_0^{-1/20}\ell_z\\
|y-z|\ge \frac{1}{2}\varepsilon_0^{-1/20}(\ell_y+\ell_z)\end{subarray}}\!\!\!\!\left[
T_y(x,z)+T_z(x,y)\right]\,d\nu(x)\,d\nu(y)\,d\nu(z).\nonumber
\end{align*}
The last two triple integrals are estimated as in (\ref{paper3-eq**de3}), and then it follows that
\begin{equation*}
\label{paper3-eqd4*44}
p_{\infty, \mathbf{F}}(\sigma\lfloor {G_R},\sigma\lfloor \Gamma_{B_0}\setminus{G_R},\sigma\lfloor \Gamma_{B_0}\setminus {G_R})
\lesssim \varepsilon_0^{1/40}\,\Theta_\mu(2B_R)^2\,\mu(R).
\end{equation*}

\medskip

\textbf{4.} Finally, the arguments for $p_{\infty, \mathbf{F}}(\sigma\lfloor {G_R},\sigma\lfloor {G_R},\sigma\lfloor \Gamma_{B_0}\setminus {G_R})$ are very similar. In this case, both terms $T_x(y,z)$ and $T_y(x,z)$
vanish, and analogously we also get
\begin{equation*}
\label{paper3-eqd4*45}
p_{\infty, \mathbf{F}}(\sigma\lfloor {G_R},\sigma\lfloor  {G_R},\sigma\lfloor \Gamma_{B_0}\setminus {G_R})
\lesssim  \varepsilon_0^{1/40}\,\Theta_\mu(2B_R)^2\,\mu(R).
\end{equation*}

This finishes the proof of Lemma~\ref{paper3-lemma_35}.

\medskip

\section{The packing condition for {\sf Top} cubes and the end of the proof of Main Lemma}
\label{paper3-sec_pack_Top}

\subsection{Properties of the trees} In order to prove the packing condition for ${\sf Top}$ cubes we will first extract some necessary results from
Lemmas~\ref{paper3-DbTree_properties},~\ref{paper3-measure_BP_F_stop_cubes},~\ref{paper3-lemma_ld_small},~\ref{paper3-lemma_bs_small} and~\ref{paper3-lemma_mu_G_R}. We suppose that all the parameters and thresholds from Section~\ref{paper3-sec_param} are  chosen  properly.  Recall also the definition (\ref{paper3-good_R}) of~$G_R$.
\begin{lemma}
\label{paper3-MainLemma} Let $\mu$ be a finite measure with compact support such that
$$
p_0(\mu)<\infty.
$$
Considering the David-Mattila dyadic lattice $\mathcal{D}$  associated with $\mu$, let $R\in\mathcal{D}^{db}$.
Then there exists a $C_F\theta_0$-Lipschitz function $F:L_R\to L_R^\perp$, where   $C_F>0$ is independent of $R$,
a family of pairwise disjoint cubes ${\sf Stop}(R)\subset \mathcal{D}(R)$ and a set $G_R\subset R$ such that

$\bf (a)$ $G_R$ is contained in
$\Gamma_R=F(L_R)$ and moreover $\mu\lfloor G_R$ is absolutely continuous with respect
to $\Theta_\mu(2B_R)\mathcal{H}^1_{\Gamma_R}$;

$\bf (b)$ for any $Q\in{\sf Tree}(R)$,
$$
\Theta_\mu(2B_Q)\lesssim A\,\Theta_\mu(2B_R);
$$

$\bf (c)$ if $R \in {\sf T}_{VF}(\theta_0)$, then
$$
\sum_{\begin{subarray}{c}
Q\in{\sf Stop}(R)\\
Q\notin  {\sf HD}(R)\cup{\sf UB}(R)
\end{subarray}} \!\!\!\!\!\!\!\!\!\!\!\mu(Q)\le \sqrt{\tau}\,\mu(R) +
\frac{1}{\alpha^2\Theta_\mu(2B_{R})^2}\!\!\sum_{\tilde{Q}\in{\sf Tree}(R)}\!\!\!\!
p^{[\delta,\tilde{Q}]}_0(\mu\lfloor 2B_{\tilde{Q}},\mu\lfloor 2B_R,\mu\lfloor 2B_R);
$$
if $R \notin {\sf T}_{VF}(\theta_0)$, then
$$
\sum_{\begin{subarray}{c}
Q\in{\sf Stop}(R)\\
Q\notin  {\sf HD}(R)\cup{\sf UB}(R) \cup{\sf BS}(R)
\end{subarray}}\!\!\!\!\!\!\!\!\!\!\!\!\!\!\!\mu(Q)\le \sqrt{\tau}\,\mu(R) +
\frac{1}{\alpha^2\Theta_\mu(2B_{R})^2}\!\!\sum_{\tilde{Q}\in{\sf Tree}(R)}\!\!\!\!
p^{[\delta,\tilde{Q}]}_0(\mu\lfloor 2 B_{\tilde{Q}},\mu\lfloor 2B_R,\mu\lfloor 2B_R).
$$
\end{lemma}

\subsection{New families of stopping cubes}

According to Section~\ref{paper3-sec_stopping} and Lemma~\ref{paper3-MainLemma}, each $R\in\mathcal{D}^{db}$ generates several families of cubes fulfilling certain properties. In this subsection we will introduce some variants of these families. The idea is to have stopping cubes that are always different from $R$ and are in $\mathcal{D}^{db}$, cf. Remark~\ref{paper3-remark3}.

Recall that each cube in  ${\sf HD}(R)$ is in $\mathcal{D}^{db}$ and is clearly different from $R$ due to the fact that $Q\in\HD(R)$ satisfies $\Theta_\mu(2B_Q)>A\,\Theta_\mu(2B_R)$ with $A\gg 1$.

Now we turn our attention to the family ${\sf UB}(R)$. By
Lemma \ref{paper3-lemma_balanced_balls}, if $Q\in{\sf UB}(R)$, i.e. it is $\gamma$-unbalanced, there exists a family of
pairwise disjoint cubes $\{P\}_{P\in I_Q}\subset\mathcal{D}^{db}(Q)$ such that $\diam(P) \gtrsim  \gamma\diam(Q)$  and
$\Theta_\mu(2B_{P})\gtrsim \gamma^{-1}\,\Theta_\mu(2B_Q)$ for each
$P\in I_Q$, and
\begin{equation}\label{paper3-eqsgk32**}
\sum_{P\in I_Q} \Theta_\mu(2B_{P})^2\,\mu(P)\gtrsim
\gamma^{-2}\,\Theta_\mu(2B_Q)^2\,\mu(Q).
\end{equation}
Let
$I'_Q$ be a family of (not necessarily doubling) cubes contained in $Q$, with side
length comparable to $a\diam(Q)$ with some $a>0$, disjoint from the ones from
$I_Q$, so that
$$
Q=\bigcup_{P\in I_Q} P \cup \bigcup_{P\in I'_Q} P.
$$

To continue, we introduce additional notation. Given a cube $Q\in\mathcal{D}$, we denote by $\MD(Q)$ the family of maximal cubes (with respect to inclusion)
from $\mathcal{D}^{db}(Q)$. By Lemma \ref{paper3-lemcobdob}, this family covers $\mu$-almost all $Q$.
Furthermore, using the definition just given, we denote by $\widetilde{I}_Q$ the family $\bigcup_{P\in I'_Q}\mathcal{MD}(P)$. Moreover, we set
$$
\widetilde{\UB}(R) = \bigcup_{Q\in{\sf UB}(R)} (I_Q \cup \widetilde{I}_Q).
$$
One can deduce from (\ref{paper3-eqsgk32**}) that $R\not\in\wt{\UB}(R)$  for $a$ and $\gamma$ small enough.

Now consider ${\sf BS}(R)$. Each cube in this family is in $\mathcal{D}^{db}$ by construction. Moreover, $R\notin {\sf BS}(R)$ due to the condition $\measuredangle (L_Q,L_R)>\theta(R)>0$ for each $Q\in {\sf BS}(R)$.

To continue, we set
\begin{equation*}
\label{paper3-def_O}
 {\mathsf{O}}(R) = {\sf Stop}(R)\setminus
(\HD(R)\cup{\sf UB}(R)\cup {\sf BS}(R))=\textsf{LD}(R)\cup \textsf{BP}(R)\cup {\sf F}(R)
\end{equation*}
and
$$
\widetilde{\mathsf{O}}(R) = \left\{\bigcup_{Q\in\mathcal{D}}\mathcal{MD}(Q):\,\mbox{$Q$ is a son of some cube from ${\mathsf{O}}(R)$}\right\}.
$$
This guarantees that $R\notin\tilde{\mathsf{O}}(R)$ as cubes in $\tilde{\mathsf{O}}(R)$ are descendants of cubes in
${\sf Tree}(R)$.

Finally, let
\begin{equation*}
\label{paper3-def_Next}
{\sf Next}(R) =\HD(R) \cup\widetilde{\UB}(R) \cup \widetilde{\mathsf{O}}(R)\cup {\sf BS}(R).
\end{equation*}
By construction, all cubes in ${\sf Next}(R)$ are disjoint, doubling and different from $R$. Moreover,
\begin{equation}
\label{paper3-new_good}
R\setminus
\bigcup_{Q\in \nex(R)} Q=R\setminus
\bigcup_{Q\in {\sf Stop}(R)} Q.
\end{equation}
Using the small boundaries property of the David-Mattila lattice and the definition (\ref{paper3-good_R}), one can also show  that
\begin{equation}
\label{paper3-good_R_thru_stop}
\mu\left(R\setminus
\bigcup_{Q\in {\sf Stop}(R)} Q\right)=\mu(G_R).
\end{equation}
For the record, notice also that, by construction, if
$P\in\nex(R)$, then
\begin{equation}
\label{paper3-eqdens***}
\Theta_\mu(2B_S)\lesssim_{\tau,A}\,\Theta_\mu(2B_R) \quad\mbox{for all
$S\in\mathcal{D}$ such that $P\subset S\subset R$.}
\end{equation}

\subsection{The corona decomposition}

Recall that we assumed that $\mu$ has compact support.
Let
$$
R_0:=\spt \mu.
$$
Obviously we may suppose that $R_0\in \mathcal{D}^{db}$. We will
construct the family ${\sf Top}$ contained in $R_0$
inductively applying Lemma~\ref{paper3-MainLemma} so that ${\sf
Top}=\bigcup_{k\ge0}{\sf Top}_k$. Let
$${\sf Top}_0=\{R_0\}.$$
Assuming ${\sf Top}_k$ to be defined, we set
$${\sf Top}_{k+1}  = \bigcup_{R\in{\sf Top}_k}  \nex(R).$$
Note that cubes in $\nex(R)$, with $R\in{\sf Top}_k$, are pairwise
disjoint.

\subsection{The families of cubes $ID_H$, $ID_U$ and $ID$}

We distinguish two types of cubes $R\in{\sf Top}$. We write $R\in
ID_H$ (increasing density because of high density cubes) if
$$
\mu\biggl(\,\bigcup_{Q\in {\HD}(R)} Q\biggr)\ge \frac{1}{4}  \,\mu(R).
$$
Also, we write $R\in ID_U$ (increasing density because of unbalanced
cubes) if
$$
\mu\biggl(\,\bigcup_{Q\in\wt{\UB}(R)} Q\biggr)\ge \frac{1}{4}  \,\mu(R).
$$
Additionally, let
$$
ID = ID_H\cup ID_U.
$$

\begin{lemma}[Lemma 5.4 and its proof in \cite{AT}]\label{paper3-lemID}
If $R\in ID$, then
\begin{align*}
&\Theta_\mu(2B_R)^2\,\mu(R) \lesssim \frac{1}{A^2}\sum_{Q\in {\sf HD}(R)}\Theta_\mu(2B_Q)^2\,\mu(Q)\\
&\Theta_\mu(2B_R)^2\,\mu(R) \lesssim \frac{\gamma^2}{\tau^2} \sum_{Q\in \widetilde{{\sf UB}}(R)}\Theta_\mu(2B_Q)^2\,\mu(Q).
\end{align*}
Moreover, if $A$ is such that $A^{-1}\le \tau^2$ and $\gamma\le \tau^3$, then
$$
\Theta_\mu(2B_R)^2\,\mu(R)\le
c\tau^4 \sum_{Q\in
\nex(R)}\Theta_\mu(2B_Q)^2\,\mu(Q),
$$
where $c>0$ is some absolute constant.
\end{lemma}

\subsection{The packing condition}
Recall that we assume linear growth of $\mu$, i.e.
\begin{equation}\label{paper3-eqc*}
\mu(B(x,r))\le  C_*r\qquad \forall x\in \spt \mu,\quad r>0,
\end{equation}
for some constant $C_*>0$ (see (\ref{linear_growth})). Using this assumption, we will prove the following.
\begin{lemma}\label{paper3-lemkey62}
If the parameters and thresholds in Section~\ref{paper3-sec_param} are chosen properly, then
\begin{equation}\label{paper3-eqsum441}
\sum_{R\in{\sf Top}}\Theta_\mu(2B_R)^2\,\mu(R)\le
{\sf c_5}\,p_0(\mu)+c\,C_*^2\,\mu(\mathbb{C}),
\end{equation}
where ${\sf c_5}={\sf c_5}(\tau,A,\theta_0,\gamma,\varepsilon_0,\alpha,\delta)>0$ and $c>0$.
\end{lemma}
\begin{proof}
For a given $k\ge 0$, we set
${\sf Top}_0^k = \bigcup_{0\le  j\le k}{\sf Top}_j$ and $ID_0^k = ID \cap {\sf Top}_0^k$.

To prove (\ref{paper3-eqsum441}), first we deal with the cubes from the $ID$
family. By Lemma \ref{paper3-lemID},
\begin{align*}
\sum_{R\in ID_0^k} \Theta_\mu(2B_R)^2\,\mu(R) &\le c\tau^2 \sum_{R\in
ID_0^k}\sum_{Q\in \nex(R)}\Theta_\mu(2B_Q)^2\,\mu(Q)\\
& \le c\tau^2\sum_{R\in {\sf Top}_0^{k+1}}\Theta_\mu(2B_R)^2\,\mu(R),
\end{align*}
because the cubes from $\nex(R)$ with $R\in{\sf Top}_0^k$ belong to
${\sf Top}_0^{k+1}$. So we have
\begin{align*}
&\sum_{R\in {\sf Top}_0^k} \Theta_\mu(2B_R)^2\,\mu(R)\\
&= \sum_{R\in
{\sf Top}_0^k\setminus ID_0^k}\!\!\! \Theta_\mu(2B_R)^2\,\mu(R)+
\sum_{R\in ID_0^k} \Theta_\mu(2B_R)^2\,\mu(R)\\
& \le \sum_{R\in {\sf Top}_0^{k}\setminus
ID_0^k}\!\!\!\Theta_\mu(2B_R)^2\,\mu(R) + c\tau^2\!\!\!\sum_{R\in
{\sf Top}_0^{k}}\!\Theta_\mu(2B_R)^2\,\mu(R) + c\tau^2\!\!\!\sum_{R\in
{\sf Top}_{k+1}}\!\Theta_\mu(2B_R)^2\,\mu(R)\\
& \le \sum_{R\in {\sf Top}_0^{k}\setminus
ID_0^k}\!\!\!\Theta_\mu(2B_R)^2\,\mu(R) + c\tau^2\!\!\!\sum_{R\in
{\sf Top}_0^{k}}\!\Theta_\mu(2B_R)^2\,\mu(R) +c\,\tau^2\,C_*^2\mu(R_0),
\end{align*}
where we took into account that $\Theta_\mu(2B_R)\lesssim C_*$
for every $R\in{\sf Top}$ (and in particular for all $R\in{\sf Top}_{k+1}$). So, having $\tau$ small enough, we deduce that
\begin{equation}
\label{paper3-R_Tree_0_k}
\sum_{R\in {\sf Top}_0^k} \Theta_\mu(2B_R)^2\,\mu(R)  \le 1.1\sum_{R\in {\sf Top}_0^{k}\setminus ID_0^k}\Theta_\mu(2B_R)^2\,\mu(R) + c\tau^2\,C_*^2\mu(R_0).
 \end{equation}

Let us estimate the first term in the right hand side of (\ref{paper3-R_Tree_0_k}).
First note that
$$
 \mu\biggl(R\setminus \bigcup_{Q\in {\HD}(R)\cup
\wt{\UB}(R)} Q\biggr)\ge \frac{1}{2}\,\mu(R)\qquad \text{for any} R\in{\sf Top}_0^{k}\setminus ID_0^k.
$$
Next, by applying the inequalities $(c)$ in  Lemma~\ref{paper3-MainLemma} and recalling (\ref{paper3-new_good}) and (\ref{paper3-good_R_thru_stop}) we get
\begin{align*}
\mu(R)& \le 2\,\mu\biggl(R\setminus
\bigcup_{Q\in \nex(R)} Q\biggr) + 2\,\mu\biggl(
\bigcup_{Q\in \wt{\textsf{O}}(R)\cup {\sf BS}(R)}Q\biggr)\\
& \le 2\,\mu(G_R) + 2
\mu\biggl(\bigcup_{Q\in \wt{\textsf{O}}(R)}Q\biggr) +2
\sum_{\begin{subarray}{c}
Q\in {\sf BS}(R)\\
(\text{if } R\in {\sf T}_{VF}(\theta_0))
\end{subarray}
}\mu(Q)
+2\sum_{\begin{subarray}{c} Q\in {\sf BS}(R)\\
(\text{if }R\notin {\sf T}_{VF}(\theta_0))
\end{subarray}}\mu(Q)
\\
& \le 2 \,\mu(G_R)+2\sqrt{\tau}\,\mu(R)
+2\sum_{\begin{subarray}{c} Q\in {\sf BS}(R)\\
(\text{if }R\notin {\sf T}_{VF}(\theta_0))
\end{subarray}}\mu(Q)\\
&\quad+\frac{2\alpha^{-2}}{\Theta_\mu(2B_{R})^2}\sum_{{Q}\in{\sf Tree}(R)}
p^{[\delta,Q]}_0(\mu\lfloor 2B_{Q},\mu\lfloor 2B_R,\mu\lfloor 2B_R).
\end{align*}
Suppose that $\tau$ is small enough to get
\begin{align*}
\mu(R)
& \le 2.1 \,\mu(G_R)+
2.1\sum_{\begin{subarray}{c} Q\in {\sf BS}(R)\\
(\text{if }R\notin {\sf T}_{VF}(\theta_0))
\end{subarray}}\mu(Q)\\
&\qquad+\frac{2.1\alpha^{-2}}{\Theta_\mu(2B_{R})^2}\sum_{{{Q}}\in{\sf Tree}(R)}
p^{[\delta,{{Q}}]}_0(\mu\lfloor 2 B_{{Q}},\mu\lfloor 2B_R,\mu\lfloor 2B_R).
\end{align*}
So we deduce from (\ref{paper3-R_Tree_0_k}) that
\begin{align}
\nonumber
\sum_{R\in {\sf Top}_0^k} \Theta_\mu(2B_R)^2\,&\mu(R) \\
&\label{paper3-eqal163}
\!\!\!\!\!\! \!\!\!\!\!\! \!\!\!\!\!\!  \le 3 \sum_{R\in {\sf Top}_0^k\setminus ID_0^k}
\Theta_\mu(2B_R)^2\,\mu(G_R)\\
&\!\!\!\!\!\! \!\!\!\!\!\! \!\!\!\!\!\! \quad+ \frac{3}{\alpha^2} \sum_{R\in {\sf Top}_0^k}\sum_{Q\in{\sf Tree}(R)}
p^{[\delta,Q]}_0(\mu\lfloor 2B_Q,\mu\lfloor 2B_R,\mu\lfloor 2B_R)\nonumber\\
&\nonumber \!\!\!\!\!\! \!\!\!\!\!\! \!\!\!\!\!\! \quad+3\sum_{R\in {\sf Top}_0^k\setminus (ID_0^k\cup {\sf T}_{VF}(\theta_0))} \Theta_\mu(2B_R)^2\sum_{Q\in {\sf BS}(R)}\mu(Q)\\
&\nonumber \!\!\!\!\!\! \!\!\!\!\!\! \!\!\!\!\!\! \quad+ c\tau^2\,C_*^2\mu(R_0).
\end{align}
In order to deal with the first sum on the right hand side we take into
account that $\Theta_\mu(2B_R)\lesssim C_*$ for all $R\in{\sf Top}$ by
\eqref{paper3-eqc*} and that the sets $G_R$ with $R\in{\sf Top}$ are pairwise disjoint. Then we get
$$
\sum_{R\in {\sf Top}_0^k\setminus ID_0^k} \Theta_\mu(2B_R)^2\,\mu(G_R)\le c\,C_*^2\,\mu(R_0).
$$
On the other hand, the double sum in (\ref{paper3-eqal163}) does not exceed
$$
2\sum_{Q\in \mathcal{D}}
 p^{[\delta,Q]}_0(\mu\lfloor 2B_Q,\mu\lfloor 2B_R,\mu\lfloor 2B_R) \le
c(\delta)\,p_0(\mu),
$$
by the finite superposition of the domains of integration. Recall that $\delta=\delta(\gamma,\varepsilon_0)$. So we obtain
\begin{align}
\label{paper3-R_Tree_0_k_TV_Far_Final}
\sum_{R\in {\sf Top}_0^k} \Theta_\mu(2B_R)^2\,\mu(R) \le &c\,C_*^2\,\mu(R_0)+ c(\tau,A,\gamma,\varepsilon_0,\alpha)\, p_0(\mu)\\
&\nonumber\quad +c\sum_{R\in {\sf Top}_0^k\setminus (ID_0^k\cup {\sf T}_{VF}(\theta_0))}\!\!\!\!\!\! \Theta_\mu(2B_R)^2\sum_{Q\in {\sf BS}(R)}\mu(Q).
\end{align}
The third term in (\ref{paper3-R_Tree_0_k_TV_Far_Final}) without the constant may be written as the sum
$$
\sum_{R\in {\sf Top}_0^k\setminus (ID_0^k\cup {\sf T}_{VF}(\theta_0))} \Theta_\mu(2B_R)^2(S_1(R)+S_2(R)),
$$
where
$$
S_1(R)=\sum_{Q\in {\sf BS}(R)\cap {\sf T}_{VF}(\theta_0)\setminus ID_0^{k+1}}\!\!\!\!\!\!\mu(Q)\quad\text{and}\quad S_2(R)=\sum_{Q\in {\sf BS}(R)\cap {\sf T}_{VF}(\theta_0)\cap  ID_0^{k+1}}\!\!\!\!\!\!\mu(Q).
$$
Note that we have the intersection with ${\sf T}_{VF}(\theta_0)$ in these sums. This is so because for any $Q\in {\sf BS}(R)$, where $R\in {\sf Top}\setminus {\sf T}_{VF}(\theta_0)$, it holds that
$$
\theta_{V}(L_Q)\ge \measuredangle(L_Q,L_R)-\theta_{V}(L_R)\ge 2(1+C_F)\theta_0-(1+C_F)\theta_0=(1+C_F)\theta_0,
$$
and thus $Q\in {\sf T}_{VF}(\theta_0)$.

Let us estimate $S_1(R)$. Since $Q\in {\sf T}_{VF}(\theta_0)\setminus ID_0^{k+1}$, we deduce from
$(c)$ in Lemma~\ref{paper3-MainLemma} that
\begin{align*}
\mu(Q)& \le 2\,\mu\biggl(Q\setminus
\bigcup_{P\in \nex(Q)} P\biggr) + 2\,\mu\biggl(
\bigcup_{P\in \wt{\textsf{O}}(Q)\cup {\sf BS}(Q)}P\biggr)\\
& \le 2\,\mu(G_Q) + 2
\mu\biggl(\bigcup_{Q\in \wt{\textsf{O}}(R)}Q\biggr)+2\sum_{\begin{subarray}{c}P\in {\sf BS}(Q)\\(\text{if }Q\in {\sf T}_{VF}(\theta_0))\end{subarray}}\mu(P)
\\
& \le 2 \,\mu(G_Q)+2\sqrt{\tau}\,\mu(Q)+\frac{2\alpha^{-2}}{\Theta_\mu(2B_{Q})^2}\sum_{P\in{\sf Tree}(Q)}
p^{[\delta,P]}_0(\mu\lfloor 2 B_P,\mu\lfloor 2B_Q,\mu\lfloor 2B_Q).
\end{align*}
If $\tau$ is small enough, then
\begin{align*}
\mu(Q)
& \le 2.1 \,\mu(G_Q)+\frac{2.1\alpha^{-2}}{\Theta_\mu(2B_{Q})^2}\sum_{P\in{\sf Tree}(Q)}
p^{[\delta,P]}_0(\mu\lfloor 2 B_P,\mu\lfloor 2B_Q,\mu\lfloor 2B_Q).
\end{align*}
Recall that  ${\sf BS}(R)\cap {\sf T}_{VF}(\theta_0)\setminus ID_0^{k+1}\subset \nex(R)$. So we deduce that
\begin{align*}
S_1(R)&\le 2.1 \!\!\!\sum_{Q\in \nex(R)}\biggl(
\mu(G_R)+\frac{\alpha^{-2}}{\Theta_\mu(2B_{R})^2}
 \sum_{P\in{\sf Tree}(Q)}
p^{[\delta,P]}_0(\mu\lfloor 2B_P,\mu\lfloor 2B_Q,\mu\lfloor 2B_Q)\biggr).
\end{align*}
Consequently, using that $\Theta_\mu(2B_R)\lesssim C_*$, we obtain
\begin{align*}
\sum_{R\in {\sf Top}_0^k\setminus (ID_0^k\cup {\sf T}_{VF}(\theta_0))} \!\!\!\!\!\!\!\!\!\!\!\!&\Theta_\mu(2B_R)^2S_1(R)\\
&\!\!\!\!\!\!\!\!\!\!\!\!\!\!\!\!\!\!\le c\,C_*^2\sum_{R\in {\sf Top}_0^k} \sum_{Q\in \nex(R)}
\mu(G_Q)\\
&\!\!\!\!\!\!\!\!\!\!\!\!\!\!\!\!\!\!\quad+ \frac{c}{\alpha^2}\sum_{R\in {\sf Top}_0^k\setminus (ID_0^k\cup {\sf T}_{VF}(\theta_0))}  \sum_{Q\in \nex(R)}
 \sum_{P\in{\sf Tree}(Q)}
p^{[\delta,P]}_0(\mu\lfloor 2B_P,\mu\lfloor 2B_Q,\mu\lfloor 2B_Q)\\
&\!\!\!\!\!\!\!\!\!\!\!\!\!\!\!\!\!\!\le c\,C_*^2\sum_{R\in {\sf Top}_0^{k+1}}
\mu(G_R)+ \frac{c}{\alpha^2}\sum_{R\in {\sf Top}_0^{k+1}}
 \sum_{P\in{\sf Tree}(R)}
p^{[\delta,P]}_0(\mu\lfloor 2B_P,\mu\lfloor 2B_R,\mu\lfloor 2B_R).
\end{align*}
Take into account that the sets $G_R$ with $R\in {\sf Top}$ are disjoint and that
the last (double) sum is controlled by $c(\delta)\,p_0(\mu)$ by the finite superposition of the domains of integration. So we have
$$
\sum_{R\in {\sf Top}_0^k\setminus (ID_0^k\cup {\sf T}_{VF}(\theta_0))} \Theta_\mu(2B_R)^2S_1(R)\le c\,C_*^2\mu(R_0)+c(\tau,A,\delta,\alpha)\,p_0(\mu).
$$

\medskip

Now we estimate $S_2(R)$. Since ${\sf BS}(R)\cap {\sf LD}(R)=\varnothing$, for each $Q \in {\sf BS}(R)$ we have $\Theta_\mu(2B_Q)\ge \tau\Theta_\mu(2B_R)$ and thus
$$
S_2(R)\le \frac{1}{\tau^2\Theta_\mu(2B_R)^2}\sum_{Q\in {\sf BS}(R)\cap {\sf T}_{VF}(\theta_0)\cap  ID_0^{k+1}}\Theta_\mu(2B_Q)^2\mu(Q).
$$
Since $Q\in ID_0^{k+1}$, by Lemma~\ref{paper3-lemID},
\begin{align*}
S_2(R)&\le \frac{1}{\tau^2\Theta_\mu(2B_R)^2}\sum_{Q\in {\sf BS}(R)\cap {\sf T}_{VF}(\theta_0)\cap  ID_0^{k+1}}c\tau^4\sum_{P\in \nex(Q)}\Theta_\mu(2B_P)^2\mu(P)\\
& \le \frac{c\tau^2}{\Theta_\mu(2B_R)^2}\sum_{Q\in {\sf BS}(R)\cap {\sf T}_{VF}(\theta_0)\cap  ID_0^{k+1}}\sum_{P\in \nex(Q)}\Theta_\mu(2B_P)^2\mu(P).
\end{align*}
Consequently, taking into account that ${\sf BS}(R)\cap {\sf T}_{VF}(\theta_0)\cap  ID_0^{k+1} \subset \nex(R)$ and ${\sf Top}_0^k\setminus (ID_0^k\cup {\sf T}_{VF}(\theta_0))\subset {\sf Top}_0^k$, we obtain
\begin{align*}
\sum_{R\in {\sf Top}_0^k\setminus (ID_0^k\cup {\sf T}_{VF}(\theta_0))} \Theta_\mu(2B_R)^2S_2(R) &\le c\tau^2 \sum_{R\in {\sf Top}_0^k}\sum_{Q\in \nex(R)}\sum_{P\in \nex(Q)}\Theta_\mu(2B_P)^2\mu(P)\\
&\le c\tau^2\sum_{R\in {\sf Top}_0^{k+2}}\Theta_\mu(2B_R)^2\mu(R)\\
&\le c\tau^2\sum_{R\in {\sf Top}_0^{k}}\Theta_\mu(2B_R)^2\mu(R)+c\tau^2 C^2_*\mu(R_0).
\end{align*}

\medskip

Coming back to (\ref{paper3-R_Tree_0_k_TV_Far_Final}), we deduce that
\begin{align*}
\sum_{R\in {\sf Top}_0^k}& \Theta_\mu(2B_R)^2\,\mu(R) \\
&\le  c\,C_*^2\,\mu(R_0)+c(\tau,A,\delta,\alpha)\,p_0(\mu)+c\tau^2\sum_{R\in {\sf Top}_0^{k}}\Theta_\mu(2B_R)^2\mu(R).
\end{align*}
Choosing $\tau$ small enough and recalling the information in Section~\ref{paper3-sec_param} yield
$$
\sum_{R\in {\sf Top}_0^k} \Theta_\mu(2B_R)^2\,\mu(R)\le  {\sf c_5} p_0(\mu)+c\,C_*^2\,\mu(R_0),
$$
where ${\sf c_5}$ actually depends on all the parameters and thresholds mentioned in Section~\ref{paper3-sec_param}.

Letting $k\to\infty$ finishes the proof of Lemma~\ref{paper3-lemkey62}.
\end{proof}

\subsection{The end of the proof of Main Lemma}
\label{paper3-sec-the-end}

We first prove an additional property. For $Q,\tilde{Q}\in \mathcal{D}$ with $Q\subset \tilde{Q}$, define
$$
\delta_\mu(Q,\tilde{Q})=\int_{2B_{\tilde{Q}}\setminus 2B_{Q}}\frac{1}{|y-z_Q|}d\mu(y),
$$
where $z_Q$ is the center of $B_Q$, see Lemma~\ref{paper3-lemcubs}. Then the following statement holds.
\begin{lemma}
\label{paper3-one_more_lemma}
For all $Q\in \nex(R)$ there exists a cube $\tilde{Q}\in {\sf DbTree}(R)$ such that $\delta_\mu(Q,\tilde{Q})\lesssim_{\tau,A}\Theta_\mu(2B_R)$ and $2B_{\tilde{Q}}\cap \Gamma_R\neq \varnothing$.
\end{lemma}
\begin{proof}
Take $Q'\supset Q$ such that $Q'\in {\sf Stop}(R)$. By Lemma~\ref{paper3-lemdob32}, there exists $\tilde{Q}\in {\sf DbTree}(R)$ such that $Q'\subset \tilde{Q}$ and $r(Q')\approx_{\tau,A}r(\tilde{Q})$. Moreover, one can easily deduce from Lemma~\ref{paper3-lemma_gamma_R_3} that $2B_{\tilde{Q}}\cap \Gamma_R\neq \varnothing$ if $\varepsilon_0$ is small enough (since $\tilde{Q}\in {\sf DbTree}(R)$, there is $x\in \tilde{Q}\setminus R_{\sf Far}$).

Furthermore, split
$$
\delta_\mu(Q,\tilde{Q})=\int_{2B_{\tilde{Q}}\setminus 2B_{Q'}}\frac{1}{|y-z_Q|}d\mu(y)+\int_{2B_{Q'}\setminus 2B_{Q}}\frac{1}{|y-z_Q|}d\mu(y).
$$
In the first integral we have $|y-z_Q|\gtrsim r(Q') \approx_{\tau,A} r(\tilde{Q})$ as $y\notin  2B_{Q'}$ and therefore
$$
\int_{2B_{\tilde{Q}}\setminus 2B_{Q'}}\frac{1}{|y-z_Q|}d\mu(y)\lesssim_{\tau,A}\Theta_\mu(2B_{\tilde{Q}})\lesssim_{\tau,A} \Theta_\mu(2B_R),
$$
where we also used the right hand side inequality in (\ref{paper3-theta_Q_Tree}) in Lemma~\ref{paper3-DbTree_properties}.
To estimate the second integral we take into account that by construction there are no doubling cubes strictly between $Q$ and $Q'$. This together with Lemma~\ref{paper3-lemcad23} and  properties of $Q'$ and $\tilde{Q}$ imply by standard estimates (in particular, splitting the domain of integration into annuli with respect to the intermediate cubes between $Q$ and $Q'$) that
$$
\int\limits_{2B_{Q'}\setminus 2B_{Q}}\frac{1}{|y-z_Q|}d\mu(y)\lesssim \Theta_\mu(100B(Q'))\lesssim_{\tau,A}\Theta_\mu(2B_{\tilde{Q}}) \lesssim_{\tau,A} \Theta_\mu(2B_R).
$$

Thus we are done.
\end{proof}

Lemma~\ref{paper3-MainLemma}, Lemma~\ref{paper3-one_more_lemma} and the property (\ref{paper3-eqdens***}) allow us to use arguments as in \cite[Lemma 17.6]{Tolsa-delta} in order to show that if $\mu(B(x,r))\le C_* r$ for all
$x\in\mathbb{C}$, then
$$
c^2(\mu) \lesssim  \sum_{R\in {\sf
Top}}\Theta_\mu(2B_R)^2\,\mu(R)
$$
for our family ${\sf Top}$. By combining this estimate and the identity (\ref{curvature_pointwise}) with Lemma~\ref{paper3-lemkey62} for fixed suitable
parameters from Section~\ref{paper3-sec_param},  we obtain
$$
p_\infty(\mu)\lesssim p_0(\mu)+{C_*^2}\mu(\mathbb{C})
$$
as wished.

\section{The case of curvature. The bi-Lipschitz invariance of the Cauchy transform}
\label{paper3-sec_pack_Tree}

Here we come back to the notion of curvature $c^2(\mu)$. Recall that
$p_\infty(\mu)=\tfrac{1}{4} c^2(\mu)$. 

It is easy to see that one can exchange $p_0$ for $c^2$ in the
stopping conditions. Then we can prove the following
analogue of Lemma~\ref{paper3-lemkey62} by the arguments used above.
\begin{lemma}
\label{paper3-lemkey2}
If the parameters and thresholds in Section~\ref{paper3-sec_param} are chosen properly, then
\begin{equation}\label{paper3-eqsum441_new}
\sum_{R\in{\sf Top}}\Theta_\mu(2B_R)^2\,\mu(R)\le
{\sf c_6}\,c^2(\mu)+c\,C_*^2\,\mu(\mathbb{C}),
\end{equation}
where ${\sf c_6}={\sf c_6}(\tau,A,\theta_0,\gamma,\varepsilon_0,\alpha,\delta)>0$ and $c>0$.
\end{lemma}
A more direct way to prove this is to use Lemma~\ref{paper3-lemkey62} and the inequality (\ref{paper2-Main_inequality}).

Now recall the following theorem from \cite{AT}:

\vspace{1mm}
{\it
If $\mu$ is a finite compactly supported measure such that
$\mu(B(x,r))\le r$ for all $x\in\mathbb{C}$ and $r>0$, then
\begin{equation}\label{eqFF9}
c^2(\mu) + \mu(\mathbb{C})\approx \iint_0^\infty
\beta_{\mu,2}(x,r)^2\,\Theta_\mu(x,r)\,\frac{dr}r\,d\mu(x) +
\mu(\mathbb{C}),
\end{equation}
where the implicit constants are absolute.
}\vspace{1mm}

Note that the part $\lesssim$ of \eqref{eqFF9} was proved in \cite{AT} by means of the
David-Mattila lattice and a corona type construction similar to the one we considered in this chapter. However, the part $\gtrsim$ was proved in  \cite{AT} by the corona decomposition of
\cite{Tolsa-bilip} that involved the usual dyadic lattice $\mathcal{D}(\mathbb{C})$, instead of the David-Mattila
lattice $\mathcal{D}$.

Using Lemma~\ref{paper3-lemkey2}  we can also prove the part $\gtrsim$ of \eqref{eqFF9}
using only the David-Mattila lattice and an associated corona type
construction and thus unify the approach with the proof of the part
$\lesssim$ in \cite{AT}. As predicted in \cite[Section
19]{Tolsa-delta}, this indeed simplifies some of the technical difficulties
arising from the lack of a well adapted dyadic lattice to the
measure $\mu$ in \cite{Tolsa-bilip}.

Clearly, we need  to show that
\begin{equation}
\label{supernew}
\iint_0^\infty
\beta_{\mu,2}(x,r)^2\,\Theta_\mu(x,r)\,\frac{dr}r\,d\mu(x)\lesssim
c^2(\mu) + \mu(\mathbb{C})
\end{equation}
or, equivalently, in a discrete form  that
\begin{equation}\label{paper3-eqff89}
\sum_{Q\in \mathcal{D}}
\beta_{\mu,2}(2B_Q)^2\,\Theta_\mu(2B_Q)\,\mu(Q)\lesssim
 c^2(\mu) + \mu(\mathbb{C}).
 \end{equation}
By the packing condition (\ref{paper3-eqsum441_new}) for $C_*=1$, to prove
(\ref{paper3-eqff89}) it suffices to show that for every $R\in{\sf Top}$ the following estimate holds true:
$$
S(R)=\sum_{Q\in \widetilde{{\sf Tree}}(R)} \beta_{\mu,2}(2B_Q)^2\,\Theta_\mu(2B_Q)\,\mu(Q) \lesssim
\Theta_\mu(2B_R)^2 \mu(R),
$$
where $\widetilde{{\sf Tree}}(R)$ contains cubes in $R$ not strictly
contained in $\widetilde{{\sf Stop}}(R)$. By ${\sf St}(R)$ we denote
cubes in ${\sf Stop}(R)$ not strictly contained in $\widetilde{{\sf
Stop}}(R)$. Obviously, $\beta_{\mu,2}(2B_Q)^2\le 4\Theta_\mu(2B_Q)$
for any $Q\in \widetilde{{\sf Tree}}(R)$ and therefore
$$
S(R)\le \sum_{Q\in {\sf Tree}(R)\setminus{\sf Stop}(R)} \beta_{\mu,2}(2B_Q)^2\,\Theta_\mu(2B_Q)\,\mu(Q) +\sum_{Q\in {\sf St}(R)} \,\Theta_\mu(2B_Q)^2\,\mu(Q).
$$
By Lemma~\ref{paper3-lemcad23}, the density of all intermediate cubes between $\widetilde{{\sf Stop}}(R)$ and ${\sf Stop}(R)$,
 i.e. cubes in ${\sf St}(R)$, is controlled by the density of cubes from ${\sf Stop}(R)$ so
 it can be shown that
$$
\sum_{P\in {\sf St}(R)} \,\Theta_\mu(2B_P)^2\,\mu(P)\lesssim
\sum_{Q\in {\sf Stop}(R)}\,\Theta_\mu(2B_Q)^2\,\mu(Q).
$$
Moreover,
\begin{align*}
&\sum_{Q\in {\sf Stop}(Q)}
\Theta_\mu(2B_Q)^2\,\mu(Q)\lesssim_{A}\Theta_\mu(2B_R)^2\sum_{Q\in
{\sf Stop}(R)} \mu(Q)\lesssim_{A} \Theta_\mu(2B_R)^2 \mu(R),
\end{align*}
as  cubes in ${\sf Stop}(R)$ are disjoint subsets of $R$.

What is more, arguments similar to those in Lemmas~\ref{paper3-lemma_beta_perm_initial}  and \ref{paper3-lemma_R_far_1} imply that
\begin{align*}
&\sum_{Q\in {\sf Tree}(R)\setminus{\sf Stop}(R)}
\beta_{\mu,2}(2B_Q)^2\,\Theta_\mu(2B_Q)\,\mu(Q)\\
&\qquad\qquad\qquad
 \lesssim_{\gamma} \Theta_\mu(2B_R)^2\sum_{Q\in
{\sf Tree}(R)\setminus{\sf Stop}(R)}
 \frac{c_{[\delta,Q]}^2(\mu\lfloor
2B_Q)}{\Theta_\mu(2B_R)^2} \\
&\qquad\qquad\qquad \lesssim_{\alpha,\gamma}
\Theta_\mu(2B_R)^2\mu(R).
\end{align*}

Thus, $S_R\lesssim_{\gamma,\alpha,A} \Theta_\mu(2B_R)^2\mu(R)$,
where $\gamma$, $\alpha$ and $A$ depend on other parameters and thresholds and are suitably chosen and fixed at the end.

\bigskip

The arguments above  also provide a new proof of the bi-Lipschitz invariance of the $L^2$-boundedness of the Cauchy transform, first proved in \cite{Tolsa-bilip}:\vspace{1mm}

{\it Let $\mu$ is a finite measure in the plane with linear growth and let $\wt{\mu}:=\varphi_{\#} \mu$ be the image measure of $\mu$ under a bi-Lipschitz map $\varphi$. If the Cauchy transform $C_\mu$ is $L^2(\mu)$-bounded, then the Cauchy transform $C_{\wt\mu}$ is also $L^2(\tilde{\mu})$-bounded.}

\vspace{1mm}

By an easy application of the $T1$ theorem, to prove this statement it suffices to show that for any finite measure $\mu$ with linear growth,
\begin{equation}\label{eqfin342}
c^2(\wt\mu) \lesssim
c^2(\mu)+\mu(\mathbb{C}),
\end{equation}
with the implicit constant depending only on the linear growth of $\mu$.

In Lemma \ref{paper3-lemkey2} we have shown that the measure $\mu$ has a corona decomposition
with a suitable packing condition. From this corona decomposition one can obtain an analogous one for $\wt\mu$. Indeed, consider
the lattice of the cubes 
$$\mathcal D' = \{\varphi(Q):Q\in\mathcal D\},$$
and set 
$${\sf Top}'= \{\varphi(Q):Q\in{\sf Top}\}.$$
The corona decomposition for $\wt\mu$  in terms of the family ${\sf Top}'$ satisfies the packing condition
$$\sum_{R'\in{\sf Top'}}\Theta_{\wt\mu}(2B_{R'})^2\,\wt\mu(R')\lesssim \sum_{R\in{\sf Top}}\Theta_\mu(2B_R)^2\,\mu(R)\lesssim
c^2(\mu)+\mu(\mathbb{C}).$$
Then arguing as in \cite[Main Lemma 8.1]{Tolsa-bilip}, one derives
$$c^2(\wt\mu) \lesssim
\sum_{R'\in{\sf Top'}}\Theta_{\wt\mu}(2B_{R'})^2\,\wt\mu(R') + \wt\mu(\mathbb{C})\lesssim
c^2(\mu)+\mu(\mathbb{C}),$$
which yields \eqref{eqfin342}.

\section{Acknowledgements}

The authors are very grateful to Tuomas Orponen for valuable comments on the first version of the paper.

P.C. and X.T. were supported by the ERC grant 320501 of the European Research Council (FP7/2007-2013). X.T. was also partially supported by MTM-2016-77635-P.
J.M. was supported by MTM2013-44699 (MINECO) and
MTM2016-75390 (MINECO). J.M. and X.T. were also supported 
by MDM-2014-044 (MICINN, Spain), Marie Curie ITN MAnET (FP7-607647) and 2017-SGR-395 (Generalitat de Catalunya).

\end{document}